\documentclass[a4paper]{article}
\usepackage{amsfonts}
\usepackage{amsmath}
\usepackage{mathrsfs}
\usepackage{graphicx}
\usepackage{amsmath,amsthm,amssymb,amscd}
\usepackage{color}
\usepackage{enumerate}
\usepackage{tikz}
\usepackage[colorlinks,linkcolor=black,anchorcolor=black,citecolor=black,hyperindex=true,CJKbookmarks=true]{hyperref}

\theoremstyle{plain}
\newtheorem{maintheorem}{Theorem}

\newtheorem{maincorollary}{Corollary}

\newtheorem{mainquestion}{Question}

\newtheorem{Thm}{Theorem}[section]
\newtheorem{Lem}[Thm]{Lemma}
\newtheorem{Prop}[Thm]{Proposition}
\newtheorem{Cor}[Thm]{Corollary}
\theoremstyle{remark}
\newtheorem{Def}[Thm] {Definition}
\newtheorem{Rem}[Thm] {Remark}

\newtheorem{Ex}[Thm] {Example}
\newtheorem{Que}[Thm] {Question}

\theoremstyle{remark}
\theoremstyle{definition}

\DeclareMathOperator{\cov}{cov}

\newcommand{\eps}{\varepsilon}

\newcommand{\N}{\mathbb{N}}
\newcommand{\Z}{\mathbb{Z}}
\newcommand{\set}[1]{\left\{#1\right\}}

\newcommand{\htop}{h_{\text{top}}}

\newcommand{\E}{\mathcal{E}}

\newcommand{\B}{\mathcal{B}}
\newcommand{\C}{\mathfrak{C}}

\begin{document}

\title{Abundance of Smale's horseshoes and ergodic measures via multifractal analysis and various quantitative spectrums}



\author{{Yiwei Dong$^{*},$ Xiaobo Hou$^\S$ and Xueting Tian$^{\dag}\,^\ddag$}\\
{\em\small School of Mathematical Science,  Fudan University}\\
{\em\small Shanghai 200433, People's Republic of China}\\
{\small $^*$Email:dongyiwei06@gmail.com;  $^\S$Email:xiaobohou@fudan.edu.cn;   $^{\dag}$Email:xuetingtian@fudan.edu.cn}\\}
\date{}
\maketitle

\renewcommand{\baselinestretch}{1.2}
\large\normalsize
\footnotetext {$^\ddag$ Tian is the corresponding author.}
\footnotetext { Key words and phrases: Metric entropy; Dimension; Uniformly hyperbolic systems; Partially hyperbolic systems; Symbolic dynamics; Shadowing property.  }
\footnotetext {AMS Review:    37A35; 37C45; 37D20; 37D30; 37B10; 37C50.   }

\begin{abstract}
In this article, we combine the perspectives of density, entropy, and multifractal analysis to investigate the structure of ergodic measures.  We prove that for each transitive topologically Anosov system  $(X,f)$, each continuous function $\varphi$ on $X$  and each $(a,h)\in \mathrm{Int}\{(\int \varphi d\mu, h_\mu(f)):\mu\in \mathcal{M}_f(X)\},$  the set $\{\mu\in \mathcal{M}_f^e(X): (\int \varphi d\mu, h_\mu(f))=(a,h)\}$ is non-empty and contains a dense $G_\delta$ subset of $\{\mu\in \mathcal{M}_f(X): (\int \varphi d\mu, h_\mu(f))=(a,h)\}.$ Meanwhile, combining the development of non-hyperbolic systems and cocycles we give a general framework and use it to obtain intermediate entropy property of ergodic measures with same Lyapunov exponent for non-hyperbolic step skew-products, elliptic $\operatorname{SL}(2, \mathbb{R})$ cocycles and robustly non-hyperbolic transitive diffeomorphisms. Moreover,  we get  generalized results on multiple functions and use them to obtain the intermediate Hausdorff dimension of ergodic measures for transitive average conformal or quasi-conformal  Anosov diffeomorphisms, that is $\left\{\operatorname{dim}_H \mu: \mu\in\mathcal{M}_f^e(M)\right\}= \left\{\operatorname{dim}_H \mu: \mu\in\mathcal{M}_f(M)\right\}.$  In this process, we introduce and establish a 'multi-horseshoe' entropy-dense property and use it to get the goal combined with the well-known  conditional variational principles.  As applications, we also obtain many new observations on various other quantitative spectrums including Lyapunov exponents, first return rate,   geometric pressure, unstable Hausdorff dimension, etc. 
\end{abstract}

\section{Introduction}
Throughout the paper, by a dynamical system $(X,f)$ we mean that $(X,d)$ is a compact metric space and $f:X \rightarrow X$ is a continuous map.
Denote by $\mathcal{M}(X)$, $\mathcal{M}_f(X)$, $\mathcal{M}_{f}^{e}(X)$ the set of probability measures, $f$-invariant probability measures, $f$-ergodic probability measures, respectively.  The  set of continuous functions on $X$ is denoted by by $C(X)$.
In the 19th century, the work of Boltzmann and Gibbs on statistical mechanics raised a mathematical problem which can be stated as follows: given a  dynamical system $(X,f)$, an $f$-invariant measure $\mu$ and  an integrable function $\varphi:X\to\mathbb{R}$, find conditions under which the limit
\begin{equation}\label{equa-A}
	\lim\limits_{n\to \infty}\frac{1}{n}\sum_{i=0}^{n-1}\varphi(f^i(x))
\end{equation}
exists and is constant almost everywhere.  In 1931 Birkhoff \cite{Birkhoff1931} proved that the limit (\ref{equa-A}) exists  almost everywhere. From this result, he showed that a necessary and sufficient condition for its value to be constant almost everywhere (the constant is $\int \varphi d\mu$) is that there exists no Borel set $A$ such that $0<\mu(A)<1$ and $f^{-1}(A)=A.$ Measures that satisfy this condition are called $f$-ergodic measures.   Denote  $G_\mu=\{x\in X:\lim\limits_{n\to \infty}\frac{1}{n}\sum_{i=0}^{n-1}\varphi(f^i(x))=\int \varphi d\mu \text{ for any }\varphi\in C(X)\}.$  Birkhoff's ergodic theorem implies that  $\mu(G_\mu)=1$ for every $\mu\in\mathcal{M}_f^e(X),$ and while for every $\nu\in \mathcal{M}_f(X)\setminus \mathcal{M}_f^e(X)$ one has $\nu(G_\nu)=0$ (see \cite[Proposition 5.10]{DGS}). Moreover, in 1973 Bowen \cite{Bowen1973} proved the remarkable result that $h_{top}(f,G_\mu)=h_{\mu}(f)$ for every $\mu\in\mathcal{M}_f^e(X).$ Here $h_{top}(f,G_\mu)$ is the topological entropy of $G_\mu$ and $h_\mu(f)$ is the metric entropy of $\mu$.
For non-ergodic measures the situation is quite different,
it is not difficult to provide non-ergodic measure with $G_\mu=\emptyset$ and $h_{\mu}(f)>0.$
These results all demonstrate the crucial role of ergodic measures in ergodic theory and dynamical systems. Over the last decades, the properties of ergodic measures have been investigated extensively by researchers.   In particular, people are interested in the following question:
$$\emph{How abundant are ergodic measures in the set of invariant measures for a given dynamical system?}$$
People 
have investigated the richness of ergodic measures in the set of invariant measures from various perspectives, including density, entropy, multifractal analysis, and others. In the following, we will review some of the results  obtained in this field of study.

From the perspective of {\bf density},  researchers mainly focus on the condition under what one has $\overline{\mathcal{M}_f^e(X)}=\mathcal{M}_f(X).$ $\mathcal{M}_f^e(X)$ naturally forms a nonempty $G_\delta$ subset of $\mathcal{M}_f(X)$ (see \cite[Proposition 5.7]{DGM2019}). Thus $\mathcal{M}_f^e(X)$ is a residual subset of $\mathcal{M}_f(X)$ when $\overline{\mathcal{M}_f^e(X)}=\mathcal{M}_f(X).$    Researchers developed two methods to prove  the density of ergodic measures in chaotic dynamics, 
\begin{itemize}
	\item[$\bullet$] \emph{Constructing periodic measure.} Since periodic measures are ergodic, it's enough to obtain the density of ergodic measures by  proving that every invariant measure is the weak limit of  periodic measures. For  full shifts, Ville \cite{Ville}, Parthasarathy \cite{Parthasarathy1961} and Oxtoby \cite{Oxtoby1963} proved that periodic measures are dense in the set of invariant measures.  Building on their methods, as well as on the periodic speciﬁcation property developed by Bowen \cite{Bowen1971} for basic sets of Axioms A diffeomorphisms,  in 1970 Sigmund \cite{Sigmund1970}  obtained the density of periodic measures for Axioms A diffeomorphisms.  Other important classes of dynamical systems also have the periodic speciﬁcation property, for example,   mixing shifts of ﬁnite type \cite{DGS} and mixing interval maps \cite{Buzzi1997}. 
	The closeability and linkability properties introduced by Gelfert and Kwietniak in \cite{GK2018} are more general than the periodic specification property and apply to a wide range of dynamical systems, including  $S$-gap shifts and certain geodesic ﬂows of a complete connected negatively curved manifold.  The authors showed that they imply the density of ergodic measures. In \cite{ABC2011} Abdenur, Bonatti, and Crovisier introduced the barycenter property and used it to obtain the density of periodic measures for isolated non-trivial transitive set of a $C^1$-generic diﬀeomorphism.
	\item[$\bullet$] \emph{Constructing invariant sets.}  For dynamical systems with the speciﬁcation property but without the periodic speciﬁcation property, it's difficult to find periodic measures. A new method was developed in \cite{Dateyama1981} by Dateyama for such systems. He proved the density of ergodic measures for dynamical systems with the speciﬁcation property by constructing an invariant set on which every invariant measure is close to the given invariant measure. Eizenberg, Kifer, and Weiss \cite{EKW}  obtained a stronger result for  dynamical systems with the specification property. They construct an invariant set such that not only every invariant measure on it is close to the given invariant measure, but also the topological entropy of the set is close to the metric entropy of the given invariant measure. 
	In 2005, Pﬁster and Sullivan \cite{PS2005} generalized  this result to dynamical systems with the approximate product property, and referred to this result as the entropy-dense property.   Many dynamical systems have been shown to possess the approximate product property, including $\beta$-shifts \cite{PS2005}, transitive and noninvertible graph maps \cite{KLO2016}, and transitive soﬁc shifts \cite{KLO2016}. For a broad class of symbolic systems, i.e., the subshifts with non-uniform structure, Climenhaga, Thompson, and Yamamoto \cite{CTY2017} derived a ‘horseshoe’ theorem which implies the entropy-dense property.  This has applications to  $\beta$-shifts, $S$-gap shifts, and their factors.
\end{itemize}

\begin{figure}\caption{Entropy and multifractal analysis}
	\begin{center}
		\tikzset{every picture/.style={line width=0.75pt}} 
		\begin{tikzpicture}[x=0.75pt,y=0.75pt,yscale=-1,xscale=1]
			
			\draw    (137.94,216.69) .. controls (177.16,180.72) and (189.02,146.9) .. (185.37,134.49) ;
			\draw    (113,110) .. controls (135,83) and (178.99,107.94) .. (185.37,134.49) ;
			\draw    (80,203) .. controls (73,167) and (89.25,138.63) .. (113,110) ;
			\draw    (210,122) .. controls (249.8,92.15) and (297.52,46.95) .. (392.56,118.41) ;
			\draw [shift={(394,119.5)}, rotate = 217.25] [color={rgb, 255:red, 0; green, 0; blue, 0 }  ][line width=0.75]    (10.93,-3.29) .. controls (6.95,-1.4) and (3.31,-0.3) .. (0,0) .. controls (3.31,0.3) and (6.95,1.4) .. (10.93,3.29)   ;
			\draw    (212,221) .. controls (241.7,263.08) and (350.79,250.75) .. (391.78,221.39) ;
			\draw [shift={(393,220.5)}, rotate = 143.13] [color={rgb, 255:red, 0; green, 0; blue, 0 }  ][line width=0.75]    (10.93,-3.29) .. controls (6.95,-1.4) and (3.31,-0.3) .. (0,0) .. controls (3.31,0.3) and (6.95,1.4) .. (10.93,3.29)   ;
			\draw    (459,169.5) -- (459,20) ;
			\draw [shift={(459,18)}, rotate = 90] [color={rgb, 255:red, 0; green, 0; blue, 0 }  ][line width=0.75]    (10.93,-3.29) .. controls (6.95,-1.4) and (3.31,-0.3) .. (0,0) .. controls (3.31,0.3) and (6.95,1.4) .. (10.93,3.29)   ;
			\draw    (427,221) -- (613,220.51) ;
			\draw [shift={(615,220.5)}, rotate = 179.85] [color={rgb, 255:red, 0; green, 0; blue, 0 }  ][line width=0.75]    (10.93,-3.29) .. controls (6.95,-1.4) and (3.31,-0.3) .. (0,0) .. controls (3.31,0.3) and (6.95,1.4) .. (10.93,3.29)   ;
			\draw [color={rgb, 255:red, 245; green, 166; blue, 35 }  ,draw opacity=1 ][line width=3] [line join = round][line cap = round]   (112.4,164.03) .. controls (112.4,163.74) and (112.4,163.46) .. (112.4,163.17) ;
			\draw [color={rgb, 255:red, 245; green, 166; blue, 35 }  ,draw opacity=1 ][line width=3] [line join = round][line cap = round]    ;
			\draw [color={rgb, 255:red, 245; green, 166; blue, 35 }  ,draw opacity=1 ][line width=3] [line join = round][line cap = round]   (459,93.75) .. controls (459,93.42) and (459,93.08) .. (459,92.75) ;
			\draw [color={rgb, 255:red, 245; green, 166; blue, 35 }  ,draw opacity=1 ][line width=3] [line join = round][line cap = round]   (500,221) .. controls (500,220.67) and (501,220) .. (500,220) ;
			\draw    (461,214) -- (461,227) ;
			\draw    (459,34) -- (467,34) ;
			\draw    (459,169.5) -- (467,169.5) ;
			\draw    (80,203) .. controls (83,218) and (97.94,246.69) .. (137.94,216.69) ;
		
			\draw (120.08,147.86) node [anchor=north west][inner sep=0.75pt]    {$\mu $};
			\draw (478,226.4) node [anchor=north west][inner sep=0.75pt]    {$\int \varphi d\mu $};
			\draw (469,85.4) node [anchor=north west][inner sep=0.75pt]    {$h_{\mu }( f)$};
			\draw (280,56.4) node [anchor=north west][inner sep=0.75pt]    {$\mathcal{E}_f$};
			\draw (281,257.4) node [anchor=north west][inner sep=0.75pt]    {$\mathcal{P}_\varphi$};
			\draw (75.15,252.91) node [anchor=north west][inner sep=0.75pt]    {$\mathcal{M}_f(X)$};
			\draw (472,164.4) node [anchor=north west][inner sep=0.75pt]    {$0$};
			\draw (473,26.4) node [anchor=north west][inner sep=0.75pt]    {$h_{top}( f)$};
			\draw (457,232.4) node [anchor=north west][inner sep=0.75pt]    {$0$};
		\end{tikzpicture}
	\end{center}
\end{figure}
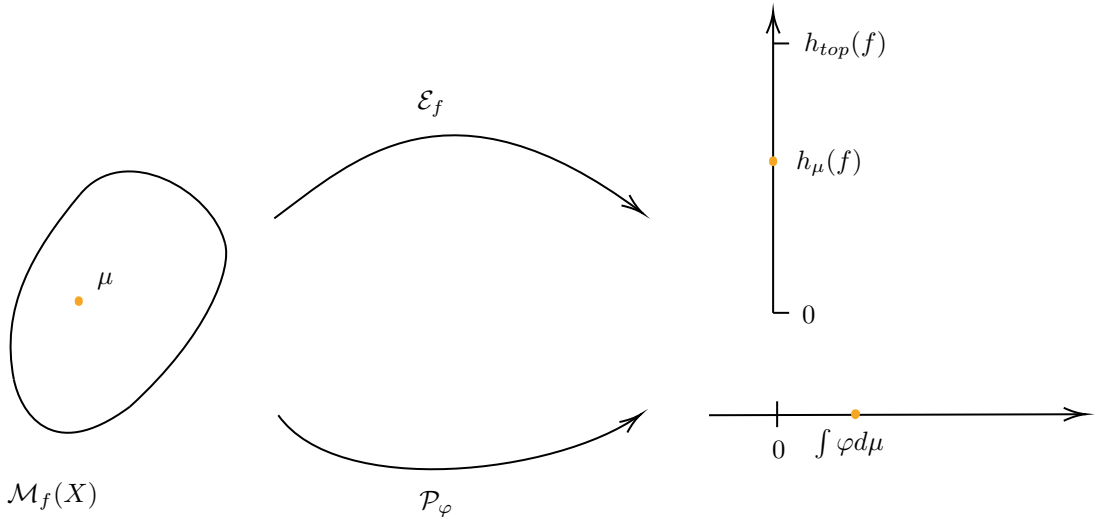

From the perspective of {\bf entropy},  researchers focus on the size of  $\mathcal{E}_f(\mathcal{M}_f^e(X))$,  where $\mathcal{E}_f:\mu\to h_{\mu}(f)$ is the entropy function of $(X,f).$ The following problems have attracted people's attention: is $\mathcal{E}_f(\mathcal{M}_f^e(X))$ dense in $\mathcal{E}_f(\mathcal{M}_f(X))$? is $\mathcal{E}_f(\mathcal{M}_f^e(X))$ equal to $\mathcal{E}_f(\mathcal{M}_f(X))$? Can  every invertible non-atomic ergodic measure-preserving system $(\Omega, \mu , T)$ with measure-theoretic entropy strictly less than the topological entropy of $(X,f)$ be  embedded  into $(X, f)$?  Now we review the progress of these problems.
\begin{itemize}
	\item[$\bullet$] \emph{Dense intermediate entropies of ergodic measures.} The classical variational principle shows that $\htop(f)=\sup\mathcal{E}_f(\mathcal{M}_f(X))=\sup\mathcal{E}_f(\mathcal{M}_f^e(X)),$ where $\htop(f)$ is the topological entropy of $(X,f)$.  Using the entropy-dense property, Eizenberg, Kifer, and Weiss \cite[Theorem B]{EKW}  showed that if $(X,f)$ has the specification property and its entropy function is upper semi-continuous, then any invariant measure $\mu$ is the weak limit of a sequence of ergodic measures
	$\{\mu_n\}_{n=1}^{\infty}$,   such that the entropy of $\mu$ is the limit of the entropies of the $\{\mu_n\}_{n=1}^{\infty}$. This implies dense intermediate entropies of ergodic measures, i.e.,  $\overline{\mathcal{E}_f(\mathcal{M}_f^e(X))}=\mathcal{E}_f(\mathcal{M}_f(X)).$ Pﬁster and Sullivan \cite{PS2005} generalized this result to dynamical systems with the approximate product property and upper semi-continuous entropy function. In \cite{Sun2012}, the density of intermediate entropies of ergodic measures is verified for linear toral automorphisms. Recently,  Sun \cite{Sun2020} shows that systems with the Climenhaga-Thompson structure have dense intermediate entropies of ergodic measures, and applies it to the Mañé diﬀeomorphisms.
	\item[$\bullet$] \emph{Intermediate entropies of ergodic measures.} 
	In \cite{Katok} Katok proved a milestone result that every $C^{1+\alpha}$ diffeomorphism $f$ in dimension $2$ has horseshoes of large entropies. Thus such systems  have ergodic measures of arbitrary intermediate metric entropies, that is $\mathcal{E}_f(\mathcal{M}_f^e(X))$ includes $[0, \htop(f)).$ This implies $\mathcal{E}_f(\mathcal{M}_f^e(X))=\mathcal{E}_f(\mathcal{M}_f(X)).$     Katok believed that this result holds for systems of any dimension. It is called Katok's conjecture or intermediate entropy problem. In the last decade, Katok's conjecture has been verified in many systems, including certain skew product systems \cite{Sun201001,Sun201002},
	some partially hyperbolic diﬀeomorphisms with one-dimensional center bundles \cite{Ures2012,YZ2020,DGM2017}, certain homogeneous dynamics \cite{GSW2017}, hereditary shifts \cite{KKK2018}, transitive systems with shadowing property \cite{LiOpro2018}, systems with approximate product property and asymptotic entropy expansiveness \cite{Sun2019}, star ﬂows \cite{LSWW2020}, and affine transformations of nilmanifolds \cite{HXX}.
	\item[$\bullet$] \emph{Universality.}  A  dynamical system $(X, f)$ is said to be universal if for every invertible non-atomic ergodic measure-preserving system $(\Omega, \mu , T)$ with measure-theoretic entropy strictly less than the topological entropy of $(X,f)$ there exists an embedding of $(\Omega, \mu , T)$ into $(X, f)$.  Obviously, for a universal dynamical system, one has $\mathcal{E}_f(\mathcal{M}_f^e(X))=\mathcal{E}_f(\mathcal{M}_f(X)).$ The Krieger ﬁnite generator theorem \cite{Krieger1970,Krieger1972} says that the full shift is universal.  Quas and Soo \cite{QuasSoo2016} extended Krieger’s theorem to dynamical systems that satisfy almost weak speciﬁcation, asymptotic entropy expansiveness, and the small boundary property.
	Burguet \cite{Burguet2020} improved the result of Quas and Soo to request only the almost weak speciﬁcation property. Recently, Chandgotia and Meyerovitch \cite{CM2021} define a combinatorial condition (`flexible marker sequence') that is a sufficient condition for a topological $\mathbb{Z}^d$ dynamical system to be universal. This condition is a suitable form of specification property. Specifically, it means the result applies to generic homeomorphisms of two-dimensional compact manifolds and certain systems defined by tiling and coloring conditions. 
\end{itemize}

From the perspective of {\bf multifractal analysis}, researchers focus on the relation between $\mathcal{P}_\varphi(\mathcal{M}_f^e(X))$ and $\mathcal{P}_\varphi(\mathcal{M}_f(X))$ for a given function $\varphi.$ Here
$\mathcal{P}_\varphi(\mu)=\int \varphi d\mu$ is the integral of $\mu$ with respect to $\varphi.$ According to type of $\varphi,$ people study Birkhoff spectrum  and Lyapunov spectrum. 
\begin{itemize}
	\item[$\bullet$] \emph{Birkhoff spectrum.} From the definition of weak* topology on the space of probability measures, if $\mathcal{M}_f^e(X)$ is dense in $\mathcal{M}_f(X)$, then $\overline{\mathcal{P}_\varphi(\mathcal{M}_f^e(X))}=\mathcal{P}_\varphi(\mathcal{M}_f(X))$ for any $\varphi\in C(X).$ In the research of entropy spectrum of the Birkhoff averages, that is, the topological entropy of level sets of points with a common given average,
	 Barreira and Saussol \cite{BarreiraSaussol2001} showed that if $\mathcal{E}_f$ is upper semi-continuous, and $b\varphi+c$ has  a unique equilibrium measure for any real numbers $b$ and $c,$ then for any $a\in \mathrm{Int}(\mathcal{P}_\varphi(\mathcal{M}_f(X))),$ $\mathcal{M}_f^e(X)\cap \mathcal{P}_\varphi^{-1}(a)\neq\emptyset.$ This means $\mathrm{Int}(\mathcal{P}_\varphi(\mathcal{M}_f^e(X)))=\mathrm{Int}(\mathcal{P}_\varphi(\mathcal{M}_f(X))).$ Combining the ergodic decomposition theorem,  one has $\mathcal{P}_\varphi(\mathcal{M}_f^e(X))=\mathcal{P}_\varphi(\mathcal{M}_f(X)).$ Tian, Wang, and Wang \cite{TWW2019} generalized it to dynamical systems with the periodic gluing orbit property and any continuous functions.  
	\item[$\bullet$] \emph{Lyapunov spectrum.} 
	In  the  research of entropy spectrum of the Lyapunov exponents, that is, the topological entropy of level sets of points with a common given exponent, of some type of skew product with circle fibers, Díaz, Gelfert, and Rams \cite{DGM2019,DGM2022} showed that the set $\{\chi(\mu):\mu \in \mathcal{M}_f^e(X)\}$ contains a closed interval, where $\chi(\mu)$ is the Lyapunov exponent of $\mu.$  B. Bárány,  T.Jordan,  A. Käenmäki, and M. Rams \cite{BJKR2021} calculated the Lyapunov spectrum of strongly irreducible planar self-affine sets satisfying the strong open set condition and obtained similar results on Lyapunov exponents of ergodic measures.
\end{itemize}

Some aforementioned results show the richness of ergodic measures from multiple perspectives simultaneously. On one hand, the entropy-dense property shows that the proportion of ergodic measures is signiﬁcant from the viewpoints of  topology and entropy simultaneously. On the other hand, the conditional variational principle obtained in \cite{BarreiraSaussol2001} gives a partial description of ergodic measures from the viewpoints of multifractal analysis and entropy simultaneously. Given  $n\in \mathbb{N}$ and $C\subset \mathbb{R}^n,$ denote the interior of $C$ by $$\mathrm{Int}(C)=\{x\in C: \text{there is an open subset } B \text{ of } \mathbb{R}^n\text{ such that }x\in B_x\subset C \}.$$
Barreira and Saussol \cite{BarreiraSaussol2001} gave the following result, called conditional variational principle.
\begin{Thm}\cite[Theorem 4 and Lemma 4]{BarreiraSaussol2001}
	Suppose $(X,  f)$ is a dynamical system whose entropy function $\mathcal{E}_f$ is upper semi-continuous. Given $\varphi\in C(X).$ If $b_1\varphi+b_2$ has  a unique equilibrium measure for any $b_1,b_2\in\mathbb{R},$  then for any $a\in \mathrm{Int}(\mathcal{P}_\varphi(\mathcal{M}_f(X))),$ there exists $\mu_a\in \mathcal{M}_f^e(X)\cap \mathcal{P}_\varphi^{-1}(a)$  such that $\mathcal{E}_f(\mu_a)=\sup\mathcal{E}_f(\mathcal{P}_\varphi^{-1}(a)).$ Here  $\sup\mathcal{E}_f(\mathcal{P}_\varphi^{-1}(a))=\sup_{\mu\in \mathcal{P}_\varphi^{-1}(a)} \mathcal{E}_f(\mu).$ 
\end{Thm}

Combining $\mathcal{E}_f$ and $\mathcal{P}_\varphi$, we define a map on $\mathcal{M}_f(X)$ as following: $$\mathcal{T}_{\varphi,f}:\mu\to (\mathcal{P}_\varphi(\mu),\mathcal{E}_f(\mu))=(\int \varphi d\mu, h_\mu(f)).$$ We draw the graph of $\mathcal{T}_{\varphi,f}.$ Then the conditional variational principle
implies that every point in the green line of Figure \ref{fig-1} can be attained by ergodic measures. 
\begin{figure}\caption{Graph of $(\int \varphi d\mu, h_\mu(f))$}\label{fig-1}
	\begin{center}
		\tikzset{every picture/.style={line width=0.75pt}} 
		\begin{tikzpicture}[x=0.75pt,y=0.75pt,yscale=-1,xscale=1]
			
			\draw    (105.94,208.69) .. controls (145.16,172.72) and (157.02,138.9) .. (153.37,126.49) ;
			\draw    (81,102) .. controls (103,75) and (146.99,99.94) .. (153.37,126.49) ;
			\draw    (48,195) .. controls (41,159) and (57.25,130.63) .. (81,102) ;
			\draw    (170,160) .. controls (216.77,126.17) and (236.8,125.01) .. (322.7,163.42) ;
			\draw [shift={(324,164)}, rotate = 204.15] [color={rgb, 255:red, 0; green, 0; blue, 0 }  ][line width=0.75]    (10.93,-3.29) .. controls (6.95,-1.4) and (3.31,-0.3) .. (0,0) .. controls (3.31,0.3) and (6.95,1.4) .. (10.93,3.29)   ;
			\draw [color={rgb, 255:red, 245; green, 166; blue, 35 }  ,draw opacity=1 ][line width=3] [line join = round][line cap = round]   (80.4,156.03) .. controls (80.4,155.74) and (80.4,155.46) .. (80.4,155.17) ;
			\draw [color={rgb, 255:red, 245; green, 166; blue, 35 }  ,draw opacity=1 ][line width=3] [line join = round][line cap = round]    ;
			\draw    (48,195) .. controls (51,210) and (65.94,238.69) .. (105.94,208.69) ;
			\draw [color={rgb, 255:red, 126; green, 211; blue, 33 }  ,draw opacity=1 ]   (411,86) .. controls (496,41) and (570,98) .. (591,115) ;
			\draw [color={rgb, 255:red, 208; green, 2; blue, 27 }  ,draw opacity=1 ]   (591,241) -- (591,115) ;
			\draw [color={rgb, 255:red, 245; green, 166; blue, 35 }  ,draw opacity=1 ][line width=3] [line join = round][line cap = round]   (460.4,180.03) .. controls (460.4,179.74) and (460.4,179.46) .. (460.4,179.17) ;
			\draw  (351,240.24) -- (615,240.24)(379.51,18) -- (379.51,269) (608,235.24) -- (615,240.24) -- (608,245.24) (374.51,25) -- (379.51,18) -- (384.51,25)  ;
			\draw [color={rgb, 255:red, 208; green, 2; blue, 27 }  ,draw opacity=1 ]   (411,86) -- (411,241) ;
			\draw  [dash pattern={on 0.84pt off 2.51pt}]  (460.4,180.03) -- (460.4,240.03) ;
			\draw  [dash pattern={on 0.84pt off 2.51pt}]  (460.4,180.03) -- (379.4,180.03) ;
			
			\draw (88.08,138.86) node [anchor=north west][inner sep=0.75pt]    {$\mu $};
			\draw (224,105.4) node [anchor=north west][inner sep=0.75pt]    {$\mathcal{T}_{\varphi,f}$};
			\draw (43.15,244.91) node [anchor=north west][inner sep=0.75pt]    {$\mathcal{M}_f(X)$};
			\draw (362,242.4) node [anchor=north west][inner sep=0.75pt]    {$0$};
			\draw (458.03,144.03) node [anchor=north west][inner sep=0.75pt]  [rotate=-359.83]  {$\left(\int \varphi d\mu, h_{\mu }( f) \right)$};
			\draw (332.08,179.86) node [anchor=north west][inner sep=0.75pt]    {$h_{\mu }( f)$};
			\draw (446.08,246.86) node [anchor=north west][inner sep=0.75pt]    {$\int \varphi d\mu $};
		\end{tikzpicture}
	\end{center}
\end{figure}
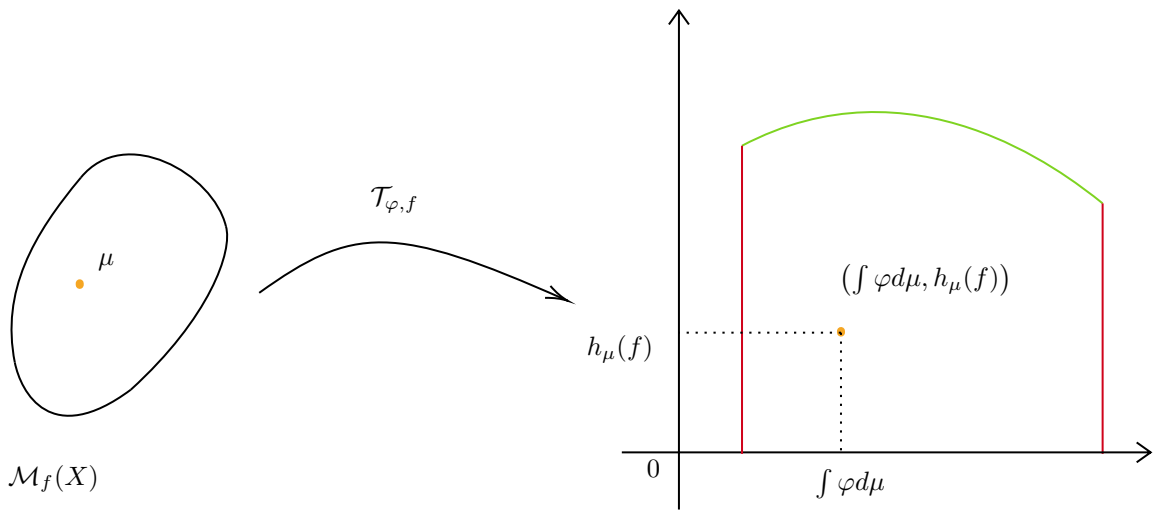

However, it is a question that whether the points outside the green line can be attained by ergodic measures, i.e., whether one has $\mathcal{T}_{\varphi,f}(\mathcal{M}_f^e(X))=\mathcal{T}_{\varphi,f}(\mathcal{M}_f(X))$?
In this article, we aim to answer this question.  And we
delve deeper into the abundance of ergodic measures by combining the perspectives of entropy, topology, and multifractal analysis. Precisely, we consider the following question.
\begin{mainquestion}\label{Conjecture-2}
	For every typical diffeomorphism $f$ on a compact Riemannian manifold $M$ and every continuous function $\varphi$ on $M,$  whether one has
	$\mathcal{T}_{\varphi,f}(\mathcal{M}_f^e(M))=\mathcal{T}_{\varphi,f}(\mathcal{M}_f(M))$?
	Moreover, for any $c\in \mathcal{T}_{\varphi,f}(\mathcal{M}_f(M)),$ is  $\mathcal{T}_{\varphi,f}^{-1}(c)\cap \mathcal{M}_f^e(M)$ residual in $\mathcal{T}_{\varphi,f}^{-1}(c)?$
\end{mainquestion}
\noindent In a Baire space, a set is {\it residual} if it contains a countable intersection of dense open sets.
Question \ref{Conjecture-2} in the context of the graph is whether every point in the closed region of Figure \ref{fig-1} can be attained by ergodic measures. If so, then for any point, does the set of all ergodic measures that attain the point form a residual set in the set of all invariant measures that attain that point?

There are three topics about Figure \ref{fig-1}:
\begin{itemize}
	\item[$\bullet$] \emph{Conditional variational principle or multifractal analysis.} On this subject, researchers are mainly concerned about whether the green line is the entropy spectrum of level sets and whether every point in the green line can be attained by ergodic measures.  We refer the reader to \cite{BarreiraSaussol2001,FLP2008,BarreiraDoutor2009,IJ2015,BH21} for more progress.
	\item[$\bullet$] \emph{Ergodic optimization.}  Ergodic optimization is the study of problems relating to maximizing/minimizing invariant measures which are the measures that attain the red lines, and under which conditions a dynamical system has a unique maximizing/minimizing  measure which is supported on a periodic orbit. We refer the reader to \cite{Cont2016,Bochi2018,HLXMZ2019,Jenkinson2019,GS2022} for more progress. When a  system has a unique maximizing/minimizing  measure, the red lines will degenerate into two dots.
	\item[$\bullet$] \emph{Intermediate entropy property of ergodic measures with same level.} There is no result about the interior of the closed region of Figure \ref{fig-1}  as far as we know. In the present paper, we  mainly consider the interior and give a partial answer to Question \ref{Conjecture-2}.
\end{itemize}

\subsection{Intermediate entropy property of ergodic measures with same level}
\subsubsection{Topologically Anosov systems}
First, we consider topologically Anosov systems.
\begin{Def}
	A homeomorphism $f:X\to X$ of a compact metric space is called  {\it topologically hyperbolic} or  {\it topologically Anosov},   if $X$ has infinitely many points,  $(X,f)$ is expansive and satisfies the shadowing property. 
\end{Def}
We denote the support of a measure $\mu$ by $S_\mu:=\{x\in X:\mu(U)>0\ \text{for any neighborhood}\ U\ \text{of}\ x\}.$ Let $\rho$ be a metric for the weak*-topology on $\mathcal{M}(X)$ (see definition in section \ref{section-space of measure}). 
Now we state our main result on topologically Anosov systems.
\begin{maintheorem}\label{thm-continuous}
	Suppose that  $(X,f)$ is transitive and topologically Anosov. Let $\varphi$ be a continuous function on $X$ with $\mathrm{Int}(\mathcal{P}_\varphi(\mathcal{M}_f(X)))\neq\emptyset$. 
    Then:
    \begin{description}
    	\item[(I)] For any $a\in \mathrm{Int}(\mathcal{P}_\varphi(\mathcal{M}_f(X))),$ any $\mu\in \mathcal{P}_\varphi^{-1}(a)$ and any $\eta,  \zeta>0$, there is $\nu\in \mathcal{P}_\varphi^{-1}(a)\cap \mathcal{M}_f^e(X)$ such that $\rho(\nu,\mu)<\zeta$ and $|h_{\nu}(f)-h_{\mu}(f)|<\eta.$ 
    	\item[(II)] For any $a\in \mathrm{Int}(\mathcal{P}_\varphi(\mathcal{M}_f(X))),$ any $\mu\in \mathcal{P}_\varphi^{-1}(a),$ any $0\leq h\leq h_{\mu}(f)$ and any $\eta,  \zeta>0$, there is $\nu\in \mathcal{P}_\varphi^{-1}(a)\cap \mathcal{M}_f^e(X)$ such that $\rho(\nu,\mu)<\zeta$ and $|h_{\nu}(f)-h|<\eta.$ 
    	\item[(III)] For  any $a\in \mathrm{Int}(\mathcal{P}_\varphi(\mathcal{M}_f(X)))$ and $0\leq h< \sup\mathcal{E}_f(\mathcal{P}_\varphi^{-1}(a)),$ the set $\mathcal{T}_{\varphi,f}^{-1}(a,h)\cap \mathcal{M}_f^e(X)\cap \{\mu:S_\mu=X\}$ is residual in $\mathcal{T}_{\varphi,f}^{-1}(\{a\}\times[h,+\infty)).$
    	\item[(IV)] $\mathrm{Int}(\mathcal{T}_{\varphi,f}(\mathcal{M}_f(X)))=\mathrm{Int}(\mathcal{T}_{\varphi,f}(\mathcal{M}_f^e(X))).$
    	\item[(V)]  If further $b_1\varphi+b_2$ has  a unique equilibrium measure for any $b_1,b_2\in\mathbb{R},$ then $\{\mathcal{T}_{\varphi,f}(\mu):\mu\in \mathcal{M}_f(X),\ \mathcal{P}_\varphi(\mu) \in\mathrm{Int}(\mathcal{P}_\varphi(\mathcal{M}_f(X)))\}
    	=\{\mathcal{T}_{\varphi,f}(\mu):\mu\in \mathcal{M}_f^e(X),\ \mathcal{P}_\varphi(\mu) \in \mathrm{Int}(\mathcal{P}_\varphi(\mathcal{M}_f(X)))\}.$
    \end{description}
\end{maintheorem}
\begin{Rem}
	The set of continuos function $\varphi$ satisfying $\mathrm{Int}(\mathcal{P}_\varphi(\mathcal{M}_f(X)))\neq\emptyset$ is open and dense in $C(X),$ see Proposition \ref{proposition-AD}.
\end{Rem}

\begin{Rem}
	As the readers can see, Item (I) can be obtained directly from Item (II) of Theorem \ref{thm-continuous}. The reason why we list Item (I) here is that Item (I) is the most important step in the proof of Theorem \ref{thm-continuous} (see Lemma \ref{lemma-F}). 
\end{Rem}
\begin{Rem}\label{remark-AA}
	Now, we compare Theorem \ref{thm-continuous} with some known results.
	\begin{itemize}
		\item[$\bullet$] \emph{Entropy-dense property.} In \cite{EKW} Eizenberg,   Kifer and Weiss proved for systems with the specification property that  for every invariant measure $\mu$, there exists a sequence of ergodic measures $\{\mu_{n}\}_{n=1}^{\infty}$ such that $\lim\limits_{n \rightarrow \infty} \mu_{n}=\mu$ and $\lim\limits_{n \rightarrow \infty} h_{\mu_{n}}(f)\geq h_{\mu}(f)$. Pfister and Sullivan referred to this property as the entropy-dense property \cite{PS2005} and proved that this property holds for systems with the approximate product property. From Theorem \ref{thm-continuous}(I), if  $(X,f)$ is transitive and topologically Anosov,  then the entropy-dense property holds in $\mathcal{P}_\varphi^{-1}(a)$ for any $a\in \mathrm{Int}(\mathcal{P}_\varphi(\mathcal{M}_f(X)))$ and any $\varphi\in C(X)$.
		\item[$\bullet$] \emph{Refined entropy-dense property.} 
		Li and Oprocha proved in \cite{LiOpro2018} gave a refined entropy-dense property for transitive dynamical systems with the shadowing property, that is, for every invariant measure $\mu$ and every $0 \leq h \leq$ $h_{\mu}(f),$ there exists a sequence of ergodic measures $\{\mu_{n}\}_{n=1}^{\infty}$ such that $\lim\limits_{n \rightarrow \infty} \mu_{n}=\mu$ and $\lim\limits_{n \rightarrow \infty} h_{\mu_{n}}(f)=h$. If further, the entropy function is upper semi-continuous, they proved that for every $0 \leq h<\htop(f)$, the set of ergodic measures with entropy $h$ is residual in the space of invariant measures with entropy at least $h$. From Theorem \ref{thm-continuous}(II) and (III) we obtain more refined results for $(X,f)$ which  is transitive and topologically Anosov,  that is, for any  $\varphi \in C(X)$,  any $a\in \mathrm{Int}(\mathcal{P}_\varphi(\mathcal{M}_f(X))),$ any $\mu\in \mathcal{P}^{-1}_\varphi(a),$ and any $0 \leq h \leq$ $h_{\mu}(f),$ there exists a sequence of ergodic measures $\{\mu_{n}\}_{n=1}^{\infty}\subseteq \mathcal{P}^{-1}_\varphi(a)$ such that $\lim\limits_{n \rightarrow \infty} \mu_{n}=\mu$ and $\lim\limits_{n \rightarrow \infty} h_{\mu_{n}}(f)=c,$ and for any $a\in \mathrm{Int}(\mathcal{P}_\varphi(\mathcal{M}_f(X)))$ and $0\leq h< \sup\mathcal{E}_f(\mathcal{P}_\varphi^{-1}(a)),$ in $\mathcal{P}^{-1}_\varphi(a)$ the set of ergodic measures with entropy $h$ and full support is residual in the set of invariant measures with entropy at least $h$. 
		\item[$\bullet$] \emph{Question \ref{Conjecture-2}.} 
		From Theorem \ref{thm-continuous}(III), the set $\mathcal{T}_{\varphi,f}^{-1}(a,h)\cap \mathcal{M}_f^e(X)\cap \{\mu:S_\mu=X\}$ is residual in $\mathcal{T}_{\varphi,f}^{-1}(a,h)$.
		Thus, from Theorem A(III)(IV)(V) we give a partial answer to Question \ref{Conjecture-2} for  transitive topologically Anosov systems.
	\end{itemize}
\end{Rem}
From Lemma \ref{LH} and \ref{SFT}, every system restricted on a locally maximal hyperbolic set or a  two-sided subshift of finite type is topologically Anosov. So we have the following corollary.
\begin{maincorollary}
	Suppose that  $(X,f)$ is a system restricted on a transitive locally maximal hyperbolic set or a transitive two-sided subshift of finite type.  Then the results of Theorem \ref{thm-continuous} hold.
\end{maincorollary}
The results of Theorem \ref{thm-continuous} are also applicable to some dynamical systems beyond uniform hyperbolicity. From \cite{Gogo2010} we know that non-hyperbolic diffeomorphism $f$ with $C^{1+Lip}$ smoothness, conjugated to a transitive Anosov diffeomorphism $g$, exists and even the conjugation and its inverse are Hölder  continuous. Such $f$ is also topologically Anosov and transitive.

\subsubsection{A general framework }\label{section-nonhyperblic}
Now, we consider  dynamical systems beyond uniform hyperbolicity.
We introduce a general framework and apply it to non-hyperbolic step skew-products, elliptic $\operatorname{SL}(2, \mathbb{R})$ cocycles and robustly non-hyperbolic transitive diffeomorphisms.
Let $f$ be a $C^{1}$ diffeomorphism on a compact Riemannian manifold  $M$ and $p$ be a hyperbolic periodic point.  A  hyperbolic periodic point $q$ is said to be homoclinically related to $p$ if the stable manifold of the orbit of $q$ transversely meets the unstable one of the orbit of $p$ and vice versa. 
We say  $\mu\in \mathcal{M}_f(M)$ can be approximated by $p$-horseshoes if for any $\varepsilon>0,$ there are an  $f$-invariant compact subset $\Lambda_{\varepsilon}$ and an invariant  measure $\mu_\varepsilon\in \mathcal{M}_f(\Lambda_{\varepsilon})$ satisfying the following three properties 
\begin{description}
	\item[(1)] $\Lambda_\varepsilon$ is a transitive locally maximal hyperbolic set which contains a hyperbolic saddle $q$ homoclinically related to $p.$
	\item[(2)] $\rho(\mu,\mu_\varepsilon)<\varepsilon.$
	\item[(3)] $h_{\mu_\varepsilon}(f)>h_{\mu}(f)-\varepsilon.$
\end{description}
We denote $\mathcal{M}_{horse}(p)$ the set of invariant measures  which can be approximated by $p$-horseshoes, and  denote $\mathcal{M}_{horse}^{e}(p)=\mathcal{M}_{horse}(p)\cap \mathcal{M}_f^e(M).$
Now we state our main result on $\mathcal{M}_{horse}(p)$.
\begin{maintheorem}\label{thm-continuous-2}
	Let $f$ be a $C^{1}$ diffeomorphism on a compact Riemannian manifold  $M$ and $p$ be a hyperbolic periodic point.  Assume that $\mu\to h_\mu(f)$ is upper semi-continuous on $\mathcal{M}_{horse}(p)$. Let $\varphi$ be a continuous function on $M$ with $\mathrm{Int}(\mathcal{P}_\varphi(\mathcal{M}_{horse}(p)))\neq\emptyset$. 
	Then:
	\begin{description}
		\item[(I)] For any $a\in \mathrm{Int}(\mathcal{P}_\varphi(\mathcal{M}_{horse}(p))),$ any $\mu\in \mathcal{P}_\varphi^{-1}(a)\cap \mathcal{M}_{horse}(p)$ and any $\eta,  \zeta>0$, there is $\nu\in \mathcal{P}_\varphi^{-1}(a)\cap \mathcal{M}_{horse}^e(p)$ such that $\rho(\nu,\mu)<\zeta$ and $|h_{\nu}(f)-h_{\mu}(f)|<\eta.$ 
		\item[(II)] For any $a\in \mathrm{Int}(\mathcal{P}_\varphi(\mathcal{M}_{horse}(p))),$ any $\mu\in \mathcal{P}_\varphi^{-1}(a)\cap \mathcal{M}_{horse}(p),$ any $0\leq h\leq h_{\mu}(f)$ and any $\eta,  \zeta>0$, there is $\nu\in \mathcal{P}_\varphi^{-1}(a)\cap \mathcal{M}_{horse}^e(p)$ such that $\rho(\nu,\mu)<\zeta$ and $|h_{\nu}(f)-h|<\eta.$ 
		\item[(III)] For  any $a\in \mathrm{Int}(\mathcal{P}_\varphi(\mathcal{M}_{horse}(p)))$ and $0\leq h< \sup\mathcal{E}_f(\mathcal{P}_\varphi^{-1}(a)\cap \mathcal{M}_{horse}(p)),$ the set $\mathcal{T}_{\varphi,f}^{-1}(a,h)\cap \mathcal{M}_{horse}^e(p)$ is residual in $\mathcal{T}_{\varphi,f}^{-1}(\{a\}\times [h,+\infty))\cap \mathcal{M}_{horse}(p) .$
		\item[(IV)] $\mathrm{Int}(\mathcal{T}_{\varphi,f}(\mathcal{M}_{horse}(p)))=\mathrm{Int}(\mathcal{T}_{\varphi,f}(\mathcal{M}_{horse}^e(p))).$
	\end{description}
\end{maintheorem}

\begin{Rem}
	In \cite{ABC2011} Abdenur, Bonatti, and Crovisier introduced the barycenter property and used it to obtain the density of ergodic measures for isolated non-trivial homoclinic class of a $C^1$-generic diﬀeomorphism. However, even entropy-dense property is unknown for such homoclinic classes. So, it's difficult to answer Question \ref{Conjecture-2} in the framework of \cite{ABC2011}.
\end{Rem}
\begin{Rem}
	In Section \ref{section-thm}  we give more general results than Theorem \ref{thm-continuous} and  \ref{thm-continuous-2}, see Theorem \ref{thm-almost} and  \ref{thm-almost-2}. They generalize Theorem \ref{thm-continuous} and  \ref{thm-continuous-2} from the following aspects:
	\begin{itemize}
		\item[$\bullet$] continuous functions $\rightarrow$ asymptotically additive sequences of continuous functions;
		\item[$\bullet$]  Hölder continuous functions  $\rightarrow$  almost additive sequences of continuous functions;
		\item[$\bullet$]  single function  $\rightarrow$  multiple   functions;
		\item[$\bullet$]  entropy   $\rightarrow$  abstract pressure.
	\end{itemize}
We will use the result on multiple   functions to obtain intermediate Hausdorff  dimension of ergodic measures and Lyapunov spectrum for  transitive Anosov diffeomorphisms, see Theorem \ref{thm-inter-huasdorff}, \ref{thm-Lyapunov} and \ref{thm-first-return}.
   The concept of almost additive sequences is introduced to study  the Lyapunov exponents of nonconformal transformations \cite{BarreiraGelfert2006}, and the concept of asymptotically additive sequences  is mainly motivated by some works on the Lyapunov exponents of matrix products \cite{FH2010}, so that the results of the present paper are suitable for the cases of \cite{BarreiraGelfert2006,FH2010}.
\end{Rem}

\subsubsection{Non-hyperbolic step skew-products}
Now we consider non-hyperbolic step skew-products with circle fibers. In \cite{DGM2019,DGM2022}, Díaz, Gelfert and Rams derive a multifractal analysis for the topological entropy of the level sets of Lyapunov exponent for these systems. Consider a finite family $f_i: \mathbb{S}^1 \rightarrow \mathbb{S}^1, i=$ $0, \ldots, N-1$ for $N \geq 2$, of $C^1$ diffeomorphisms and the associated step skew-product
\begin{equation}\label{equation-step-skew}
	F: \Sigma_N \times \mathbb{S}^1 \rightarrow \Sigma_N \times \mathbb{S}^1, \quad F(\xi, x)=\left(\sigma(\xi), f_{\xi_0}(x)\right),
\end{equation}
where $\Sigma_N=\{0, \ldots, N-1\}^{\mathbb{Z}}$. We consider the class $\operatorname{SP}_{\text {shyp }}^1\left(\Sigma_N \times \mathbb{S}^1\right)$ of such maps which are topologically transitive and "nonhyperbolic in a nontrivial way". Readers can refer to  \cite{DGM2019,DGM2022} for precise definitions.
Given $X=(\xi, x) \in \Sigma_N \times \mathbb{S}^1$, consider the (fiber) Lyapunov exponent of $X$
$$
\chi(X) {=} \lim _{n \rightarrow \pm \infty} \frac{1}{n} \log \left|\left(f_{\xi}^n\right)^{\prime}(x)\right|,
$$
(where $f_{\xi}^{-n} {=} f_{\xi_{-n}} \circ \cdots \circ f_{-1}$ and $f_{\xi}^n {=} f_{\xi_{n-1}} \circ \cdots \circ f_{\xi_0}$, $\xi=(\dots\xi_{-1}\xi_0\xi_1\dots)\in \Sigma_N$ ) where we assume that both limits $n \rightarrow \pm \infty$ exist and coincide.  Given $\alpha \in \mathbb{R}$ let
$$
\mathcal{L}(\alpha) {=}\left\{X \in \Sigma_N \times \mathbb{S}^1: \chi(X)=\alpha\right\}.
$$
Given an $F$-invariant measure $\mu$, denote by $\chi(\mu)$ the Lyapunov exponent of $\mu$ defined by
$$
\chi(\mu) {=} \int \log \left|\left(f_{\xi_0}\right)^{\prime}(x)\right| d \mu(\xi, x) .
$$
Denote $a_{\min }=\inf\{a:\mathcal{L}(a) \neq \emptyset\}$ and $a_{\max }=\sup\{a:\mathcal{L}(a) \neq \emptyset\}$. By Lemma \ref{Lem-skew-min-max} we have $$a_{\min }=\min\{\chi(\mu):\mu  \in \mathcal{M}_{F}^e(\Sigma_N \times \mathbb{S}^1)\}$$ $$a_{\max }=\max\{\chi(\mu):\mu \in  \mathcal{M}_{F}^e(\Sigma_N \times \mathbb{S}^1)\}.$$ By \cite[Section 7.1]{DGM2019}, there exist $F$-ergodic measures  with positive exponent or negative exponent. It implies $a_{\min}<0<a_{\max}.$
\begin{Thm}(\cite[Theorem A]{DGM2019} and \cite[Theorem A]{DGM2022})\label{Thm-skew}
	For every $N \geq 2$ and every $F \in \operatorname{SP}_{\text {shyp }}^1\left(\Sigma_N \times \mathbb{S}^1\right)$ we have $\mathcal{L}(a) \neq \emptyset$ if and only if $a \in\left[a_{\min }, a_{\max }\right]$. Moreover,  the map $\alpha \mapsto h_{top}(\mathcal{L}(a))$ is continuous and concave on each interval $\left[a_{\min }, 0\right]$ and $\left[0, a_{\max }\right]$ and for every $a \in\left[a_{\min }, a_{\max }\right]$, one has
	$$
	h_{top}(\mathcal{L}(a))=\sup \{h_\mu(F): \mu \in \mathcal{M}_{F}^e(\Sigma_N \times \mathbb{S}^1),\chi(\mu)=a\} .
	$$
\end{Thm}
Using  Theorem \ref{thm-continuous-2}, we show that every $F \in \operatorname{SP}_{\text {shyp }}^1\left(\Sigma_N \times \mathbb{S}^1\right)$ has intermediate entropy property of ergodic measures with same Lyapunov exponent.
\begin{maintheorem}\label{maintheorem-skew}
	For every $N \geq 2$,  every $F \in \operatorname{SP}_{\text {shyp }}^1\left(\Sigma_N \times \mathbb{S}^1\right)$,   every $a \in\left(a_{\min }, 0\right) \cup\left(0, a_{\max }\right)$ and every $0\leq h<\sup \{h_\mu(F): \mu \in \mathcal{M}_{F}^e(\Sigma_N \times \mathbb{S}^1),\chi(\mu)=a\},$ there exists $\mu_{a,h}\in  \mathcal{M}_{F}^e(\Sigma_N \times \mathbb{S}^1)$ such that $\chi(\mu_{a,h})=a$ and $h_{\mu_{a,h}}(F)=h,$  that is,
	$$
	[0,H(f,\chi,a))\subset  \{h_\mu(F): \mu \in \mathcal{M}_{F}^e(\Sigma_N \times \mathbb{S}^1), \chi(\mu)=a\} \subset [0,H(f,\chi,a)],
	$$
	where $H(f,\chi,a)=\sup \{h_\mu(F): \mu \in \mathcal{M}_{F}^e(\Sigma_N \times \mathbb{S}^1),\chi(\mu)=a\}.$
	Moreover, combining with Theorem \ref{Thm-skew} we have
	$$
	[0,h_{top}(\mathcal{L}(a)))\subset  \{h_\mu(F): \mu \in \mathcal{M}_{F}^e(\Sigma_N \times \mathbb{S}^1), \chi(\mu)=a\} \subset [0,h_{top}(\mathcal{L}(a)].
	$$
\end{maintheorem}
We draw the graph of $(\chi(\mu), h_\mu(F))$ in Figure \ref{fig-3}. The blue line denotes the supremum of metric entropy of ergodic measures with Lyapunov exponent. Theorem \ref{Thm-skew} implies the blue line coincides with the graph of 
$h_{top}(\mathcal{L}(a)).$
And Theorem \ref{maintheorem-skew} implies that every point in the interior of two closed regions of Figure \ref{fig-3} can be attained by ergodic measures.

\begin{figure}\caption{Graph of $(\chi(\mu), h_\mu(F))$}\label{fig-3}
	\begin{center}

		\tikzset{every picture/.style={line width=0.75pt}} 
		
		\begin{tikzpicture}[x=0.75pt,y=0.75pt,yscale=-1,xscale=1]
			
			\draw    (320,210.5) -- (519,210.5) ;
			\draw [shift={(521,210.5)}, rotate = 180] [color={rgb, 255:red, 0; green, 0; blue, 0 }  ][line width=0.75]    (10.93,-3.29) .. controls (6.95,-1.4) and (3.31,-0.3) .. (0,0) .. controls (3.31,0.3) and (6.95,1.4) .. (10.93,3.29)   ;
			\draw    (419,238) -- (419.99,74.5) ;
			\draw [shift={(420,72.5)}, rotate = 90.35] [color={rgb, 255:red, 0; green, 0; blue, 0 }  ][line width=0.75]    (10.93,-3.29) .. controls (6.95,-1.4) and (3.31,-0.3) .. (0,0) .. controls (3.31,0.3) and (6.95,1.4) .. (10.93,3.29)   ;
			\draw [color={rgb, 255:red, 208; green, 2; blue, 27 }  ,draw opacity=1 ]   (349,170) -- (349,210) ;
			\draw [color={rgb, 255:red, 74; green, 144; blue, 226 }  ,draw opacity=1 ]   (349,170) .. controls (364,119) and (388,82) .. (419,115) ;
			\draw [color={rgb, 255:red, 74; green, 144; blue, 226 }  ,draw opacity=1 ]   (419,115) .. controls (464,77) and (479,137) .. (490,170) ;
			\draw [color={rgb, 255:red, 208; green, 2; blue, 27 }  ,draw opacity=1 ]   (490,170) -- (490,187.5) -- (490,210) ;
			\draw    (111.05,188.74) .. controls (141.94,160.31) and (151.28,133.57) .. (148.41,123.76) ;
			\draw    (91.4,104.4) .. controls (108.73,83.06) and (143.38,102.78) .. (148.41,123.76) ;
			\draw    (65.41,177.92) .. controls (59.9,149.46) and (72.69,127.04) .. (91.4,104.4) ;
			\draw [color={rgb, 255:red, 245; green, 166; blue, 35 }  ,draw opacity=1 ][line width=3] [line join = round][line cap = round]   (90.93,147.11) .. controls (90.93,146.89) and (90.93,146.66) .. (90.93,146.43) ;
			\draw    (65.41,177.92) .. controls (67.77,189.78) and (79.54,212.46) .. (111.05,188.74) ;
			\draw    (173,153) .. controls (199.6,117.54) and (266.94,119.92) .. (296.67,150.58) ;
			\draw [shift={(298,152)}, rotate = 227.82] [color={rgb, 255:red, 0; green, 0; blue, 0 }  ][line width=0.75]    (10.93,-3.29) .. controls (6.95,-1.4) and (3.31,-0.3) .. (0,0) .. controls (3.31,0.3) and (6.95,1.4) .. (10.93,3.29)   ;
			\draw [color={rgb, 255:red, 245; green, 166; blue, 35 }  ,draw opacity=1 ][line width=3] [line join = round][line cap = round]   (447.23,170.08) .. controls (447.23,169.86) and (447.23,169.63) .. (447.23,169.41) ;
			\draw  [dash pattern={on 0.84pt off 2.51pt}]  (419,170) -- (447.23,170.08) ;
			\draw  [dash pattern={on 0.84pt off 2.51pt}]  (447.23,170.08) -- (447,210) ;
			
			\draw (95.71,131.95) node [anchor=north west][inner sep=0.75pt]    {$\mu $};
			\draw (48,212.4) node [anchor=north west][inner sep=0.75pt]    {$\mathcal{M}_{F}( \Sigma _{N} \times \mathbb{S}^1)$};
			\draw (186,248.4) node [anchor=north west][inner sep=0.75pt]    {$\mu \rightarrow ( \chi ( \mu ) ,h_{\mu }( F))$};
			\draw (401.71,212.95) node [anchor=north west][inner sep=0.75pt]    {$0$};
			\draw (431,218.4) node [anchor=north west][inner sep=0.75pt]    {$\chi ( \mu )$};
			\draw (378,162.4) node [anchor=north west][inner sep=0.75pt]    {$h_{\mu }( F)$};
			\draw (335.71,213.95) node [anchor=north west][inner sep=0.75pt]    {$a_{\min}$};
			\draw (475.71,215.95) node [anchor=north west][inner sep=0.75pt]    {$a_{\max}$};

		\end{tikzpicture}
		
	\end{center}
\end{figure}
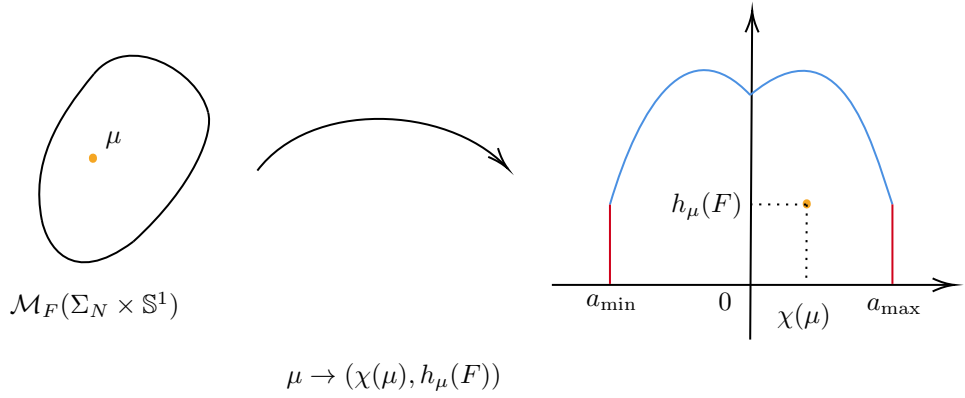

\subsubsection{Elliptic $\operatorname{SL}(2, \mathbb{R})$ cocycles}
Next, we apply Theorem \ref{maintheorem-skew} to elliptic $\operatorname{SL}(2, \mathbb{R})$ cocycles.  $\operatorname{SL}(2, \mathbb{R})$ is the set of $2 \times 2$ matrices with real coefficients and determinant one. Given $N \geq 2,$ a continuous map $A: \Sigma_N^+ \rightarrow \operatorname{SL}(2, \mathbb{R})$ is called a $2 \times 2$ matrix cocycle, where $\Sigma_N^{+}=\{0, \ldots, N-1\}^{\mathbb{N}_0}$. If $A$ is piecewise constant and depends only on the zeroth coordinate of the sequences $\xi \in \Sigma_N^+$, that is $A(\xi)=A_{\xi_0}$ where $\mathbf{A} \stackrel{\text { def }}{=}\left\{A_0, \ldots, A_{N-1}\right\} \in \operatorname{SL}(2, \mathbb{R})^N$, then we refer to it as the one-step cocycle generated by $\mathbf{A}$ or simply as the one-step cocycle $\mathbf{A}$. 
We denote
$$
\mathbf{A}^n\left(\xi^{+}\right) {=} A_{\xi_{n-1}} \circ \cdots \circ A_{\xi_1} \circ A_{\xi_0}, \quad \xi^{+} \in \Sigma_N^{+}, n \geq 0.
$$
The Lyapunov exponents of the cocycle $\mathbf{A}$ at $\xi^{+} \in \Sigma_N^{+}$ are the limits
$$
\lambda_1\left(\mathbf{A}, \xi^{+}\right) {=} \lim _{n \rightarrow \infty} \frac{1}{n} \log \left\|\mathbf{A}^n\left(\xi^{+}\right)\right\| \,
$$
where $\|L\|$ denotes the norm of the matrix $L$, whenever they exist. Given $\alpha \in \mathbb{R}$,  consider the level set
$$
\mathcal{L}_{\mathbf{A}}^{+}(\alpha) {=}\left\{\xi^{+} \in \Sigma_N^{+}: \lambda_1\left(\mathbf{A}, \xi^{+}\right)=\alpha\right\}.
$$
Given $v$ an invariant measure on $\Sigma_N^{+}$(with respect to $\sigma^{+}: \Sigma_N^{+} \rightarrow \Sigma_N^{+}$), denote
$$
\lambda_1(\mathbf{A}, \nu){=} \lim _{n \rightarrow \infty} \int \frac{1}{n} \log \left\|\mathbf{A}^n\left(\xi^{+}\right)\right\| d \nu
$$
Denote by $\langle\mathbf{A}\rangle$ the semigroup generated by $\mathbf{A}$. An element $R \in \operatorname{SL}(2, \mathbb{R})$ is elliptic if the absolute value of its trace is strictly less than 2. The set $\mathfrak{E}_N$ of elliptic cocycles is the set of cocycles $\mathbf{A} \in \operatorname{SL}(2, \mathbb{R})^N$ such that $\langle\mathbf{A}\rangle$ contains an elliptic element. In \cite{DGM2019} it is introduced an open and dense subset $\mathfrak{E}_{N \text {,shyp }}$ of $\mathfrak{E}_N$, the so-called elliptic cocycles having some hyperbolicity. 
\begin{Thm}(\cite[Theorem B]{DGM2019} and \cite[Theorem B]{DGM2022})\label{Thm-cocycle}
	For every $N \geq 2$ and  every $\mathbf{A}\in \mathfrak{E}_{N \text {,shyp }}$ there are numbers $a_{\max }>0$ such that the map $a \mapsto h_{top}\left(\mathcal{L}_{\mathbf{A}}^{+}(a)\right)$ is continuous and concave on $\left[0, a_{\max }\right]$, and for every $a \in[0, a_{\max}]$ one has
	$$
	h_{top}(\mathcal{L}_{\mathbf{A}}^{+}(a))=\sup\{h_\nu(\sigma^+): \nu \in \mathcal{M}_{\sigma^+}^e(\Sigma_N^{+}), \lambda_1(\mathbf{A}, \nu)=a\} .
	$$
\end{Thm}
Using  Theorem \ref{maintheorem-skew}, we show that every $A \in\mathfrak{E}_{N \text {,shyp }}$ has intermediate entropy property of ergodic measures with same Lyapunov exponent.
\begin{maintheorem}\label{maintheorem-cocycle}
	For every $N \geq 2$,  every $\mathbf{A}\in \mathfrak{E}_{N \text {,shyp }},$ every $a \in\left(0, a_{\max }\right)$ and every $0\leq h<h_{top}(\mathcal{L}_{\mathbf{A}}^{+}(a)),$ there exists $\nu_{a,h}\in  \mathcal{M}_{\sigma^+}^e(\Sigma_N^{+})$ such that $\lambda_1(\mathbf{A},\nu_{a,h})=a$ and $h_{\nu_{a,h}}(\sigma^+)=h,$  that is,
	$$
	[0,h_{top}(\mathcal{L}_{\mathbf{A}}^{+}(a)))\subset \{h_\nu(\sigma^+): \nu \in \mathcal{M}_{\sigma^+}^e(\Sigma_N^{+}), \lambda_1(\mathbf{A}, \nu)=a\} \subset [0,h_{top}(\mathcal{L}_{\mathbf{A}}^{+}(a))].
	$$
\end{maintheorem}

\subsubsection{Robustly non-hyperbolic transitive diffeomorphisms}
In \cite{YZ2020} Yang and Zhang study a rich family of robustly non-hyperbolic transitive diffeomorphisms and we show that each ergodic measure is approached by hyperbolic sets in weak*-topology and in entropy. 
A diffeomorphism $f$ on a smooth closed Riemannian manifold $M$ is said to be partially hyperbolic, if there exist an invariant splitting $T M=E^s \oplus E^c \oplus E^u$ and a metric $\|\cdot\|$ such that for each $x \in M$, one has
$$
\left\|\left.D f\right|_{E^s(x)}\right\|<\min \left\{1, m\left(\left.D f\right|_{E^c(x)}\right)\right\} \leqslant \max \left\{1,\left\|\left.D f\right|_{E^c(x)}\right\|\right\}<m\left(\left.D f\right|_{E^u(x)}\right) .
$$
Consider the set $\mathcal{U}(M)$ of all $C^1$ partially hyperbolic diffeomorphisms on $M$ satisfying that for each $f \in \mathcal{U}(M)$, one has:
\begin{enumerate}
	\item $f$ is partially hyperbolic with one-dimensional center bundle;
	\item $f$ has hyperbolic periodic points of different indices;
	\item the strong stable and unstable foliations are robustly minimal.
\end{enumerate}
By definition, $\mathcal{U}(M)$ is an open set. Yang and Zhang proved that there exists a $C^1$ open and dense subset $\mathcal{V}(M)$ of $\mathcal{U}(M)$ such that for any $f \in \mathcal{V}(M)$, each $f$-ergodic measure $\mu$ is approached by hyperbolic sets in weak*-topology and in entropy. 
Given $f\in \mathcal{V}(M)$ and  $x\in M$, consider the center Lyapunov exponent of $x$
$$
\chi(x) \stackrel{\text { def }}{=} \lim _{n \rightarrow \pm \infty} \frac{1}{n} \log ||Df^n|_{E^c(x)}||,
$$
where we assume that both limits $n \rightarrow \pm \infty$ exist and coincide.  Given $\alpha \in \mathbb{R}$ let
$$
\mathcal{L}(\alpha) \stackrel{\text { def }}{=}\left\{x\in M: \chi(x)=\alpha\right\}.
$$
Given an $f$-invariant measure $\mu$, denote by $\chi(\mu)$ the Lyapunov exponent of $\mu$ defined by
$$
\chi(\mu) \stackrel{\text { def }}{=} \int \log||Df^n|_{E^c(x)}||d\mu.
$$
Following the argument of Theorem \ref{maintheorem-skew}, we have the following result. 
\begin{maintheorem}\label{maintheorem-robust}
	There exists a $C^1$ open and dense subset $\mathcal{V}(M)$ of $\mathcal{U}(M)$ such that for  any $f \in \mathcal{V}(M)$, there are numbers $a_{\min }<0<a_{\max }$ such that for every $a \in\left(a_{\min }, 0\right) \cup\left(0, a_{\max }\right)$ and every $0\leq h<\sup \{h_\mu(f): \mu \in \mathcal{M}_{f}^e(M),\chi(\mu)=a\},$ there exists $\mu\in  \mathcal{M}_{f}^e(M)$ such that $\chi(\mu)=a$ and $h_\mu(f)=h,$  that is,
	$$
	[0,H(f,\chi,a))\subset  \{h_\mu(f): \mu \in \mathcal{M}_{f}^e(M), \chi(\mu)=a\} \subset [0,H(f,\chi,a)],
	$$
	where $H(f,\chi,a)=\sup \{h_\mu(f): \mu \in \mathcal{M}_{f}^e(M),\chi(\mu)=a\}.$ Moreover, if $h_{top}(\mathcal{L}(a))=\sup \{h_\mu(f): \mu \in \mathcal{M}_{f}^e(M),\chi(\mu)=a\},$ then we have
	$$
	[0,h_{top}(\mathcal{L}(a)))\subset  \{h_\mu(f): \mu \in \mathcal{M}_{f}^e(M), \chi(\mu)=a\} \subset [0,h_{top}(\mathcal{L}(a)].
	$$
\end{maintheorem}

\subsection{'Multi-horseshoe' entropy-dense property}
Now, we present the key points in the proof of Theorems \ref{thm-continuous} and \ref{thm-continuous-2}. Previous research from the perspective of topology often involves constructing periodic or ergodic measures. When considering from the perspective of entropy, a useful method is to construct a horseshoe or an invariant set with large entropy. When considering from the perspective of multifractal analysis, one idea is to construct a sequence of periodic measures. However, when combining these three perspectives, the difficulty lies in finding an ergodic measure that is  close to the given invariant measure from the perspectives of topology and entropy, and has the same integral as the given invariant measure. Unfortunately, none of the aforementioned methods can satisfy all three requirements simultaneously.
To overcome this difficulty, we propose using the conditional variational principle locally to find ergodic measures. To do so, we establish the 'multi-horseshoe' entropy-dense property, as stated in Theorem \ref{Mainlemma-convex-by-horseshoe} and \ref{def-strong-basic-2}. By using this property, we can use the conditional variational principle locally to find ergodic measures that satisfy all three requirements, ultimately leading to the proof of our theorems.

We give the precise statement of the 'multi-horseshoe' entropy-dense property. This property has its  independent significance. We believe it is  a powerful tool and potentially has plenty of applications in other problems.

For any $m\in\N$ and $\{\nu_i\}_{i=1}^m \subseteq \mathcal{M}(X)$,   we write $\cov\{\nu_i\}_{i=1}^m$ for the convex combination of $\{\nu_i\}_{i=1}^m$,   namely,
$$\cov\{\nu_i\}_{i=1}^m=\cov(\nu_1,\cdots,\nu_m):=\left\{\sum_{i=1}^mt_i\nu_i:t_i\in[0,  1],  1\leq i\leq m~\textrm{and}~\sum_{i=1}^mt_i=1\right\}.$$ We denote the \textit{Hausdorff distance} between two nonempty subsets of $\mathcal{M}(X),$ $A$ and $B,$  by $$d_H(A,  B):=\max\set{\sup_{\mu\in A}\inf_{\nu\in B}\rho(\mu,  \nu),\sup_{\nu\in B}\inf_{\mu\in A}\rho(\nu, \mu)  }.$$
\begin{Def}\label{def-strong-basic-A}
	We say $(X,  f)$ satisfies the {\it 'multi-horseshoe' entropy-dense property} (abbrev. {\it 'multi-horseshoe' dense property}) if for any positive integer $m,$ any $f$-invariant measures $\{\mu_i\}_{i=1}^m\subseteq \mathcal{M}_f(X),$ any $x\in X$ and any $\eta,  \zeta>0$,   there exist compact invariant subsets $\Lambda_i\subseteq\Lambda\subsetneq X$ such that for each $1\leq i\leq m$
	\begin{enumerate}
		\item $(\Lambda_i,f)$ and $(\Lambda,f)$ conjugate to transitive two-sided subshifts of ﬁnite type (and thus they are transitive and topologically Anosov).
		\item $\htop(f,  \Lambda_i)>h_{\mu_i}(f)-\eta$. 
		\item $d_H(K,  \mathcal{M}_f(\Lambda))<\zeta$,   $d_H(\mu_i,  \mathcal{M}_f(\Lambda_i))<\zeta$, where $K=\cov\{\mu_i\}_{i=1}^m.$
	\end{enumerate}
\end{Def}
\begin{Cor}\label{Cor-interior}
	Suppose $(X,  f)$ is a dynamical system. If $(X,  f)$ satisfies the 'multi-horseshoe' dense property, then for any continuous function $\varphi$ on $X,$ any $a\in \mathrm{Int}(\mathcal{P}_{\varphi}(\mathcal{M}_f(X))),$ any $\mu\in \mathcal{P}_{\varphi}^{-1}(a)$ and any $ \zeta>0$, there is a compact invariant subset $\Lambda\subset X$ such that
	$a\in  \mathrm{Int}(\mathcal{P}_{\varphi}(\mathcal{M}_f(\Lambda))),$ $\rho(\mu,  \nu)<\zeta$ for any $\nu\in \mathcal{M}_f(\Lambda),$ and $(\Lambda,f)$ conjugates to a transitive two-sided subshift of ﬁnite type.
\end{Cor}
In Section \ref{Almost Additive} we give a more general result than Corollary \ref{Cor-interior}. See Lemma \ref{Lem-interior}.

\begin{maintheorem}\label{Mainlemma-convex-by-horseshoe-A}
	Suppose $(X,  f)$ is topologically Anosov and transitive.   Then $(X,  f)$ satisfies the 'multi-horseshoe' dense property.
\end{maintheorem}

\subsection{Intermediate Hausdorff  dimension of ergodic measures}
Given a Borel probability measure $\mu$ on a compact Riemannian manifold $M$, the Hausdorff dimension of the measure $\mu$ is defined as
$$
\dim_H \mu=\inf \left\{\dim_H Y: Y \subset M \text{ and } \mu(Y)=1\right\},
$$
where $\dim_H Y $ is the Hausdorff dimension of $Y.$
There are very different properties between  Hausdorff dimension and entropy. When the entropy map $\mu\to h_{\mu}(f)$ of $(M,f)$ is upper semi-continuous, there exists $\mu_{\max}\in \mathcal{M}_f(M)$ such that  $h_{\mu_{\max}}(f)=\sup\{h_{\mu}(f):\mu\in\mathcal{M}_f(M)\}.$ 
However, the map $\mu \mapsto \operatorname{dim}_H \mu$ enjoys no continuity property even if the entropy map is upper semi-continuous.  So it is a question to prove the existence of measures of maximal dimension, that is, try to find an invariant measure that attains the supremum of the  quantity
$
\sup \left\{\operatorname{dim}_H \mu: \mu\in\mathcal{M}_f(M)\right\}.
$
For hyperbolic diffeomorphisms, Barreira and Wolf \cite{BarreiraWolf} show that if $f:M\to M$ is a diffeomorphism on a compact surface and $\Lambda$ is a topological mixing locally maximal hyperbolic set, then there exists an ergodic measure of maximal dimension, see \cite[Ch.5]{Barreira2008} for the hyperbolic conformal case. In \cite{Rams2005}, Rams gave the existence of a measure of maximal dimension for piecewise linear horseshoe maps. In \cite{Wolf}, Wolf proved that there exist ﬁnitely many measures of maximal dimension for polynomial automorphisms of $C^2$.
Recently, Chen, Luo, and Zhao \cite[Theorem C]{CLZ2018}  show that there exists an ergodic measure of maximal dimension if $f: M \mapsto M$ is a $C^{1+\alpha}$ diffeomorphism on a compact Riemannian manifold $M$ and  $\Lambda\subset M$ is a mixing locally maximal hyperbolic set which is  average conformal.
From their result, one has 
\begin{equation}\label{max-dim}
	\max \left\{\operatorname{dim}_H \mu: \mu\in\mathcal{M}_f^e(M)\right\}=\max \left\{\operatorname{dim}_H \mu: \mu\in\mathcal{M}_f(M)\right\}.
\end{equation}
It is natural to raise the following question on  intermediate Hausdorff  dimension of ergodic measures:
\begin{mainquestion}\label{Conjecture-3}
	For every typical diffeomorphism $f$ on a compact Riemannian manifold $M$,  whether one has
	$$
	 \left\{\operatorname{dim}_H \mu: \mu\in\mathcal{M}_f^e(M)\right\}= \left\{\operatorname{dim}_H \mu: \mu\in\mathcal{M}_f(M)\right\}?
	$$
\end{mainquestion}

Using our result on multiple functions (Theorem \ref{thm-almost}), we answer Question \ref{Conjecture-3} for average conformal Anosov diffeomorphisms and diffeomorphisms with hyperbolic ergodic measures. 

Given  $(X,f),$ $d \in \mathbb{N}$ and $\Phi_d=\{\varphi_i\}_{i=1}^{d}\subset C(X)$.  Denote $\mathcal{P}_{\Phi_d}(\mu)=(\int \varphi d\mu,\dots,\int \varphi_d d\mu)$ for any $\mu\in\mathcal{M}_f(X).$  Then  $\mathcal{P}_{\Phi_d}(\mathcal{M}_f(X))$ is a convex compact subset of $\mathbb{R}^d.$  Denote the relative interior of a set $C$ by $\mathrm { relint } (C)$ (see definition in Section \ref{sec-convex}).

\begin{maintheorem}\label{thm-inter-huasdorff}
	Let $f: M \mapsto M$ be a $C^{1}$ diffeomorphism on a compact Riemannian manifold $M$ such that $M$ is  average conformal. 
	\begin{enumerate}
		\item [(1)] If $f: M \mapsto M$ is a transitive  Anosov diffeomorphism, then $$[0,\sup\limits_{\mu\in \mathcal{M}_f^e(M)}\dim_H\mu)\subset \{\dim_H\mu:\mu\in \mathcal{M}_f^e(M)\}\subset  [0,\sup\limits_{\mu\in \mathcal{M}_f^e(M)}\dim_H\mu].$$
		Moreover, given $d\in \mathbb{N}$ and $\Phi_{d}=\{\varphi_i\}_{i=1}^{d}\subset C(M)$, then for any $a\in \mathrm { relint } (\mathcal{P}_{\Phi_{d}}(\mathcal{M}_f(M))),$ we have $\sup\limits_{\mu\in \mathcal{M}_f^e(M)\cap   \mathcal{P}_{\Phi_{d}}^{-1}(a)}\dim_H\mu>0$ and $$[0,\sup\limits_{\mu\in \mathcal{M}_f^e(M)\cap   \mathcal{P}_{\Phi_{d}}^{-1}(a)}\dim_H\mu)\subset \{\dim_H\mu:\mu\in \mathcal{M}_f^e(M)\cap   \mathcal{P}_{\Phi_{d}}^{-1}(a)\}\subset  [0,\sup\limits_{\mu\in \mathcal{M}_f^e(M)\cap   \mathcal{P}_{\Phi_{d}}^{-1}(a)}\dim_H\mu].$$ 
		\item [(2)] If $f: M \mapsto M$ is a $C^{1+\alpha}$  transitive Anosov diffeomorphism, then $$\{\operatorname{dim}_H \mu: \mu\in\mathcal{M}_f^e(M)\}= \left\{\operatorname{dim}_H \mu: \mu\in\mathcal{M}_f(M)\right\}.$$
		\item [(3)] Let $\nu$ be an $f$-invariant ergodic hyperbolic measure. If $f$ is $C^{1+\alpha}$ for some $0<\alpha<1$ or the Oseledec splitting of $\nu$ is dominated,   then $$\{\operatorname{dim}_H \mu: \mu\in\mathcal{M}_f^e(M)\}\supset [0,\dim_H\nu].$$ 
		 In particular, if further $\dim_H\nu=\sup\limits_{\mu\in \mathcal{M}_f^e(M)}\dim_H\mu,$ then $$\{\operatorname{dim}_H \mu: \mu\in\mathcal{M}_f^e(M)\}= \left\{\operatorname{dim}_H \mu: \mu\in\mathcal{M}_f(M)\right\}.$$
	\end{enumerate}
\end{maintheorem}

\begin{figure}\caption{Graph of $(\int \varphi d\mu, \dim_H\mu)$}\label{fig-2}
	\begin{center}

		\tikzset{every picture/.style={line width=0.75pt}} 
		
		\begin{tikzpicture}[x=0.75pt,y=0.75pt,yscale=-1,xscale=1]
			
			\draw    (135.05,170.74) .. controls (165.94,142.31) and (175.28,115.57) .. (172.41,105.76) ;
			\draw    (115.4,86.4) .. controls (132.73,65.06) and (167.38,84.78) .. (172.41,105.76) ;
			\draw    (89.41,159.92) .. controls (83.9,131.46) and (96.69,109.04) .. (115.4,86.4) ;
			\draw [color={rgb, 255:red, 245; green, 166; blue, 35 }  ,draw opacity=1 ][line width=3] [line join = round][line cap = round]   (114.93,129.11) .. controls (114.93,128.89) and (114.93,128.66) .. (114.93,128.43) ;
			\draw [color={rgb, 255:red, 245; green, 166; blue, 35 }  ,draw opacity=1 ][line width=3] [line join = round][line cap = round]    ;
			\draw    (89.41,159.92) .. controls (91.77,171.78) and (103.54,194.46) .. (135.05,170.74) ;
			\draw [color={rgb, 255:red, 126; green, 211; blue, 33 }  ,draw opacity=1 ]   (353.32,73.75) .. controls (420.27,38.18) and (478.56,83.24) .. (495.1,96.68) ;
			\draw [color={rgb, 255:red, 208; green, 2; blue, 27 }  ,draw opacity=1 ]   (495.1,196.28) -- (494,68) ;
			\draw [color={rgb, 255:red, 245; green, 166; blue, 35 }  ,draw opacity=1 ][line width=3] [line join = round][line cap = round]   (392.23,148.08) .. controls (392.23,147.86) and (392.23,147.63) .. (392.23,147.41) ;
			\draw  (306.06,195.68) -- (514,195.68)(328.52,20) -- (328.52,218.42) (507,190.68) -- (514,195.68) -- (507,200.68) (323.52,27) -- (328.52,20) -- (333.52,27)  ;
			\draw [color={rgb, 255:red, 208; green, 2; blue, 27 }  ,draw opacity=1 ]   (354.5,58.25) -- (353.32,196.28) ;
			\draw  [dash pattern={on 0.84pt off 2.51pt}]  (392.23,148.08) -- (392.23,195.51) ;
			\draw  [dash pattern={on 0.84pt off 2.51pt}]  (392.23,148.08) -- (328.43,148.08) ;
			\draw [color={rgb, 255:red, 74; green, 144; blue, 226 }  ,draw opacity=1 ]   (354.5,58.25) .. controls (394.5,28.25) and (471.5,45.75) .. (494,68) ;
			\draw    (176,126) .. controls (215.4,96.45) and (266.44,107.65) .. (296.64,132.84) ;
			\draw [shift={(298,134)}, rotate = 220.91] [color={rgb, 255:red, 0; green, 0; blue, 0 }  ][line width=0.75]    (10.93,-3.29) .. controls (6.95,-1.4) and (3.31,-0.3) .. (0,0) .. controls (3.31,0.3) and (6.95,1.4) .. (10.93,3.29)   ;
			
			\draw (119.71,113.95) node [anchor=north west][inner sep=0.75pt]    {$\mu $};
			\draw (79.6,197.77) node [anchor=north west][inner sep=0.75pt]    {$\mathcal{M}_{f}( X)$};
			\draw (313.45,195.8) node [anchor=north west][inner sep=0.75pt]    {$0$};
			\draw (271.92,143.05) node [anchor=north west][inner sep=0.75pt]    {$\dim_{H} \mu $};
			\draw (375.33,197.65) node [anchor=north west][inner sep=0.75pt]    {$\int \varphi d\mu $};
			\draw (166,228.4) node [anchor=north west][inner sep=0.75pt]    {$\mu \rightarrow \left(\int \varphi d\mu ,\dim_{H} \mu \right)$};

		\end{tikzpicture}
		
	\end{center}
\end{figure}
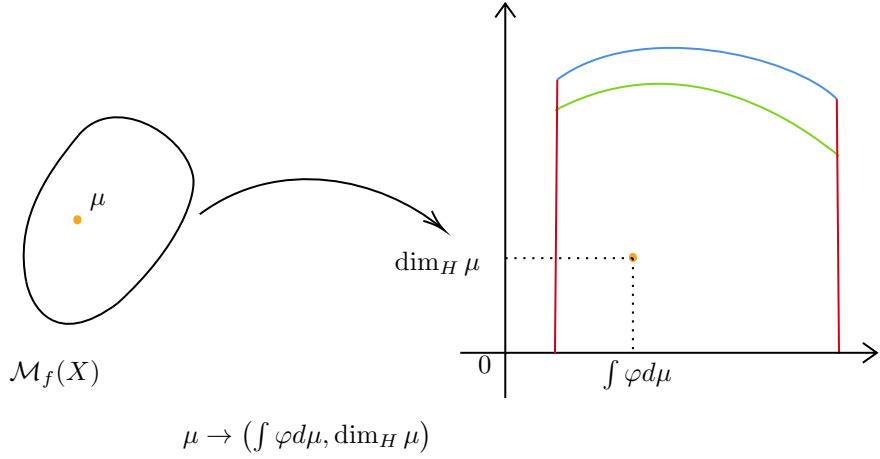
We draw the graph of $(\int \varphi d\mu, \dim_H\mu).$ The green line denotes the supremum of Hausdorff  dimension of ergodic measures with same level, and the blue line denotes the supremum of Hausdorff  dimension of invariant measures with same level. Theorem \ref{thm-inter-huasdorff} implies that every point in the interior of the smaller closed region of Figure \ref{fig-2} can be attained by ergodic measures.  We don't know the relationship between the blue line and the green line.

\begin{Rem}
	Note that every quasi-conformal diffeomorphism is average conformal. So the result of Theorem \ref{thm-inter-huasdorff} holds for quasi-conformal transitive  Anosov diffeomorphisms. See definitions of quasi-conformal and average conformal in Section \ref{subsection-LayHyp}.
\end{Rem}

A basic known strategy to obtain intermediate entropies of ergodic measures is combining refined entropy-dense property and upper semi-continuity of entropy function. However, it's very difficult to prove refined Hausdorff  dimension-dense property and upper semi-continuity of Hausdorff  dimension function $\mu \mapsto \operatorname{dim}_H \mu$. In this paper, we  use item (III) and (IV) of Theorem \ref{thm-almost} on $\psi^u(x)=\log|\det D_xf|_{E_x^{u}}|$ and $\psi^s(x)=\log|\det D_xf|_{E_x^{s}}|$  to show that $\mathrm{Int}(\mathcal{T}_{\psi^u,\psi^s,f}(\mathcal{M}_f^e(M)))$  is a nonempty convex subset of $\mathbb{R}^3,$ where $\mathcal{T}_{\psi^u,\psi^s,f}(\mu)=(\int \psi^ud\mu,\int \psi^sd\mu, h_\mu(f)).$ Then we construct a continuous map $\mathcal{Q}$ from $\mathbb{R}^3$ to $\mathbb{R}$ such that $\mathcal{Q}(\mathcal{T}_{\psi^u,\psi^s,f}(\mu))=\dim_H\mu$ for any $f$-ergodic measure $\mu.$
This allow us to obtain that  $\{\operatorname{dim}_H \mu: \mu\in\mathcal{M}_f^e(M)\}$ is an interval containing $[0,\sup\limits_{\mu\in \mathcal{M}_f^e(M)}\dim_H\mu).$

\subsection{Lyapunov spectrum}
Now we consider Lyapunov spectrum. First, we introduce a conception, multi-average conformal. Given an invariant measure $\mu$,  for $\mu$ a.e. $x$,  denote by $$
\chi_1(x) \geq \chi_2(x) \geq \cdots \geq \chi_{\dim M}(x).
$$
the Lyapunov exponents at $x.$ 
For any $1\leq i\leq \dim M,$ denote $\chi_i(\mu)=\int \chi_i(x) d\mu.$ 
We say a  Anosov diffeomorphism $f:M\to M$ is multi-average conformal, if  there are $t_u,t_s\in\mathbb{N^{+}},$ $d_1,d_2,\dots,d_{t_u+t_s}\in\mathbb{N^{+}},$   and $E^{1},\dots,E^{t_u+t_s}\subset TM$ for a such that    
\begin{enumerate}
	\item   $\sum_{j=1}^{t_u}d_j=\dim E^u,$  $\sum_{j=t_u+1}^{t_u+t_s}d_j=\dim E^s$ and $\dim E^j=d_j$ for  $1\leq j\leq  t_u+t_s.$
	\item  $D_xf(E_x^j)=E_{f(x)}^j$  for any $x\in M$ and $1\leq j\leq  t_u+t_s.$ 
	\item    $x\to E^{j}_x$ is continuous for any $1\leq j\leq  t_u+t_s.$
	\item $E_x^u=E^1_x\oplus \dots\oplus E^{t_1}_x,$ $E_x^s=E^{t_2+1}\oplus \dots\oplus E^{\dim M}_x$ for any $x\in M$.
	\item $\chi_{\sum_{k=1}^{j}d_k+1}(\mu)=\chi_{\sum_{k=1}^{j}d_k+2}(\mu)\dots=\chi_{\sum_{k=1}^{j+1}d_k}(\mu)$ for any $0\leq j\leq t_u+t_s-1$ and any $\mu\in\mathcal{M}_f(M).$
\end{enumerate}   
Denote $\mathrm{Lya}(\mu)=(\chi_1(\mu),\dots,\chi_{\dim M}(\mu))$ for any $\mu\in \mathcal{M}_f(M).$ Then  $\mathrm{Lya}(\mathcal{M}_f(X))$ is a subset of $\mathbb{R}^{t_u+t_s}.$
Denote $\mathcal{T}_{\mathrm{Lya},f}(\mu)=(\mathrm{Lya}(\mu),h_\mu(f))$ for any $\mu\in \mathcal{M}_f(M).$ Then  $\mathcal{T}_{\mathrm{Lya},f}(\mathcal{M}_f(X))$ is a subset of $\mathbb{R}^{t_u+t_s+1}.$
Using our result on multiple functions (Theorem \ref{thm-almost}),  we have the following result.
\begin{maintheorem}\label{thm-Lyapunov}
	Let $f: M \mapsto M$ be a $C^1$ transitive multi-average conformal  Anosov diffeomorphism on a compact Riemannian manifold $M$. Then   
	\begin{enumerate}
		\item [(1)] $\mathrm { relint } (\mathrm{Lya}(\mathcal{M}_f(X)))=\mathrm { relint } (\mathrm{Lya}(\mathcal{M}_f^e(X))).$
		\item [(2)] $\mathrm { relint } (\mathcal{T}_{\mathrm{Lya},f}(\mathcal{M}_f(X)))=\mathrm { relint } (\mathcal{T}_{\mathrm{Lya},f}(\mathcal{M}_f^e(X))).$
		\item [(3)] if $t_u=t_s=1,$ then for any $a\in \mathrm { relint } (\mathrm{Lya}(\mathcal{M}_f(X))),$ we have $\sup\limits_{\mu\in \mathcal{M}_f^e(M)\cap   \mathrm{Lya}^{-1}(a)}\dim_H\mu>0$ and $$[0,\sup\limits_{\mu\in \mathcal{M}_f^e(M)\cap   \mathrm{Lya}^{-1}(a)}\dim_H\mu)\subset \{\dim_H\mu:\mu\in \mathcal{M}_f^e(M)\cap   \mathrm{Lya}^{-1}(a)\}\subset  [0,\sup\limits_{\mu\in \mathcal{M}_f^e(M)\cap   \mathrm{Lya}^{-1}(a)}\dim_H\mu].$$ 
	\end{enumerate} 
\end{maintheorem}
\begin{Rem}
	For Anosov diffeomorphisms, $\mathrm { Int } (\mathrm{Lya}(\mathcal{M}_f(X)))$ may be empty, but  $\mathrm { relint } (\mathrm{Lya}(\mathcal{M}_f(X)))\neq\emptyset$ always holds. For example, when $f$ is an Anosov toral automorphism, $\mathrm{Lya}(\mathcal{M}_f(X)$ is a singleton, then $\mathrm { Int } (\mathrm{Lya}(\mathcal{M}_f(X)))=\emptyset,$ $\mathrm { relint } (\mathrm{Lya}(\mathcal{M}_f(X)))=\mathrm{Lya}(\mathcal{M}_f(X)).$
\end{Rem}
In \cite{TWW2019} the authors proved that for any system $(X,f)$ satisfying  the periodic gluing orbit property and any continuous function $\varphi\in C(X)$, one has $\{\int \varphi d\mu:\mu\in\mathcal{M}_f(X)\}=\{\int \varphi d\mu:\mu\in\mathcal{M}_f^e(X)\}.$ Using this result, we can obtain 
$\{\chi_i(\mu):\mu\in \mathcal{M}_f(M)\}=\{\chi_i(\mu):\mu\in \mathcal{M}_f^e(M)\}$ for any $1\leq i\leq \dim M.$ However, the result of \cite{TWW2019} can not be extended from one function to multiple functions, since (3.9) of \cite{TWW2019} may fail for multiple functions. So we can not obtain Theorem \ref{thm-Lyapunov} using the method of \cite{TWW2019}.  In this paper, we  use Theorem \ref{thm-almost} on $\psi^i(x)=\log|\det Df|_{E_x^{i}}|$, $1\leq i\leq t_u+t_s$  to obtain Theorem  \ref{thm-Lyapunov}.


\subsection{First return rate of ergodic measures}
Given $x \in M$ and $r>0$, denote the first return time of a ball $B(x, r)$ radius $r$ at $x$ by
$$
\tau(B(x, r)):=\min \left\{k>0 : f^k(B(x, r)) \cap B(x, r) \neq \emptyset\right\}
$$
Define the first return rate of $x$ by $$r_f(x)=\lim _{r \rightarrow 0} \frac{\tau(B(x, r))}{-\log r}$$  if the limit exists. 
For any $\mu\in \mathcal{M}_f(M),$ we define first return rate of $\mu$ by $r_f(\mu)=\int r_f(x) d\mu$. For transitive  Anosov diffeomorphisms, if $\mu$ is an ergodic measure with $h_\mu(f)>0$ and with only two Lyapunov exponents  $\chi_s(\mu)<0<\chi_u(\mu)$, then $r_f(\mu)=\frac{1}{\lambda_u(\mu)}-\frac{1}{\lambda_s(\mu)}$ by  \cite{STV2003,OliveTian 2013}.
Using our result on multiple functions (Theorem \ref{thm-almost}), we show that $\{r_f(\mu):\mu\in \mathcal{M}_f^e(\Lambda) \text{ and }h_\mu(f)>0\}$ is an interval.
\begin{maintheorem}\label{thm-first-return}
	Let $f: M \mapsto M$ be a $C^{1}$ transitive  Anosov diffeomorphism on a compact Riemannian manifold $M$ such that $M$ is  average conformal. Then $\{r_f(\mu):\mu\in \mathcal{M}_f^e(\Lambda) \text{ and }h_\mu(f)>0\}$ is an interval. In particular,  if there exist $\mu_1,\mu_2\in \mathcal{M}_f^e(M)$ with $h_{\mu_1}(f)>0,h_{\mu_2}(f)>0$ and $r_f(\mu_1)<r_f(\mu_2),$ then for any $r_f(\mu_1)<r<r_f(\mu_2),$ there is $\mu \in \mathcal{M}_f^e(M)$ such that $h_{\mu}(f)>0$ and $r_f(\mu)=r.$
\end{maintheorem}
The proof of Theorem \ref{thm-first-return} is similar to Theorem \ref{thm-inter-huasdorff}. We  use Theorem \ref{thm-almost}(I) on $\psi^u(x)=\log|\det D_xf|_{E_x^{u}}|$ and $\psi^s(x)=\log|\det D_xf|_{E_x^{s}}|$  to show that $\mathrm{Int}(\mathcal{T}_{\psi^u,\psi^s}(\mathcal{M}_f^e(\Lambda)))$  is a nonempty convex subset of $\mathbb{R}^2,$ where $\mathcal{T}_{\psi^u,\psi^s}(\mu)=(\int \psi^ud\mu,\int \psi^sd\mu).$ Then we construct a continuous map $\mathcal{Q}$ from $\mathbb{R}^2$ to $\mathbb{R}$ such that $\mathcal{Q}(\mathcal{T}_{\psi^u,\psi^s}(\mu))=r_\mu$ for $\mu\in \mathcal{M}_f^e(M)$ with $h_\mu(f)>0.$
This allow us to obtain that  $\{r_\mu: \mu\in\mathcal{M}_f^e(M)\text{ and }h_\mu(f)>0\}$ is an interval.

\subsection{Other applications}
In this subsection, let $f: M \mapsto M$ be a $C^1$ transitive Anosov diffeomorphism on a compact Riemannian manifold $M.$   Given an invariant measure $\mu \in\mathcal{M}_f(M)$,  for $\mu$ a.e. $x$,  denote by $$
\chi_1(x) \geq \chi_2(x) \geq \cdots \geq \chi_{\dim M}(x).
$$
the Lyapunov exponents at $x.$ 

\subsubsection{Intermediate geometric pressure of ergodic measures}
Denote
$
\chi_i^{+}(x)=\max \left\{\chi_i(x), 0\right\}.
$
By Ruelle’s inequality \cite{Ruelle1978}, for any $\mu \in \mathcal{M}_f(M)$ one has
$
h_\mu(f) \leq \int \sum_{i=1}^{\operatorname{dim} M} \chi_i^{+}(x) d \mu.
$
We define the geometric pressure of $\mu$ by $P^{u}(\mu)=h_{\mu}(f)-\int \sum_{i=1}^{\operatorname{dim} M} \chi_i^{+} d\mu.$
Then
$-\sup _{\mu\in \mathcal{M}_f(M)}\int \sum_{i=1}^{\operatorname{dim} M} \chi_i^{+} d\mu \leq P^{u}(\mu)\leq 0$ for any $\mu\in \mathcal{M}_f(M).$  Using Theorem \ref{thm-continuous}, we show that $f$ has intermediate geometric pressure of ergodic measures.
\begin{maincorollary}\label{thm-pressure}
	Let $f: M \mapsto M$ be a $C^{1}$ diffeomorphism on a compact Riemannian manifold $M$. 
	\begin{enumerate}
		\item [(1)] If $f: M \mapsto M$ is a transitive  Anosov diffeomorphism, then $$(-\psi^u_{\sup},0]\subseteq \{P^{u}(\mu):\mu\in \mathcal{M}_f^e(M)\}\subseteq [-\psi^u_{\sup},0],$$ where $\psi^u_{\sup}=\sup\limits_{\mu\in \mathcal{M}_f(M)}\int \sum_{i=1}^{\operatorname{dim} M} \chi_i^{+} d\mu.$ 
		Moreover, given $d\in \mathbb{N}$ and $\Phi_{d}=\{\varphi_i\}_{i=1}^{d}\subset C(M)$, then for any $a\in \mathrm { relint } (\mathcal{P}_{\Phi_{d}}(\mathcal{M}_f(M))),$   $\{P^u(f):\mu\in \mathcal{M}_f^e(M)\cap   \mathcal{P}_{\Phi_{d}}^{-1}(a)\}$ is an interval and $\overline{\{P^u(f):\mu\in \mathcal{M}_f^e(M)\cap   \mathcal{P}_{\Phi_{d}}^{-1}(a)\}}=\overline{\{P^u(f):\mu\in \mathcal{M}_f(M)\cap   \mathcal{P}_{\Phi_{d}}^{-1}(a)\}}.$
		\item [(2)] Let $\nu$ be an $f$-invariant ergodic hyperbolic measure. If $f$ is $C^{1+\alpha}$ for some $0<\alpha<1$ or the Oseledec splitting of $\nu$ is dominated,   then $$\{P^{u}(\mu):\mu\in \mathcal{M}_f^e(M)\}\supset (-\int \sum_{i=1}^{\operatorname{dim} M} \chi_i^{+} d\nu,P^{u}(\nu)].$$ 
		In particular, if further $P^{u}(\nu)=0$, then $$(-\int \sum_{i=1}^{\operatorname{dim} M} \chi_i^{+} d\nu,0]\subset \{P^{u}(\mu):\mu\in \mathcal{M}_f^e(M)\}\subset [-\int \sum_{i=1}^{\operatorname{dim} M} \chi_i^{+} d\nu,0].$$
	\end{enumerate}
\end{maincorollary}
\begin{Rem}
	An invariant measure $\mu$ is said to satisfy the Pesin entropy formula, if $P^{u}(\mu)= 0.$ From Corollary \ref{thm-pressure}, the ergodic measures that do not satisfy the Pesin entropy formula are very abundant.
\end{Rem}
In \cite{Sun2019} the author proved that for any system $(X,f)$ satisfying the  approximate product property  and asymptotically entropy expansiveity, and any continuous function $\varphi\in C(X)$, one has $$(\inf_{\mu\in \mathcal{M}_f(X)}P_\varphi(\mu),\sup_{\mu\in \mathcal{M}_f(X) }P_\varphi(\mu)]\subset \{P_\varphi(\mu):\mu\in\mathcal{M}_f^e(X)\},$$ where $P_\varphi(\mu)=h_\mu(f)+\int \varphi d\mu.$ Using this result, we can obtain 
$(\inf\limits_{\mu\in \mathcal{M}_f(X)}P^u(\mu),0]\subset \{P^u(\mu):\mu\in\mathcal{M}_f^e(X)\}$ under the assumption of Corollary \ref{thm-pressure}. However, we can not obtain  Corollary \ref{thm-pressure}, since $-\psi^u_{\sup}\leq \inf\limits_{\mu\in \mathcal{M}_f(X)}P^u(\mu)$.  In this paper, we  will use Theorem \ref{thm-continuous} on on $\psi^u(x)=\log|\det D_xf|_{E_x^{u}}|$  to show that $\{P^u(\mu):\mu\in\mathcal{M}_f^e(X)\}$ is an interval and  $-\psi^u_{\sup}\in \overline{ \{P^u(\mu):\mu\in\mathcal{M}_f^e(X)\}}$, and thus $-\psi^u_{\sup}= \inf\limits_{\mu\in \mathcal{M}_f(X)}P^u(\mu)$.

\subsubsection{Intermediate unstable Hausdorff  dimension of ergodic measures}  
Given an invariant measure $\mu \in\mathcal{M}_f(M)$,  define the unstable dimension of $\mu$  by $\dim^{u}_H\mu=\frac{h_{\mu}(f)}{ \int \sum_{i=1}^{\operatorname{dim} M} \chi_i^{+} d\mu}.$
Then by Ruelle’s inequality \cite{Ruelle1978},  one has 
$\dim^{u}_H\mu\leq 1$ for any $\mu\in \mathcal{M}_f(M).$  Using Theorem \ref{thm-continuous} to $\psi^u$, we show that $f$ has intermediate unstable Hausdorff dimension of ergodic measures.
\begin{maincorollary}\label{thm-inter-huasdorff-3}
	Let $f: M \mapsto M$ be a $C^{1}$ diffeomorphism on a compact Riemannian manifold $M$. 
	\begin{enumerate}
		\item [(1)] If $f: M \mapsto M$ is a transitive  Anosov diffeomorphism, then $$\{\dim^{u}_H\mu:\mu\in \mathcal{M}_f^e(M)\}=\{\dim^{u}_H\mu:\mu\in \mathcal{M}_f(M)\}= [0,1].$$ 
		Moreover, given $d\in \mathbb{N}$ and $\Phi_{d}=\{\varphi_i\}_{i=1}^{d}\subset C(M)$, then for any $a\in \mathrm { relint } (\mathcal{P}_{\Phi_{d}}(\mathcal{M}_f(M))),$ we have $\dim_H^u(a):=\sup\limits_{\mu\in \mathcal{M}_f^e(M)\cap   \mathcal{P}_{\Phi_{d}}^{-1}(a)}\dim_H^u\mu>0$  and $$[0,\dim_H^u(a))\subset \{\dim_H^u\mu:\mu\in \mathcal{M}_f^e(M)\cap   \mathcal{P}_{\Phi_{d}}^{-1}(a)\}\subset  [0,\dim_H^u(a)],$$ $$\{\dim_H^u\mu:\mu\in \mathcal{M}_f(M)\cap   \mathcal{P}_{\Phi_{d}}^{-1}(a)\}= [0,\dim_H^u(a)].$$
		\item [(2)] Let $\nu$ be an $f$-invariant ergodic hyperbolic measure. If $f$ is $C^{1+\alpha}$ for some $0<\alpha<1$ or the Oseledec splitting of $\nu$ is dominated,   then  $$\{\dim^{u}_H\mu:\mu\in \mathcal{M}_f^e(M)\}\supset [0,\dim_H^u\nu].$$ In particular, if further $\dim_H^u\nu=1$, then $$\{\dim^{u}_H\mu:\mu\in \mathcal{M}_f^e(M)\}=\{\dim^{u}_H\mu:\mu\in \mathcal{M}_f(M)\}= [0,1].$$
	\end{enumerate}
\end{maincorollary}
\begin{Rem}
	A central topic in differential dynamical systems is studying a class of 'good' measures, called Sinai-Ruelle-Bowen measures, which satisfy $\dim^{u}_H\mu=1$. From Corollary \ref{thm-inter-huasdorff-3}, the 'bad' measures are very abundant.
\end{Rem}
The proof of Corollary  \ref{thm-inter-huasdorff-3} is similar to Corollary \ref{thm-pressure}. We  use Theorem \ref{thm-continuous} on $\psi^u(x)=\log|\det D_xf|_{E_x^{u}}|$ to show  $\{\dim^{u}_H\mu:\mu\in \mathcal{M}_f^e(M)\}$ is an interval containing $[0,\sup\limits_{\mu\in\mathcal{M}_f^e(M)}\dim_H^u\mu)$. Combining with Ruelle’s inequality and the result of \cite[Corollary 2]{CatsCerEnri2011}, we have  $\max\limits_{\mu\in\mathcal{M}_f^e(M)}\dim_H^u\mu=1$ and thus get  Corollary  \ref{thm-inter-huasdorff-3}.

\subsubsection{Intermediate Hausdorff  dimension of generic points of ergodic measures}
Let $f: M \mapsto M$ be a $C^1$ diffeomorphism on a compact surface $M$. Then it has only two Lyapunov exponents $\chi_u(\nu)>0$ and $\chi_s(\nu)<0$ with respect to each $\nu \in \mathcal{M}_f^e(M)$.
Given $\mu\in\mathcal{M}_f(M),$ recall the set of generic points of $\mu$ is $G_\mu=\{x\in M:\lim\limits_{n\to \infty}\frac{1}{n}\sum_{i=0}^{n-1}\varphi(f^i(x))=\int \varphi d\mu \text{ for any continuous }\varphi\}$ In \cite{Manning}, Manning proved that  the Hausdorff dimension, denoted by $\delta_\mu,$ of $
G_\mu \cap W_{\mathrm{loc}}^u(x)$ is independent of  $x \in M$, and  $\delta_\mu=\frac{h_{\mu}(f)}{ \chi_u(\mu) }=\dim^{u}_H\mu,$ if $f: M \mapsto M$ is a $C^1$ transitive Anosov diffeomorphism on a compact surface  $M$ and $\mu$ is ergodic. 
 By Corollary \ref{thm-inter-huasdorff-3},  $f$ has intermediate  Hausdorff dimension of generic points of ergodic measures.
\begin{maincorollary}
	Let $f: M \mapsto M$ be a $C^1$ transitive Anosov diffeomorphism on a compact surface $M$. Then $$\{\delta_\mu:\mu\in \mathcal{M}_f^e(M)\}= [0,1].$$ Moreover, given $d\in \mathbb{N}$ and $\Phi_{d}=\{\varphi_i\}_{i=1}^{d}\subset C(M)$, then for any $a\in \mathrm { relint } (\mathcal{P}_{\Phi_{d}}(\mathcal{M}_f(M))),$ we have $\delta(a):=\sup\limits_{\mu\in \mathcal{M}_f^e(M)\cap   \mathcal{P}_{\Phi_{d}}^{-1}(a)}\delta_\mu>0$  and $[0,\delta(a))\subset \{\delta_\mu:\mu\in \mathcal{M}_f^e(M)\cap   \mathcal{P}_{\Phi_{d}}^{-1}(a)\}\subset  [0,\delta(a)].$
\end{maincorollary}

The relationships between the main statements in this paper may be represented by the following diagram:
\begin{figure}[htbp]
	\centering
	\includegraphics[width=14cm]{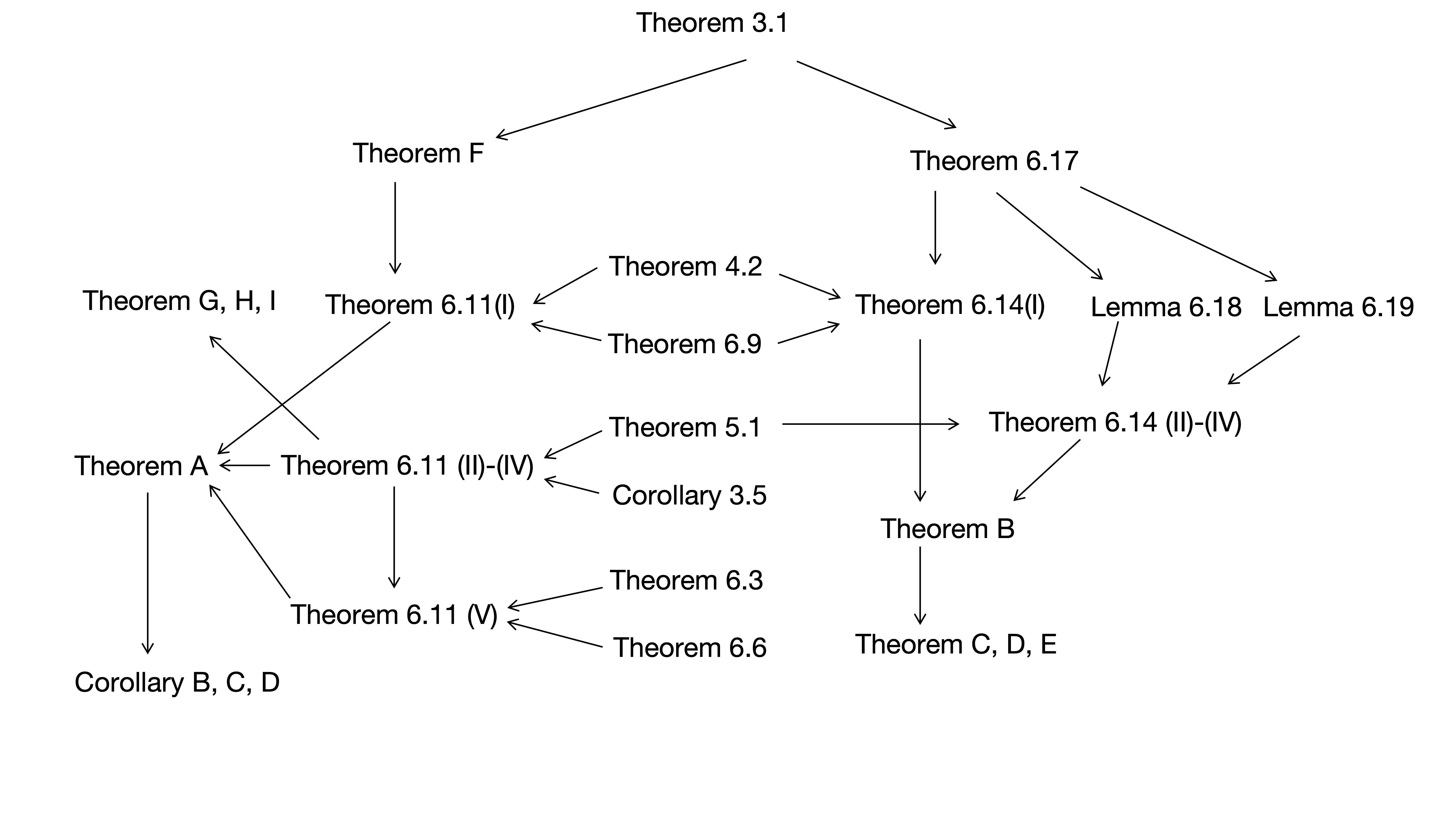}
	\caption{Relationships between the main statements}
\end{figure}

\textbf{Outline of the paper.}  
Section \ref{section-preli} is a review of deﬁnitions to make precise statements of the theorems and their proofs.
In Section \ref{section-entropy-dense} we prove that the 'multi-horseshoe' entropy-dense property   holds for transitive topologically Anosov system.
In Section \ref{Almost Additive} and \ref{section-almost2},
we give  abstract conditions on which the results of Theorem \ref{thm-continuous} hold in the more general context of asymptotically  additive sequences of continuous functions (see Theorem \ref{thm-Almost-Additive} and Theorem \ref{thm-Almost-Additive2}).
In Section \ref{section-thm} by 'multi-horseshoe' entropy-dense property and conditional variational principles  we show that the abstract conditions given in  Section \ref{Almost Additive} and \ref{section-almost2} are satisfied for transitive topologically Anosov system and the framework of Theorem \ref{thm-continuous-2}, and thus we obtain Theorem \ref{thm-continuous}
and  \ref{thm-continuous-2}. Then we give the proofs of Theorem \ref{maintheorem-skew}, \ref{maintheorem-cocycle} and \ref{maintheorem-robust} by using Theorem \ref{thm-continuous-2}. In Section \ref{section-inter-huasdorff}, we  give the proofs of Theorem \ref{thm-inter-huasdorff}-\ref{thm-first-return}. In Section \ref{section-Applications}, we  give the proofs of Corollary \ref{thm-pressure} and \ref{thm-inter-huasdorff-3}.

\section{Preliminaries}\label{section-preli}
\subsection{The space of probability measures}\label{section-space of measure}
Consider  a compact metric space $(X,d).$ The space of Borel probability measures on $X$ is denoted by $\mathcal{M}(X)$ and the set of continuous functions on $X$ by $C(X)$.   We endow $\varphi\in C(X)$ the norm $\|\varphi\|=\max\{|\varphi(x)|:x\in X\}$.
Let ${\{\varphi_{j}\}}_{j\in\mathbb{N}}$ be a dense subset of $C(X)$,   then
$$\rho(\xi,  \tau)=\sum_{j=1}^{\infty}\frac{|\int\varphi_{j}d\xi-\int\varphi_{j}d\tau|}{2^{j}\|\varphi_{j}\|}$$
defines a metric on $\mathcal{M}(X)$ for the $weak^{*}$ topology  \cite{Walters}.
For $\nu\in \mathcal{M}(X)$ and $r>0$,   we denote a ball in $\mathcal{M}(X)$ centered at $\nu$ with radius $r$ by
$\mathcal{B}(\nu,  r):=\{\mu\in \mathcal{M}(X):\rho(\nu,  \mu)<r\}.  $
One notices that
\begin{equation}\label{diameter-of-Borel-pro-meas}
	\rho(\xi,  \tau)\leq2~~\textrm{for any}~~\xi,  \tau\in \mathcal{M}(X).
\end{equation}
It is also well known that the natural imbedding $j:x\mapsto \delta_x$ is continuous.   Since $X$ is compact and $\mathcal{M}(X)$ is Hausdorff,   one sees that there is a homeomorphism between $X$ and its image $j(X)$.   Therefore,   without loss of generality
we will assume that
\begin{equation}\label{metric-on-X}
	d(x,  y)=\rho(\delta_x,  \delta_y).
\end{equation}
For $x\in X$ and $\varepsilon>0$,   we denote a ball in $X$ centered at $x$ with radius $\varepsilon$ by
$B(x,\varepsilon):=\{y\in X:d(x,y)<\varepsilon\}.$
A straight calculation using \eqref{diameter-of-Borel-pro-meas} and \eqref{metric-on-X} gives
\begin{Lem}\label{lem:prohorov}
	For any $\varepsilon > 0,\delta >0$, and $\{x_i\}_{i=0}^{n-1},\{y_i\}_{i=0}^{n-1}\subset X$, if $d(x_i,y_i)<\varepsilon$ holds for any $0\leq i\leq n-1$, then for any $J\subseteq \{0,1,\cdots,n-1\}$ with $\frac{n-|J|}{n}<\delta$, one has $\rho(\frac{1}{n}\sum_{i=0}^{n-1}\delta_{x_i},\frac{1}{|J|}\sum_{i\in J}\delta_{y_i})<\varepsilon+2\delta.$
\end{Lem}

\begin{Prop}\label{proposition-AD}
	Suppose that $(X,f)$ is a dynamical system.  If $\sharp \mathcal{M}_f(X)>1,$ then the set $\mathcal C^*=\{\varphi\in C(X):\mathrm{Int}(\mathcal{P}_\varphi(\mathcal{M}_f(X)))\neq\emptyset\}$ is an open and dense subset in $C(X)$.
\end{Prop}
\begin{proof} 
	Take $\mu\neq\nu\in  \mathcal{M}_f(X).$ Then there is $\varphi_0\in C(X)$ such that $\int \varphi_0d\mu\neq \int \varphi_0d\nu.$
	On one hand, we show $\mathcal C^*$  is  dense in $C(X)\setminus \mathcal C^*.$ Fix $\phi\in C(X)\setminus \mathcal C^*.$ Then $\int \phi d\mu= \int \phi d\nu.$ Take $\phi_n=\frac 1n \varphi_0 + \phi,\,n\geq 1.$ Then $\phi_n$ converges to $\phi$ in sup norm. By construction, it is easy to check that $\int \phi_n d\mu\neq \int \phi_n d\nu.$ That is, $\phi_n\in \mathcal C^*$.
	On the other hand, we prove that $\mathcal C^*$ is open. Fix $\phi\in \mathcal C^*$. Then   there must exist two different invariant measures $\mu_1\neq\mu_2\in  \mathcal{M}_f(X)$ such that $ \int \phi d\mu_1< \int \phi d\mu_2.$  By continuity of sup norm, we can take an open neighborhood of $\phi$, denoted by $U (\phi)$, such that for any $\varphi\in U (\phi),$  $\int \varphi d\mu_1< \int \varphi d\mu_2.$  This implies $U (\phi)\subset \mathcal C^*$ and thus 
	$\mathcal C^*$ is open.
\end{proof}

\subsection{Entropy and dimension}
\subsubsection{Topological entropy and metric entropy}
Now let us  recall the definition of topological entropy in \cite{Bowen1973} by Bowen.  Given a  dynamical system $(X,  f).$  For $x,  y\in X$ and $n\in\N$,   the Bowen distance between $x,  y$ is defined as
$d_n(x,  y):=\max\{d(f^i(x),  f^i(y)):i=0,  1,  \cdots,  n-1\}$
and the Bowen ball centered at $x$ with radius $\eps>0$ is defined as
$B_n(x,  \eps):=\{y\in X:d_n(x,  y)<\eps\}. $
Let $E\subseteq X$,   and $\mathcal {G}_{n}(E,  \sigma)$ be the collection of all finite or countable covers of $E$ by sets of the form $B_{u}(x,  \sigma)$ with $u\geq n$.   We set
$$C(E;t,  n,  \sigma,  f):=\inf_{\mathcal {C}\in \mathcal {G}_{n}(E,  \sigma)}\sum_{B_{u}(x,  \sigma)\in \mathcal {C}}e^{-tu} \,\,\,\text{   and }
C(E;t,  \sigma,  f):=\lim_{n\rightarrow\infty}C(E;t,  n,  \sigma,  f).  $$
Then we define
$\htop(E;\sigma,  f):=\inf\{t:C(E;t,  \sigma,  f)=0\}=\sup\{t:C(E;t,  \sigma,  f)=\infty\}.$
The \textit{Bowen topological entropy} of $E$ is
\begin{equation*}\label{definition-of-topological-entropy}
	\htop(f,  E):=\lim_{\sigma\rightarrow0} \htop(E;\sigma,  f).
\end{equation*}
For convenience, we denote $\htop(f)=\htop(f,X).$

Given $\mu\in\mathcal{M}_f(X).$ Let  $\xi=\{A_1,  \cdots,  A_n\}$ be a ﬁnite partition of measurable sets of $X$,   define
$H_\mu(\xi)=-\sum_{i=1}^n\mu(A_i)\log\mu(A_i).$  We denote by $\bigvee_{i=0}^{n-1}f^{-i}\xi$ the partition whose element is the set $\bigcap_{i=0}^{n-1}f^{-i}A_{j_i},  1\leq j_i\leq n$.   Then the limit
$h_\mu(f,  \xi)=\lim_{n\to\infty}\frac1n H_\mu\left(\bigvee_{i=0}^{n-1}f^{-i}\xi\right)$ exists
and we define the \textit{metric entropy} of $\mu$ as
$$h_{\mu}(f):=\sup\{h_\mu(f,  \xi):\xi~\textrm{is a finite measurable partition of X}\}. $$

\subsubsection{Hausdorff dimension}
Now we recall the definitions of Hausdorff dimension of subsets. Given a subset $Z \subset X$, for any $s \geq 0$, let
$$
\mathcal{H}_\delta^s(Z)=\inf \left\{\sum_{i=1}^{\infty}\left(\operatorname{diam} U_i\right)^s:\left\{U_i\right\}_{i \geq 1} \text { is a cover of } Z \text { with } \operatorname{diam} U_i \leq \delta, \text { for all } i \geq 1\right\}
$$
and
$
\mathcal{H}^s(Z)=\lim _{\delta \rightarrow 0} \mathcal{H}_\delta^s(Z).
$
The above limit exists, though the limit may be infinity. We call $\mathcal{H}^s(Z)$ the $s$-Hausdorff measure of $Z$.
	The following jump-up value of $\mathcal{H}^s(Z)$
	$$
	\operatorname{dim}_H Z=\inf \left\{s: \mathcal{H}^s(Z)=0\right\}=\sup \left\{s: \mathcal{H}^s(Z)=\infty\right\}
	$$
	is called the \textit{Hausdorff dimension} of $Z$.

\subsection{Transitive, mixing, expansive and shadowing property}
Consider a  dynamical system $(X,  f).$
If for every pair of non-empty open sets $U,$ $V$ there is an integer $n$ such that $f^n(U)\cap V\neq \emptyset$ then we call $(X,  f)$ \textit{transitive}.
Furthermore,   if for every pair of non-empty open sets $U,$ $V$ there exists an integer $N$ such that $f^n(U)\cap V\neq \emptyset$ for every $n>N$,   then we call $(X,  f)$ \textit{mixing}. 
When $f:X\to X$ is a homeomorphism  of a compact metric space,   we say that $(X,  f)$ is \emph{expansive} if there exists a constant $c>0$ such that for any $x\neq  y\in X$,   $d(f^i(x),  f^i(y))> c$ for some $i\in\Z$.   We call $c$ the expansive constant.
When $f:X\to X$ is a homeomorphism  of a compact metric space, we say that a subset $Y$ of $X$ is $f$-invariant if $f(Y)= Y.$ 
If $Y$ is a closed $f$-invariant subset of $X,$ then $(Y,f)$ also is a dynamical system. We will call it a subsystem of $(X,f).$ 

A finite sequence $\C=\langle x_1,  \cdots,  x_l\rangle,  l\in\N$ is called a \emph{chain}.   Furthermore,   if $d(f(x_i),  x_{i+1})<\eps,  1\leq i\leq l-1$,   we call $\C$ an \textit{$\eps$-chain with length $l.$}
For any $m\in\N$,   if there are $m$ $\eps$-chains $\mathfrak{C}_i=\langle x_{i,  1},  \cdots,  x_{i,  l_i}\rangle$,   $l_i\in\N,  1\leq i\leq m$ satisfying that $d(f(x_{i,  l_i}),x_{i+1,  1})<\eps,   1\leq i\leq m-1$,   then we can concatenate $\mathfrak{C}_i$s to constitute a new $\eps$-chain
$\langle x_{1,  1},  \cdots,  x_{1,  l_1},  x_{2,  1},  \cdots,  x_{2,  l_2},  \cdots,  x_{m,  1},  \cdots,  x_{m,  l_m}\rangle$
which we denote by $\mathfrak{C}_1\mathfrak{C}_2\cdots\mathfrak{C}_m$.   

\begin{Def}\label{def-shadowing}
	Suppose $f:X\to X$ is a homeomorphism  of a compact metric space. For any $\delta>0$,   a sequence $\{x_n\}_{n\in \Z}$ is called a \textit{$\delta$-pseudo-orbit} if
	$d(f(x_n),  x_{n+1})<\delta~\textrm{for any}~n\in\Z.$ $\{x_n\}_{n\in \Z}$ is \textit{$\eps$-shadowed} by some $y\in X$ if
	$d(f^n(y),  x_n)<\eps~\textrm{for any}~n\in\Z.$
	We say that $(X,  f)$ has the \textit{shadowing property} if for any $\eps>0$,   there exists  $\delta>0$ such that any $\delta$-pseudo-orbit is $\eps$-shadowed by some point in $X$.
\end{Def}
\begin{Lem}\cite[Theorem 2.3.3]{AH}\label{lem-AQ}
	Suppose $f:X\to X$ is a homeomorphism  of a compact metric space. Let $k>0$ be an integer. Then $(X,f)$ has shadowing property if and only if so does
	$(X,f^{k})$.
\end{Lem}
Given two  dynamical systems $(X, f)$ and $(Y, g)$, if $\pi:X \rightarrow Y$ is a homeomorphism such that $\pi \circ f=g \circ \pi$, then we say $\pi$ is a \textit{conjugation}, and $(X, f)$ conjugates to $(Y, g).$ 
Then if $(X, f)$ is a transitive topologically Anosov system and $(X, f)$ conjugates to $(Y, g),$ then $(Y, g)$ is also  a transitive topologically Anosov system.

Now, we recall the definition of entropy-dense property.
\begin{Def}
	We say $(X,  f)$ satisfies the {\it entropy-dense property},  if for any $\mu\in \mathcal{M}_f(X)$,   for any neighborhood $G$ of $\mu$ in $\mathcal{M}(X)$,
	and for any $\eta>0$,   there exists a closed $f$-invariant set $\Lambda_{\mu}\subseteq X $ such that $\mathcal{M}_f( \Lambda_{\mu})\subseteq G$ and $\htop(f,  \Lambda_{\mu})>h_{\mu}(f)-\eta$. By classical variational principle,
	it is equivalent that for any neighborhood $G$ of $\mu$ in $\mathcal{M}(X)$,   and for any $\eta>0$,   there exists a $\nu\in \mathcal{M}_f^e(X)$ such that $h_{\nu}(f)>h_{\mu}(f)-\eta$ and $\mathcal{M}_f( S_{\nu})\subseteq G$.
\end{Def}
For systems with the approximate product property,  Pfister and Sullivan had obtained the entropy-dense properties by \cite[Proposition 2.3]{PS2005}.
Note that if a dynamical system is transitivie and has the shadowing property, then it has  approximate product property by their definitions. Then we have 
\begin{Prop}\label{prop-entropy-dense-for-shadowing}
	Suppose that $(X,  f)$  is  transitive and satisfies the shadowing property.   Then $(X,  f)$ has the entropy-dense property.
\end{Prop}

In this paper, we assume that $(X,f)$ is non-degenerate (i.e. is not reduced to a single periodic orbit). Following the argument of \cite[Proposition 21.6]{DGS}, if $(X,  f)$  is  transitive and topologically Anosov, then $\htop(f)>0.$
From \cite[Corollary C]{LiOpro2018}, if $(X,  f)$  is  transitive and topologically Anosov, then  it has ergodic measures of arbitrary intermediate metric entropies, that is $[0,h_{top}(f))\subset\{h_{\mu}(f):\mu\in\mathcal{M}_f^e(X)\}.$ This implies $\{h_{\mu}(f):\mu\in\mathcal{M}_f(X)\}=\{h_{\mu}(f):\mu\in\mathcal{M}_f^e(X)\}$ by the variational principle of the topological entropy and the ergodic decomposition theorem.
\begin{Lem}\label{Lemma-inter-entropy}
	Suppose that $(X,  f)$  is  transitive and topologically Anosov. Then we have  $\htop(f)>0$ and  $[0,h_{top}(f)]=\{h_{\mu}(f):\mu\in\mathcal{M}_f(X)\}=\{h_{\mu}(f):\mu\in\mathcal{M}_f^e(X)\}.$
\end{Lem}

\subsection{Lyapunov exponents, hyperbolic sets and conformality }\label{subsection-LayHyp}
\subsubsection{Lyapunov exponents}
Let $f: M \rightarrow M$ be a $C^1$ diffeomorphism on a  compact Riemannian manifold.  For $x \in M$ and $v \in T_x M$, the Lyapunov exponent of $v$ at $x$ is the limit
$
\lambda(x, v)=\lim _{n \rightarrow \infty} \frac{1}{n} \log \left\|D_x f^n(v)\right\|
$
whenever the limit exists. Given an invariant measure $\mu \in\mathcal{M}_f(M)$, by the Oseledec multiplicative ergodic theorem \cite{Oseledec1968}, for $\mu$-almost every $x$, every vector $v \in T_x M$ has a Lyapunov exponent, and they can be denoted by $$
\chi_1(x) \geq \chi_2(x) \geq \cdots \geq \chi_{\dim M}(x).
$$ If $\mu$ is ergodic, since the Lyapunov exponents are $f$-invariant, we write the Lyapunov exponents as $$
\chi_1(\mu) \geq \chi_2(\mu) \geq \cdots \geq \chi_{\dim M}(\mu).$$  

An  $f$-invariant ergodic  measure $\mu$ is said to be \textit{hyperbolic} if it has positive and negative but no zero Lyapunov exponents.
The following is Katok’s Horseshoe Theorem:
\begin{Thm}\cite[Theorem 2.12]{BCS2022} and \cite[Theorem 3.17]{SunTian2015}\label{Thm-Katok’s Horseshoe}
	Let $f: M \rightarrow M$ be a $C^1$ diffeomorphism on a  compact Riemannian manifold, and $\nu$ an $f$-invariant ergodic hyperbolic measure with $h_\nu(f)>0$. Assume that $f$ is $C^{1+\alpha}$ for some $0<\alpha<1$ or the Oseledec splitting of $\nu$ is dominated, then for any $\varepsilon>0$ there exists a compact set $\Lambda_{\varepsilon}\subset M$ such that the following properties hold:
	\begin{enumerate}
		\item [(1)] $\Lambda_\varepsilon$ is a transitive locally maximal hyperbolic set.
		\item [(2)] $h_{\nu}(f)-\varepsilon<\htop(f,\Lambda_{\varepsilon})<h_{\nu}(f)+\varepsilon$.
		\item [(3)] $|\chi_i(\mu)-\chi_i(\nu)|<\varepsilon$ for each $1\leq i\leq \dim M$ and $\mu\in\mathcal{M}_{f}^e(\Lambda_{\varepsilon}).$
	\end{enumerate}
\end{Thm}
\begin{Rem}
	In the item(2) of Theorem \ref{Thm-Katok’s Horseshoe}, the original result does not give the inequality of the right-hand side. However, only a slight modiﬁcation can give the upper bound of $\htop(f,\Lambda_{\varepsilon}).$
\end{Rem}

A compact $f$-invariant set $\Lambda\subset M$  is said to be  {\it average conformal}  if  each $\mu \in \mathcal{M}_f^e(\Lambda)$  has only two Lyapunov exponents  $\chi_s(\mu)<0<\chi_u(\mu)$.
Let  $\Lambda\subset M$ be an average conformal compact $f$-invariant set and $\mu\in\mathcal{M}_f^e(\Lambda)$ be a hyperbolic ergodic measure.  If $f$ is $C^{1+\alpha}$ for some $0<\alpha<1$ or the Oseledec splitting of $\nu$ is dominated, then 
\begin{equation}\label{equation-WC}
	\dim_H \mu=\frac{h_\mu(f)}{\chi_u(\mu)}-\frac{h_\mu(f)}{\chi_s(\mu)}.
\end{equation}
See \cite{WC2016} for the detailed proofs, which can be viewed as an extension of Young’s results in \cite{Young1982} to the case of an average conformal setting.

\subsubsection{Hyperbolic sets}
For each $x \in M$, the  quantity
$
\left\|D_x f\right\|=\sup _{0 \neq v \in T_x M} \frac{\left\|D_x f(v)\right\|}{\|v\|}
$
is called the maximal norm of the differentiable operator $D_x f: T_x M \rightarrow T_{f(x)} M$, where $\|\cdot\|$ is the norm induced by the Riemannian metric on $M$. An $f$-invariant subset $\Lambda \subset M$ is called a {\it locally maximal hyperbolic set} if $\Lambda$ is compact, there exists an open neighborhood $U$ such that $\Lambda=\bigcap_{n \in \mathbb{Z}} f^n U$, and a continuous splitting of the tangent bundle $T_x M=E_x^s \oplus E_x^u$, and constants $0<\lambda<1, C>0$ such that for every $x \in \Lambda$ :
\begin{enumerate}
	\item $D_x f\left(E_x^u\right)=E_{f(x)}^u, D_x f\left(E_x^s\right)=E_{f(x)}^s$;
	\item for every $n \in \mathbb{N}$, one has $\left\|D_x f^{-n}(v)\right\| \leqslant C \lambda^n\|v\|$ for all $v \in E_x^u$, and $\left\|D_x f^n(v)\right\| \leqslant$ $C \lambda^n\|v\|$ for all $v \in E_x^s$.
\end{enumerate}
It's known that a locally maximal hyperbolic set is expansive by \cite[Corollary 6.4.10]{KatHas} and has shadowing property  by \cite[Theorem 18.1.2]{KatHas}. 
So we have the following.
\begin{Lem}\label{LH}
	Every system restricted on a  locally maximal hyperbolic set is topologically Anosov.
\end{Lem}
A $C^1$ diffeomorphism $f:M\to M$ is said to be an \textit{Anosov diffeomorphism} if $M$ is a hyperbolic set. It's clear that $M$ is locally maximal. So every Anosov diffeomorphism is topologically Anosov.
By spectral decomposition, every transitive Anosov diffeomorphism on a compact Riemannian manifold is mixing \cite[Corollary 18.3.5]{KatHas}. 
\begin{Lem}\label{LH2}
	Let $f:M\to M$ be a transitive Anosov diffeomorphism on a compact Riemannian manifold. Then $M$ is a mixing locally maximal hyperbolic set.
\end{Lem}

When a hyperbolic set $\Lambda\subset M$ is average conformal, then  for each $\mu \in \mathcal{M}_f^e(M)$, one has $\chi_1(\mu)=\chi_2(\mu)=\cdots=\chi_{d_u}(\mu)>0$ and $\chi_{d_u+1}(\mu)=\chi_{d_u+2}(\mu)=\cdots=\chi_{\dim M}(\mu)<0$, where $d_u=\operatorname{dim} E^u$ and $d_s=\operatorname{dim} E^s=\dim M-d_u$. In other words, $\Lambda$ has only two Lyapunov exponents  $\chi_s(\nu)<0<\chi_u(\nu)$ with respect to each $\nu \in \mathcal{M}_f^e(M)$.

Let $\Lambda\subset M$ be a compact $f$-invariant set.
A $D f$ invariant subbundle $G \subset T_\Lambda M$ is called to be {\it quasi-conformal} if for any $\varepsilon>0$, there exists $C_\varepsilon>0$ such that for any $x \in \Lambda$ and $n \geq 1$,
$$
C_\varepsilon^{-1} e^{-n \varepsilon} \leq \frac{\left\|\left.D_x f^n\right|_{G(x)}\right\|}{m\left(\left.D_x f^n\right|_{G(x)}\right)} \leq C_\varepsilon e^{n \varepsilon},
$$
where $m\left(D_x f|_{G(x)}\right)=\inf _{0 \neq v \in G(x)} \frac{\left\|D_x f(v)\right\|}{\|v\|}.$
Quasi-conformal condition implies that for any $x \in \Lambda$,
$$
\begin{aligned}
	\limsup _{n \rightarrow+\infty} \frac{1}{n} \log \left\|\left.D_x f^{ \pm n}\right|_{G(x)}\right\|&=\limsup _{n \rightarrow+\infty} \frac{1}{n} \log m\left(\left.D_x f^{ \pm n}\right|_{G(x)}\right), \\
	\liminf _{n \rightarrow+\infty} \frac{1}{n} \log\left\|\left.D_x f^{ \pm n}\right|_{G(x)}\right\|&=\liminf _{n \rightarrow+\infty} \frac{1}{n} \log m\left(\left.D_x f^{ \pm n}\right|_{G(x)}\right) .
\end{aligned}
$$
A hyperbolic set $\Lambda\subset M$  is said to be {\it quasi-conformal}  if $E^u$ and $E^s$ are both quasi-conformal. It's clear that every quasi-conformal hyperbolic set is average conformal.

\subsubsection{Lyapunov exponents of hyperbolic sets}
Let $f: M \mapsto M$ be a $C^1$ diffeomorphism on a compact Riemannian manifold $M.$
Given  $\mu \in\mathcal{M}_f(M)$,  for $\mu$ a.e. $x$,  denote by $
\chi_1(x) \geq \chi_2(x) \geq \cdots \geq \chi_{\dim M}(x)
$
the Lyapunov exponents at $x.$ Denote
$
\chi_i^{+}(x)=\max \left\{\chi_i(x), 0\right\}, \chi_i^{-}(x)=\min \left\{\chi_i(x), 0\right\}.
$
Now assume that $\Lambda\subset M$ be a hyperbolic set. Define $\psi^u(x)=\log|\det D_xf|_{E_x^{u}}|$ and $\psi^s(x)=\log|\det D_xf|_{E_x^{s}}|$ for any $x\in \Lambda.$
From \cite[Lemma 3.5]{SunTian2012}, one has $$\lim\limits_{n\to\infty}\frac{1}{n}\log|\det D_xf^n|_{E^u_x}|=\sum_{i=1}^{\operatorname{dim} M} \chi_i^{+}(x),\lim\limits_{n\to\infty}\frac{1}{n}\log|\det D_xf^n|_{E^s_x}|=\sum_{i=1}^{\operatorname{dim} M} \chi_i^{-}(x)$$
for $\mu$-a.e. $x \in \Lambda$ and any $\mu\in\mathcal{M}_f(\Lambda).$
Since $\Lambda$ is hyperbolic,  $\psi^u$ and $\psi^s$ are both continuous.  Thus by Birkhoff's ergodic theorem,
\begin{equation}\label{equation-Birk0}
	\begin{split}
		\int \psi^u d\mu&=\int \lim\limits_{n\to\infty}\frac{1}{n}\log|\det Df^n|_{E^u_x}| d\mu=\int \sum_{i=1}^{\operatorname{dim} M} \chi_i^{+}(x) d\mu.\\
		\int \psi^s d\mu&=\int \lim\limits_{n\to\infty}\frac{1}{n}\log|\det Df^n|_{E^s_x}| d\mu=\int \sum_{i=1}^{\operatorname{dim} M} \chi_i^{-}(x) d\mu.
	\end{split}
\end{equation}
When $\Lambda$ is  average conformal, it has only two Lyapunov exponents $\chi_u(\mu)>0$ and $\chi_s(\mu)<0$ with respect to each $\mu \in \mathcal{M}_f^e(\Lambda)$. Then for any $\mu \in \mathcal{M}_f^e(\Lambda)$, we have
\begin{equation}\label{equation-Birk}
	\int \psi^u d\mu=d_u\chi_u(\mu), \int \psi^s d\mu=d_s\chi_s(\mu),
\end{equation}
where $d_u=\operatorname{dim} E^u$ and $d_s=\operatorname{dim} E^s=\dim M-d_u$.

\subsection{Subshifts of finite type}
 Let $k$ be a fixed natural number and let $C=\{0, 1, \ldots, k-1\}$. Put the discrete topology on $C.$  Consider the two-sided full symbolic space $\Sigma=\prod_{-\infty}^{\infty} C$, equipped with the product topology, and the shift homeomorphism $\sigma: \Sigma \to \Sigma$ defined by $(\sigma(w))_{n}=w_{n+1}$, where $w =\left(w_{n}\right)_{n=-\infty}^{\infty}.$ $(\Sigma,\sigma)$ is called a two-sided full shift. A metric on $\Sigma$ is defined by $d(x, y)=2^{-m}$ if $m$ is the largest natural number with $x_{n}=y_{n}$ for any $|n|<m$, and $d(x, y)=1$ if $x_{0} \neq y_{0}.$ If $X$ is a closed subset of $\Sigma$ with $\sigma (X)=X,$ then $(X,\sigma)$ is called a subshift.   A subshift $(X,\sigma)$ is said to be \textit{of finite type}, if there exists some natural number $N$ and a collection of blocks of length $N+1$ with the property that $x=\left(x_{n}\right)_{n=-\infty}^{\infty} \in X$ if and only if each block $\left(x_{i}, \ldots, x_{i+N}\right)$ in $x$ of length $N+1$ is one of the prescribed blocks. 
 Recall from \cite{Walters2} a subshift satisfies shadowing property if and only if it is a subshift of finite type. As a subsystem of two-sided full shift, it is expansive. So we have the following.
 \begin{Lem}\label{SFT}
 	Every two-sided subshift of finite type is topologically Anosov.
 \end{Lem}

 \subsection{Convex set and its properties}\label{sec-convex}
 Now we recall some properties of  convex set. Readers can refer to \cite{BoyVan}.
 A subset $C$ of $\mathbb{R}^n$ is \textit{convex}, if $\{\theta x+(1-\theta)y:\theta\in[0,1]\}\subset C$ for any $x,y\in C.$
 The following properties are easy to verify. 
 \begin{Prop}\label{Prop-convex}
 	(1) If $C\subset \mathbb{R}^n$ is a convex set, then $\mathrm{Int}(C)$ is convex.
 	
 	(2) If $C\subset \mathbb{R}^n$ is a convex set with $\mathrm{Int}(C)\neq\emptyset$, then $\overline{C}=\overline{\mathrm{Int}(C)}$.
 \end{Prop}
Let $n\in\mathbb{N}$ and $C$ be a nonempty convex subset of $\mathbb{R}^n$. The set of all affine combinations of points in $C$ is called the \textit{affine hull} of $C$, and denoted $\mathrm{aff}(C)$,
$$
\mathrm { aff } (C)=\left\{\sum_{i=1}^{k}\theta_i x_i : k\geq 1,  x_i \in C, \theta_i\in\mathbb{R} \text{ for each } 1\leq i\leq k, \text{ and } \sum_{i=1}^{k}\theta_i=1\right\} .
$$
We deﬁne the \textit{affine dimension} of $C$ as the dimension of its aﬃne hull, denote by $\dim_{aff}(C)$.  Denote the \textit{relative interior} of the set $C$, denoted  $\mathrm{relint}(C)$, as its interior relative to  $\mathrm { aff } (C)$ :
$$
\mathrm { relint } (C)=\{x \in C : B(x, r) \cap \mathrm { aff } (C) \subseteq C \text { for some } r>0\}.
$$
Then $\mathrm { relint } (C)$ is nonempty and convex, and
\begin{equation}\label{equ-close-relint}
	\overline{C}=\overline{\mathrm{relint}(C)}.
\end{equation}
In particular, if $\sharp C=1,$ then we have $\mathrm { aff } (C)=\mathrm { relint } (C)=C.$ When $\dim_{aff}(C)=n,$ we have $\mathrm { relint } (C)=\mathrm{Int}(C).$ If $\dim_{aff}(C)<n,$ then $\mathrm{Int}(C)=\emptyset.$ 
\begin{Rem}
	Interior and relative interior are very different. For example, the interior of a point in an at least one-dimensional ambient space is empty, but its relative interior is the point itself.
	The interior of a disc in an at least three-dimensional ambient space is empty, but its relative interior is the same disc without its circular edge.
\end{Rem}

We assume $1\leq \dim_{aff}(C)\leq n-1.$ Take $x_0=(x_0^1,\dots,x_0^n)\in \mathrm { aff } (C),$ then $V:=\mathrm { aff } (C)-x_0=\{x-x_0:x\in \mathrm { aff } (C)\}$ is a subspace of $\mathbb{R}^n$ with dimension $\dim_{aff}(C).$
Thus there is a matrix $A \in \mathbb{R}^{(n-\dim_{aff}(C) )}\times \mathbb{R}^n$ such that  $V=\left\{\boldsymbol{x} \in \mathbb{R}^n : A \boldsymbol{x}=0\right\}.$ Denote $x=(x_1,\dots,x_n).$ There exist $I_{aff}\subset\{1,2,\dots,n\}$ with $\sharp I_{aff}=\dim_{aff}(C)$  and $\{c_{j,i}\}_{j\in [1,n]\setminus I_{aff},i\in I_{aff}}\subset \mathbb{R}$ such that $$V =\{x\in \mathbb{R}^n:x_i\in \mathbb{R} \text{ for each }i\in I_{aff}, x_j=\sum_{i\in I_{aff}}c_{j,i}x_i \text{ for each }j\in [1,n]\setminus I_{aff}\}.$$
Denote $c_0^j=-\sum_{i\in I_{aff}}c_{j,i}x_0^i +x_0^j.$  Then 
\begin{equation*}
	\begin{split}
		\mathrm{aff}(C)&=V+x_0 \\
		&=\{x+x_0\in \mathbb{R}^n:x_i\in \mathbb{R} \text{ for each }i\in I_{aff}, x_j=\sum_{i\in I_{aff}}c_{j,i}x_i \text{ for each }j\in [1,n]\setminus I_{aff}\}\\
		&=\{y\in \mathbb{R}^n:y_i\in \mathbb{R} \text{ for each }i\in I_{aff}, y_j=\sum_{i\in I_{aff}}c_{j,i}(y_i-x_0^i) +x_0^j\text{ for each }j\in [1,n]\setminus I_{aff}\}\\
		&=\{y\in \mathbb{R}^n:y_i\in \mathbb{R} \text{ for each }i\in I_{aff}, y_j=\sum_{i\in I_{aff}}c_{j,i}y_i +c_0^j\text{ for each }j\in [1,n]\setminus I_{aff}\}.
	\end{split}
\end{equation*}
Define a map from $\mathbb{R}^n$ to $\mathbb{R}^{\dim_{aff}(C)}$ as following: $$\pi_{aff}:(y_1,\dots,y_n)\to (y_i)_{i\in I_{aff}}.$$ 
Then $\pi_{aff}$ is affine and $\pi_{aff}$ is a homeomorphism from $\mathrm{aff}(C)$ to its image. So $\pi_{aff}(C)$ is a  nonempty convex subset of $\mathbb{R}^{\dim_{aff}(C)}$, and $\pi_{aff}(\mathrm { relint } (C))=\mathrm{Int}(\pi_{aff}(C))$.

Let $m,n\in\mathbb{N},$ and $C$ be a convex subset of $\mathbb{R}^{m+n}$. For $x \in \mathbb{R}^m$, let
$$
C_x=\{y\in\mathbb{R}^n :(x, y) \in C\},
$$
and let
$$
C_m=\{x \in\mathbb{R}^m: C_x \neq \emptyset\} .
$$
Then
\begin{equation}\label{equ-product-inter}
	\mathrm { relint } (C)=\{(x, y) : x \in \mathrm { relint } (C_m), y \in \mathrm { relint } (C_x)\}.
\end{equation}

\section{'Multi-horseshoe' entropy-dense property}\label{section-entropy-dense}

Now  we prove  the 'multi-horseshoe' dense property  holds for transitive topologically Anosov systems.

\begin{Thm}\label{Mainlemma-convex-by-horseshoe}
	Suppose $(X,  f)$ is topologically Anosov and transitive.   Then for any positive integer $m,$ any $f$-invariant measures $\{\mu_i\}_{i=1}^m\subseteq \mathcal{M}_f(X),$ any $x\in X$ and any $\eta,  \zeta,\eps>0$,   there exist compact invariant subsets $\Lambda_i\subseteq\Lambda\subsetneq X$ such that for each $1\leq i\leq m$
	\begin{enumerate}
		\item $(\Lambda_i,f)$ and $(\Lambda,f)$ conjugate to transitive two-sided subshifts of ﬁnite type.
		\item $\htop(f,  \Lambda_i)>h_{\mu_i}(f)-\eta$.
		\item $d_H(K,  \mathcal{M}_f(\Lambda))<\zeta$,   $d_H(\mu_i,  \mathcal{M}_f(\Lambda_i))<\zeta$, where $K=\cov\{\mu_i\}_{i=1}^m.$
		\item There is a positive integer $L$ such that for any $z$ in $\Lambda_i$ or $\Lambda$ one has  $f^{j+tL}(z) \in B(x,\eps)$ for some $0\leq j\leq L-1$ and any $t\in\Z$.
	\end{enumerate}
\end{Thm}
\begin{Rem}
	The item 4 of Theorem \ref{Mainlemma-convex-by-horseshoe} will be  used in Theorem \ref{def-strong-basic-2} and another paper.
\end{Rem}
\subsection{Two lemmas}

\begin{Prop}\label{prop-transitive-shadowing-for-f-n}
	Suppose $f:X\to X$ is a homeomorphism of a compact metric space. Consider a compact set $\Delta\subseteq X$ which satisfies $f^n(\Delta)=\Delta$ for some $n\in\N$ and let $\Lambda=\bigcup_{i=0}^{n-1}f^i(\Delta).$ If $f^i(\Delta)\cap f^j(\Delta)=\emptyset$ for any $0\leq i<j\leq n-1,$ then
	\begin{description}
		\item[(1)] if $(\Delta,  f^n)$ is expansive,   then $(\Lambda,  f)$ is expansive;
		\item[(2)] if $(\Delta,  f^n)$ is transitive,   then $(\Lambda,  f)$ is  transitive;
		\item[(3)] if $(\Delta,  f^n)$ has the shadowing property,   then $(\Lambda,  f)$ also has the shadowing property. 
	\end{description}
\end{Prop}
\begin{proof}
	Item (1) and (2) come directly from the uniform continuity of $f, \cdots, f^{m-1}.$ Since $(\Delta,  f^n)$ has the shadowing property and $f:X\to X$ is a homeomorphism, then $(f^i(\Delta),  f^n)$ has the shadowing property for any $0\leq i\leq n-1.$ Combining with $f^i(\Delta)\cap f^j(\Delta)=\emptyset$ for any $0\leq i<j\leq n-1,$  $(\Lambda,  f^n)$ also has the shadowing property.  So $(\Lambda,  f)$  has the shadowing property by Lemma \ref{lem-AQ}. 
\end{proof}

For $\eps>0$ and $n\in\N$,   two points $x$ and $y$ are $(n,  \eps)$-separated if
$d_n(x,y)>\eps.$
A subset $E$ is $(n,  \eps)$-separated if any pair of different points of $E$ are $(n,  \eps)$-separated. For $x\in X$,   we define the empirical measure of $x$ as
$
	\mathcal{E}_{n}(x):=\frac{1}{n}\sum_{j=0}^{n-1}\delta_{f^{j}(x)},
$
where $\delta_{x}$ is the Dirac mass at $x$. 
Let $F\subseteq \mathcal{M}(X)$  be a neighbourhood of $\nu \in \mathcal{M}_f(X)$, define $X_{n,F}:=\{x\in X:\mathcal{E}_{n}(x)\in F\}.$
For $k\in\N$,   let $P_k(f):=\{x\in X:f^k(x)=x\}$.

\begin{Lem}\label{Maincor-mu-gamma-n-expansive}\cite[Corollary 4.13]{DT}
	Suppose $(X,  f)$ is topologically Anosov and transitive.
	Then for any $\eta>0$,   there exists an $\eps^*_1=\eps^*_1(\eta)>0$ such that for any $\mu\in \mathcal{M}_f(X)$ and its neighborhood $F_{\mu}$, there exists an $0<\eps^*_2=\eps^*_2(\eta,\mu,F_\mu)<\eps^*_1$ such that  for any $x\in X$,   for any $0<\eps\leq\eps^*_2$,   for any $N\in\N$,   there exist an $n=n(\eta,\mu,F_\mu,  \eps,x)\geq N$ such that for any $p\in\N$,   there exists an $(pn,  \frac{\eps^*_1}{3})$-separated set $\Gamma_{pn}$ so that
	\begin{description}
		\item[(1)]\label{lem-B-aACor} $\Gamma_{pn}\subseteq X_{pn,  F_{\mu}}\cap B(x,  \eps)\cap P_{pn}(f)$;
		\item[(2)]\label{lem-B-bACor} $\frac{\log |\Gamma_{pn}|}{pn}>h_{\mu}(f)-\eta$.
	\end{description}	
\end{Lem}

\begin{Cor}\label{Corollary-zero-metric-entropy}
	Suppose $(X,  f)$ is topologically Anosov and transitive.
	Then we have 
	\begin{description}
		\item [(1)] $\{\mu\in \mathcal{M}_f(X):h_{\mu}(f)=0\}$ is dense in $\mathcal{M}_f(X).$
		\item [(2)] $\{\mu\in \mathcal{M}_f(X):h_{\mu}(f)>0\}$ is dense in $\mathcal{M}_f(X).$
		\item [(3)] there is an invariant measure $\nu$ with full support, that is, $S_\nu=X$.
		\item [(4)] $\mathcal{M}_f^e(X)$ is a dense $G_\delta$ subset of $\mathcal{M}_f(X)$.
	\end{description}
\end{Cor}
\begin{proof}
	From Lemma \ref{Maincor-mu-gamma-n-expansive}, the set of periodic measures is dense in $\mathcal{M}_f(X)$.  Since measure supported on periodic points has zero metric entropy, we obtain item(1).
	From Lemma \ref{Maincor-mu-gamma-n-expansive}, 
	the set of periodic points is dense in $X.$ Then by \cite[Proposition 6.5]{T16}, there is an invariant measure with full support.  We obtain item(3).
	By Lemma \ref{Lemma-inter-entropy}, there is $\mu_+ \in \mathcal{M}_f(X)$ with $h_{\mu_+}(f)>0.$ Given $\omega \in \mathcal{M}_f(X)$ with $h_\omega(f)=0,$ denote $\omega_\theta=\theta\omega+(1-\theta)\mu_+$ for any $\theta\in(0,1).$ Then $h_{\omega_\theta}(f)>0$ and $\lim\limits_{\theta\to0}\omega_\theta=\omega.$ So we obtain item(2). Finally, since the set of periodic measures is dense in $\mathcal{M}_f(X)$, then $\mathcal{M}_f^e(X)$ is a dense  subset of $\mathcal{M}_f(X)$. So by  \cite[Proposition 5.7]{DGM2019}, $\mathcal{M}_f^e(X)$ is a dense $G_\delta$ subset of $\mathcal{M}_f(X)$.
\end{proof}
\begin{Cor}\label{Corollary-zero-metric-entropy-2}
	Suppose $(X,  f)$ is topologically Anosov and transitive. Let $\varphi$ be a continuous function on $X$. 
	Then  $\mathrm{Int}(\mathcal{P}_\varphi(\mathcal{M}_f(X)))\subset \mathcal{P}_\varphi(\{\mu\in \mathcal{M}_f(X):h_{\mu}(f)=0\}).$
\end{Cor}
\begin{proof}
	For any $a\in \mathrm{Int}(\mathcal{P}_\varphi(\mathcal{M}_f(X))),$ there are $\mu_1,\mu_2\in \mathcal{M}_f(X)$ such that $ \mathcal{P}_\varphi(\mu_1)<a< \mathcal{P}_\varphi(\mu_2).$ By Corollary \ref{Corollary-zero-metric-entropy}(1), there are $\nu_1,\nu_2\in \{\mu\in \mathcal{M}_f(X):h_{\mu}(f)=0\}$ satisfying $\rho(\mu_1,\nu_1)$ and $\rho(\mu_2,\nu_2)$ small enough such that $ \mathcal{P}_\varphi(\nu_1)<a< \mathcal{P}_\varphi(\nu_2).$ Choose $\theta\in(0,1)$ such that $\nu=\theta\nu_1+(1-\theta)\nu_2$ satisfying $\mathcal{P}_\varphi(\nu)=a.$ Then $h_\nu(f)=\theta h_{\nu_1}(f)+(1-\theta)h_{\nu_2}(f)=0.$
\end{proof}

\subsection{Proof of Theorem \ref{Mainlemma-convex-by-horseshoe}}

Fix $m>0,$ $K=\cov\{\nu_i\}_{i=1}^m\subseteq \mathcal{M}_f(X),$ $x\in X$ and $\eta_0,\zeta_0,\eps_0>0.$ Let $\rho_0=\min\{\rho(\nu_i,  \nu_j):1\leq i<j\leq m\}.$ Then $  K\cap \mathcal{M}_f^e(X)$ is empty or finite (less than $m$). By Proposition \ref{prop-entropy-dense-for-shadowing},   $(X,  f)$ has the entropy-dense property,   so there are infinitely many ergodic measures on $X$.
This implies that $K\neq \mathcal{M}_f(X)$ and thus $d_H(K,  \mathcal{M}_f(X))>0.$
Let $\eta,\zeta>0$ with $\eta\leq \min\{\htop(f),\eta_0\}$ and $\zeta<\min\{d_H(K,  \mathcal{M}_f(X)),\rho_0,\zeta_0\}$. By the variational principle of the topological entropy, there exists $\nu_0\in \mathcal{M}_f(X)$ such that $h_{\nu_0}(f)>\htop(f)-\frac{\eta}{8}.$ 

By Lemma \ref{Maincor-mu-gamma-n-expansive},  there exist $\eps^*>0$ and $0<\tilde{\eps}^*<\eps^*$  such that for any $0<\frac{\delta}{2}<\tilde{\eps}^*$ and each $0\leq i\leq m,$ there exists an $n_i\in\N$ such that for any $p\in\N$,   there exists an $(pn_i,  \frac{\eps^*}{3})$-separated set $\Gamma_{pn_i}^{\nu_i}$ with
\begin{description}
	\item[(a)]\label{lem-B-aBasic} $\Gamma_{pn_i}^{\nu_i}\subseteq P_{pn_i}(f)\cap X_{pn_i,  \B(\nu_i,  \frac{\zeta}{4})}\cap B(x,  \frac{\delta}{2})$;
	\item[(b)]\label{lem-B-bBasic} $\frac{\log |\Gamma_{pn_i}^{\nu_i}|}{pn_i}>h_{\nu_i}(f)-\frac{\eta}{8}$.
\end{description}
We can assume that $\frac{\eps^*}{3}<\frac{c}{4}$ where $c>0$ is the expansive constant. Let $s(n,\frac{\eps^*}{3})$ denote the largest cardinality of any $(n,\frac{\eps^*}{3})$-separated set of $X$, then by \cite[Theorem 7.11]{Walters} one has
\begin{equation}
	\htop(f)=\limsup_{n\to\infty}\frac{1}{n}\log s(n,\frac{\eps^*}{3}).
\end{equation}
Then there exists $N\in\N$ such that for any $n\geq N,$ one has 
\begin{equation}\label{equation-AB}
	s(n,\frac{\eps^*}{3})<e^{n(\htop(f)+\frac{\eta}{4})}.
\end{equation}

Set $\eps=\min\{\frac{\zeta}{4},  \frac{\rho_0}{6},  \frac{\tilde{\eps}^*}{27},\frac{\eps_0}{2}\}.$  Then there exists a $0<\delta<\eps$ such that any $\delta$-pseudo-orbit can be $\eps$-shadowed by some point in $X.$ Set $n=p_0n_0n_1n_2\cdots n_m$ where $p_0$ is large enough such that for any $1\leq i\leq m$
\begin{equation}\label{equation-AA}
	n\geq 2N,\ e^{n(h_{\nu_i}(f)-\frac{\eta}{8})}-n\geq e^{n(h_{\nu_i}(f)-\frac{\eta}{4})}\ \text{ and } e^{n(\htop(f)-\frac{\eta}{4})}>\lceil\frac{n}{2}\rceil e^{\lceil\frac{n}{2}\rceil(\htop(f)+\frac{\eta}{4})}+\sum_{m=1}^{N_1-1}| P_{m}^*(f)|
\end{equation} 
where $P_{m}^*(f)$ is the set of periodic points with minimal period $m.$
Then for each $0\leq i\leq m$,   $P_{n_i}(f)\subseteq P_n(f)$ by definition and furthermore,   we can obtain an $(n,  \frac{\eps^*}{3})$-separated set $\Gamma_n^{\nu_i}$ with
\begin{description}
	\item[(a)]\label{lem-B-aBasic2} $\Gamma_{n}^{\nu_i}\subseteq P_{n}(f)\cap X_{n,  \B(\nu_i,  \frac{\zeta}{4})}\cap B(x, \frac{\delta}{2})$;
	\item[(b)]\label{lem-B-bBasic2} $\frac{\log |\Gamma_{n}^{\nu_i}|}{n}>h_{\nu_i}(f)-\frac{\eta}{8}$.
\end{description}
Since periodic points in $\Gamma_n^{\nu_0}$ with same period $l_0$ for some $l_0\in\N$ are $(l_0,\frac{\eps^*}{3})$-separated, by \eqref{equation-AB} and \eqref{equation-AA} we have
\begin{equation*}
	\begin{split}
		\sum_{m=1}^{\lceil\frac{n}{2}\rceil}| P_{m}^*(f)\cap \Gamma_n^{\nu_0}|\leq &\sum_{m=N_1}^{\lceil\frac{n}{2}\rceil} s(m,\frac{\varepsilon^*}{3})+\sum_{m=1}^{N_1-1}| P_{m}^*(f)|\\
		<&\sum_{m=N_1}^{\lceil\frac{n}{2}\rceil} e^{{m}(\htop(f)+\frac{\eta}{4})}+\sum_{m=1}^{N_1-1}| P_{m}^*(f)|\\
		\leq&{\lceil\frac{n}{2}\rceil}e^{{\lceil\frac{n}{2}\rceil}(\htop(f)+\frac{\eta}{4})}+\sum_{m=1}^{N_1-1}| P_{m}^*(f)|\\
		< &e^{n(\htop(f)-\frac{\eta}{4})}<e^{n(h_{\nu_0}(f)-\frac{\eta}{8})}<|\Gamma_{n}^{\nu_0}|.
	\end{split}
\end{equation*}
Thus there exists $x_0\in\Gamma_n^{\nu_0}$ with minimal period $n.$ Together with $\frac{\eps^*}{3}<\frac{c}{4},$ the only sub-intervals of length $n$ of $\langle x_0,  f(x_0),  \cdots,  f^{n-1}(x_0),x_0,  f(x_0),  \cdots,  f^{n-1}(x_0)\rangle$ that are $\frac{\eps^*}{9}$-shadowed by $\langle x_0,  f(x_0),  \cdots,  f^{n-1}(x_0)\rangle$ are the initial and the final sub-intervals.
By the separation assumption, for any $1\leq i\leq m$ we have $$|\{y\in\Gamma_n^{\nu_i}:d_n(y,f^j(x_0))<\frac{\eps^*}{9} \text{ for some }0\leq j\leq n-1\}|\leq n.$$ Consequently, by \eqref{equation-AA} one can find a subset $\widetilde{\Gamma}_n^{\nu_i}\subset \Gamma_n^{\nu_i}$ with $|\widetilde{\Gamma}_n^{\nu_i}|>e^{n(h_{\nu_i}-\frac{\eta}{4})}$ such that $d_n(y,f^j(x_0))\geq \frac{\eps^*}{9}$ for any $y\in\widetilde{\Gamma}_n^{\nu_i}$ and $0\leq j\leq n-1.$ 
Denote $r_i=|\widetilde{\Gamma}_n^{\nu_i}|$ and $r=\sum_{i=1}^mr_i$. Enumerate the elements of each $\widetilde{\Gamma}_n^{\nu_i}$ by $\widetilde{\Gamma}_n^{\nu_i}=\{p_1^i,  \cdots,  p_{r_i}^i\}$.  Let $\widetilde{\Gamma}_n=\{p_1^1,  \cdots,  p_{r_1}^1,  \cdots,  p_1^m, $ $ \cdots,  p_{r_m}^m\}.$

Take $l$ large enough such that
\begin{equation}\label{equation-AC}
	\frac{1}{l}<\frac{\zeta}{12} \text{ and }\frac{(l-2)\log|\widetilde{\Gamma}_n^{\nu_i}|}{nl}>h_{\nu_i}(f)-\eta \text{ for any } 1\leq i\leq m.
\end{equation}
Now let $\Gamma^i=\widetilde{\Gamma}_n^{\nu_i}\times\widetilde{\Gamma}_n^{\nu_i}\times\cdots\times\widetilde{\Gamma}_n^{\nu_i}$ whose element is $\underline{y}=(y_1,  \cdots,  y_{l-2})$ with $y_j\in\widetilde{\Gamma}_n^{\nu_i}$ for $1\leq j\leq l-2,$ and let $\Gamma=\widetilde{\Gamma}_n\times\widetilde{\Gamma}_n\times\cdots\times\widetilde{\Gamma}_n$ whose element is $\underline{y}=(y_1,  \cdots,  y_{l-2})$ with $y_j\in\widetilde{\Gamma}_n$ for $1\leq j\leq l-2.$   For any $y\in X$,  let $\mathfrak{C}_y^n=\langle y,  fy,  \cdots,  f^{n-1}y\rangle$. Then for $\underline{y}\in \Gamma^i$ or $\Gamma$ we define the following pseudo-orbit:
$$\mathfrak{C}_{\underline{y}}=\mathfrak{C}_{x_0}^n\mathfrak{C}_{x_0}^n\mathfrak{C}_{y_1}^n\mathfrak{C}_{y_2}^n\cdots\mathfrak{C}_{y_{l-2}}^n.$$
It is clear that $\mathfrak{C}_{\underline{y}}$ is a $\delta$-pseudo-orbit.   Moreover,   one notes that we can freely concatenate such $\mathfrak{C}_{\underline{y}}$s to constitutes a $\delta$-pseudo-orbit. We write $\mathfrak{C}_{\underline{y}}=\langle \omega_1,  \omega_2,  \cdots,  \omega_{ln}\rangle.$ If $\underline{y}\in\Gamma^i$ and $d(f^k(z),\omega_{k+1})\leq \eps$ for any $0\leq k\leq ln-1$ then by Lemma \ref{lem:prohorov} and \eqref{equation-AC} one has 
\begin{equation}\label{eq-AA}
	\begin{split}
		\rho(\E_{ln}(z),  \nu_i)\leq &\rho(\E_{ln}(z), \frac{1}{n(l-2)} \sum_{k=1}^{n(l-2)}\delta_{\omega_k})+\rho(\frac{1}{n(l-2)} \sum_{k=1}^{n(l-2)}\delta_{\omega_k},  \nu_i)\\
		\leq &\eps+\frac{4}{l}+\frac{1}{l-2}\sum_{j=1}^{l-2}\rho(\E_{n}(y_j),  \nu_i)\\
		<&\eps+\frac{4}{l}+\frac{\zeta}{4}<\frac34\zeta.
	\end{split}
\end{equation}
If $\underline{y}\in\Gamma$ and $d(f^k(z),\omega_{k+1})\leq \eps$ for any $0\leq k\leq ln-1,$ we denote $q_i=|\{1\leq j\leq l-2:y_j\in \widetilde{\Gamma}_n^{\nu_i}\}|,$
then $\sum_{i=1}^mq_i=l-2$ and
\begin{equation}\label{eq-AB}
	\begin{split}
		\rho(\E_{ln}(z),  \frac{\sum_{i=1}^mq_i\nu_i}{\sum_{i=1}^m q_i})\leq &\rho(\E_{ln}(z), \frac{1}{n(l-2)} \sum_{k=1}^{n(l-2)}\delta_{\omega_k})+\rho(\frac{1}{n(l-2)} \sum_{k=1}^{n(l-2)}\delta_{\omega_k}, \frac{\sum_{i=1}^mq_i\nu_i}{\sum_{i=1}^m q_i})\\
		\leq &\eps+\frac{4}{l}+\frac{1}{l-2}\sum_{i=1}^m\sum_{y_j\in \widetilde{\Gamma}_n^{\nu_i}}\rho(\E_{n}(y_j),  \nu_i)\\
		<&\eps+\frac{4}{l}+\frac{\zeta}{4}<\frac34\zeta.
	\end{split}
\end{equation}
by Lemma \ref{lem:prohorov} and \eqref{equation-AC}.
Now we define
$$\Sigma_{r_i^{l-2}}:=\{\theta=\dots\theta_{-2}\theta_{-1}\theta_{0}\theta_{1}\theta_{2}\dots:\theta_{j}=(\theta_{j,1},  \cdots,  \theta_{j,l-2})\in \Gamma^i \text{ for any } j\in\Z \},$$ 
$$\Sigma_{r^{l-2}}:=\{\theta=\dots\theta_{-2}\theta_{-1}\theta_{0}\theta_{1}\theta_{2}\dots:\theta_{j}=(\theta_{j,1},  \cdots,  \theta_{j,l-2})\in \Gamma \text{ for any } j\in\Z \}.$$
Then $(\Sigma_{r_i^{l-2}},\sigma)$ and $(\Sigma_{r^{l-2}},\sigma)$ are full shifts. This implies $(\Sigma_{r_i^{l-2}},\sigma)$ and $(\Sigma_{r^{l-2}},\sigma)$ are mixing and have shadowing property.  And for each $\theta=\dots\theta_{-2}\theta_{-1}\theta_{0}\theta_{1}\theta_{2}\dots$ in $\Sigma_{r_i^{l-2}}$ or $\Sigma_{r^{l-2}},$ 
$$\mathfrak{C}_{\theta}=\dots\mathfrak{C}_{\theta_{-2}}\mathfrak{C}_{\theta_{-1}}\mathfrak{C}_{\theta_{0}}\mathfrak{C}_{\theta_{1}}\mathfrak{C}_{\theta_{2}}\dots$$
is a $\delta$-pseudo-orbit. We write $\mathfrak{C}_{\theta}=\dots\omega_{-2}\omega_{-1}\omega_{0}\omega_{1}\omega_{2}\dots,$ by the shadowing property,
$$Y_{\theta}=\{z\in X:d(f^{j}(z),  \omega_j)\leq\eps,  j\in\Z\}$$
is nonempty and closed. 

We claim that $Y_{\theta}\cap Y_{\theta'}=\emptyset$ for any $\theta\neq\theta'$ in $\Sigma_{r_i^{l-2}}$ or $\Sigma_{r^{l-2}}.$ Next we prove the claim by the following two cases.

Case (1):  If $\theta\neq\theta'\in \Sigma_{r_i^{l-2}}$ for some $1\leq i\leq m,$ then there is $t\in\Z$ and $1\leq s\leq l-2$ such that $\theta_{t,s}\neq\theta'_{t,s}.$ Since $\theta_{t,s}$ and $\theta'_{t,s}$ are $(n,\frac{\eps^*}{3})$-separated, we have $d_n(f^{lnt+sn+n}(z),f^{lnt+sn+n}(z'))>\eps^*/3-2\eps>\eps^*/9$ for any $z\in Y_{\theta}$ and $z'\in Y_{\theta'}.$ So $Y_{\theta}\cap Y_{\theta'}=\emptyset.$

Case (2): For any  $\theta\neq\theta'\in \Sigma_{r^{l-2}},$ there is $t\in\Z$ and $1\leq s\leq l-2$ such that $\theta_{t,s}\neq\theta'_{t,s}.$ If $\theta_{t,s},\theta'_{t,s}\in\widetilde{\Gamma}_n^{\nu_i}$ for some $1\leq i\leq m,$ then $Y_{\theta}\cap Y_{\theta'}=\emptyset$ by Case (1). If there are $1\leq i\neq i'\leq m$ such that $\theta_{t,s}\in\widetilde{\Gamma}_n^{\nu_i}$ and $\theta'_{t,s}\in\widetilde{\Gamma}_n^{\nu_{i'}}$, then $$d_n(f^{lnt+sn+n}(z),f^{lnt+sn+n}(z'))>\eps$$
for any $z\in Y_{\theta}$ and $z'\in Y_{\theta'}.$ Otherwise, we have $$\rho(\mathcal{E}_{n}(f^{lnt+sn+n}(z)),\mathcal{E}_{n}(f^{lnt+sn+n}(z')))\leq \eps\leq \frac{\rho_0}{6}.$$ 
Combining with $$\rho(\mathcal{E}_{n}(f^{lnt+sn+n}(z)),\mathcal{E}_{n}(\theta_{t,s}))\leq \eps\leq \frac{\rho_0}{6},$$ $$\rho(\mathcal{E}_{n}(f^{lnt+sn+n}(z')),\mathcal{E}_{n}(\theta'_{t,s}))\leq \eps\leq \frac{\rho_0}{6}$$ 
and $$\rho(\nu_i,\mathcal{E}_{n}(\theta_{t,s}))<\frac{\zeta}{4}\leq \frac{\rho_0}{4},$$ $$\rho(\nu_{i'},\mathcal{E}_{n}(\theta'_{t,s}))<\frac{\zeta}{4}\leq \frac{\rho_0}{4},$$ 
we have  $\rho(\nu_i,\nu_{i'})<\rho_0$ which contradicts that $\rho_0=\min\{\rho(\nu_i,  \nu_j):1\leq i<j\leq m\}.$ 
So $$d_n(f^{lnt+sn+n}(z),f^{lnt+sn+n}(z'))>\eps$$ for any $z\in Y_{\theta}$ and $z'\in Y_{\theta'}.$
This implies $Y_{\theta}\cap Y_{\theta'}=\emptyset.$
Then we can define the following disjoint union:
$$\Delta_i=\bigsqcup_{\theta\in \Sigma_{r_i^{l-2}}}Y_{\theta}~\textrm{and}~\Delta=\bigsqcup_{\theta\in \Sigma_{r^{l-2}}}Y_{\theta}.  $$
Note that $f^{nl}(Y_{\theta})= Y_{\sigma(\theta)}$.   Then $f^{nl}(\Delta_i)=\Delta_i$,   $1\leq i\leq m$ and $f^{nl}(\Delta)=\Delta.$
Therefore,   if we define $\pi:\Delta\to \Sigma_{r^{l-2}}$ and $\pi_i:\Delta_i\to \Sigma_{r_i^{l-2}}$ as
\begin{equation*}
	\pi(x):=\theta \textrm{ for all } x\in Y_{\theta}~\textrm{with}~\theta\in \Sigma_{r^{l-2}},
\end{equation*}
\begin{equation*}
	\pi_i(x):=\theta' \textrm{ for all } x\in Y_{\theta'}~\textrm{with}~\theta'\in \Sigma_{r_i^{l-2}},
\end{equation*}
then $\pi$ and $\pi_i$ are surjective by the shadowing property.   Moreover,   it is not hard to check that $\pi$ and $\pi_i$ are continuous.  So $\Delta$ and $\Delta_i$ are closed. Meanwhile,   $(X,  f)$ is expansive, so $\pi,  \pi_i$ are conjugations.

Let $\Lambda_i=\cup_{k=0}^{nl-1}f^k(\Delta_i)$ and $\Lambda=\cup_{k=0}^{nl-1}f^k(\Delta).$  Then $f(\Lambda_i)=\Lambda_i$,   $1\leq i\leq m$ and $f(\Lambda)=\Lambda.$
Now let us prove that $\Lambda$ and $\Lambda_i$ satisfy the property 1-4.

(1) Since $\pi$ and $\pi_i$ are conjugations, the transitivity, expansivity and the shadowing property of $(\Sigma_{r_i^{l-2}},\sigma)$ and $(\Sigma_{r^{l-2}},\sigma)$ yield the same properties of $(\Delta,  f^{nl})$ and $(\Delta_i,  f^{nl})$. Next we show that $f^{k}(\Delta_i)\cap f^{k'}(\Delta_i)=\emptyset$ for any $0\leq k<k'\leq nl-1.$ If $f^{k}(\Delta_i)\cap f^{k'}(\Delta_i)\neq\emptyset,$ then for any $z\in f^{k}(\Delta_i)\cap f^{k'}(\Delta_i),$ there exist $\theta,\  \theta'\in \Sigma_{r_i^{l-2}}$ such that 
$$d(f^{j-k}(z),  \omega_j)\leq\eps \text{ and } d(f^{j-k'}(z),  \omega'_j)\leq\eps\ \forall j\in\Z$$
where $\mathfrak{C}_{\theta}=\dots\omega_{-2}\omega_{-1}\omega_{0}\omega_{1}\omega_{2}\dots$ and $\mathfrak{C}_{\theta'}=\dots\omega'_{-2}\omega'_{-1}\omega'_{0}\omega'_{1}\omega'_{2}\dots.$  Then we have 
\begin{equation}\label{eq-AD}
	d(\omega_{j+k},  \omega'_{j+k'})\leq2\eps\ \forall j\in\Z.
\end{equation}

Case (1): If $1\leq k'-k\leq n-1,$ then \eqref{eq-AD} implies $d(\omega_{j},  \omega'_{j+k'-k})\leq2\eps<\frac{\eps^*}{9}\ \forall 0\leq j\leq n-1.$ Note that $\omega_{0}\omega_{1}\omega_{2}\dots\omega_{2n-1}=\omega'_{0}\omega'_{1}\omega'_{2}\dots\omega'_{2n-1}=\langle x_0,  fx_0,  \cdots,  f^{n-1}x_0,x_0,  fx_0,  \cdots,  f^{n-1}x_0\rangle,$ this contradicts that the minimal period of $x_0$ is $n.$ 

Case (2): If $k'-k= n,$ then \eqref{eq-AD} implies $d(\omega_{j+n},  \omega'_{j+2n})\leq2\eps<\frac{\eps^*}{9}\ \forall 0\leq j\leq n-1.$ Note that $\omega_{n}\omega_{n+1}\dots\omega_{2n-1}=\langle x_0,  fx_0,  \cdots,  f^{n-1}x_0\rangle$ and $\omega'_{2n}\in \widetilde{\Gamma}_n^{\nu_i},$ this contradicts that $d_n(y,f^j(x_0))\geq \frac{\eps^*}{9}$ for any $y\in\widetilde{\Gamma}_n^{\nu_i}$ and $0\leq j\leq n-1.$ 

Case (3): If $n< k'-k\leq n(l-1),$ then $k'-k=tn+s$ for some $1\leq t\leq l-2$ and $0\leq s\leq n-1.$ Thus \eqref{eq-AD} implies $d(\omega_{n-s+j},  \omega'_{(t+1)n+j})\leq2\eps\ \forall 0\leq j\leq n-1.$ Note that $\omega_{0}\omega_{1}\omega_{2}\dots\omega_{2n-1}=\langle x_0,  fx_0,  \cdots,  f^{n-1}x_0,x_0,  fx_0,  \cdots,  f^{n-1}x_0\rangle,$ and $\omega'_{(t+1)n}\in \widetilde{\Gamma}_n^{\nu_i},$ this contradicts that $d_n(y,f^j(x_0))\geq \frac{\eps^*}{9}$ for any $y\in\widetilde{\Gamma}_n^{\nu_i}$ and $0\leq j\leq n-1.$ 

Case (4): If $n(l-1)< k'-k\leq nl-1,$ then \eqref{eq-AD} implies $d(\omega_{nl-k'+k+j},  \omega'_{nl+j})\leq2\eps\ \forall 0\leq j\leq n-1.$ Note that $\omega_{0}\omega_{1}\omega_{2}\dots\omega_{2n-1}=\omega'_{nl}\omega'_{nl+1}\omega'_{nl+2}\dots\omega'_{(l+2)n-1}=\langle x_0,  fx_0,  \cdots,  f^{n-1}x_0,x_0,  fx_0,  \cdots,  f^{n-1}x_0\rangle,$ and $1\leq nl-k'+k<n,$ this contradicts that the minimal period of $x_0$ is $n.$ 

So we have $f^{k}(\Delta_i)\cap f^{k'}(\Delta_i)=\emptyset$ for any $0\leq k<k'\leq nl-1.$ Therefore, Proposition \ref{prop-transitive-shadowing-for-f-n} ensures that
$(\Lambda_i,f)$ is transitive and topologically Anosov. 
In fact, if we define $$\Omega_i=\{\mathfrak{C}_{\theta}=\dots\omega_{-2}\omega_{-1}\omega_{0}\omega_{1}\omega_{2}\dots:\theta\in \Sigma_{r_i^{l-2}}\},$$
then $(\Omega_i,\sigma^{nl})$ conjugates to $(\Sigma_{r_i^{l-2}},\sigma).$ Thus $(\Omega_i,\sigma^{nl})$ conjugates to $(\Delta_i,  f^{nl}),$ and $(\cup_{k=0}^{nl-1}\sigma^k(\Omega_i),\sigma)$ conjugates to $(\Lambda_i,  f).$ This implies $(\cup_{k=0}^{nl-1}\sigma^k(\Omega_i),\sigma)$ is a subshift which is transitive and topologically Anosov. Recall from \cite{Walters2} a subshift satisfies shadowing property if and only if it is a subshift of finite type. So $(\cup_{k=0}^{nl-1}\sigma^k(\Omega_i),\sigma)$ is a transitive subshift of finite type.
By similar method, $(\Lambda,f)$ is also transitive and topologically Anosov, and $(\Lambda,f)$ conjugate to a transitive subshift of ﬁnite type.

(2)  $\htop(f,  \Lambda_i)=\frac{1}{nl}\htop(f^{nl},  \Delta_i)=\frac{1}{nl}\htop(\sigma,  \Sigma_{r_i^{l-2}})=\frac{(l-2)\log|\widetilde{\Gamma}_n^{\nu_i}|}{nl}>h_{\nu_i}(f)-\eta>h_{\nu_i}(f)-\eta_0$ by \eqref{equation-AC}.

(3) For any ergodic measure $\mu_i\in \mathcal{M}_f(\Lambda_i)$,   pick an arbitrary generic point $z_i$ of $\mu_i$ in $\Delta_i$. 
Then $$\rho(\E_{ln}(f^{tln}(z_i)),  \nu_i)<\frac34\zeta\text{ for any } t\in\N$$ by \eqref{eq-AA}.
In addition,   we have $\mu_i=\lim_{j\to\infty}\E_j(z_i)=\lim_{t\to\infty}\E_{tln}(z_i)$.   So we have $$\rho(\mu_i,  \nu_i)=\lim_{t\to\infty}\rho(\E_{tln}(z_i),  \nu_i)\leq \frac34\zeta.$$   By the ergodic decomposition Theorem,   we obtain that $d_H(\nu_i,  \mathcal{M}_f(\Lambda_i))\leq \frac34\zeta$.
Now since $K$ is convex and $\Lambda_i\subseteq\Lambda$,   one gets that $K\subseteq \B(\mathcal{M}_f( \Lambda),  \zeta)\subseteq \B(\mathcal{M}_f( \Lambda),  \zeta_0)$.

On the other hand,   for any ergodic measure $\mu\in \mathcal{M}_f( \Lambda)$,   pick a generic point $z$ of $\mu$ in $\Delta$.   Then $z$ $\eps$-shadows some $\delta$-pseudo-orbit $\mathfrak{C}_{\theta}$ with $\theta\in\Sigma_{r^{l-2}}$.   Then for any $t\in\N$, there exist nonnegative integers $q_i,  1\leq i\leq m$ such that $$\rho\left(\E_{ln}(f^{tln}(z)),  \frac{\sum_{i=1}^mq_i\nu_i}{\sum_{i=1}^m q_i}\right)< \frac34\zeta$$ by \eqref{eq-AB}.
So $\mu\in \B(K,  \frac34\zeta).$ By the ergodic decomposition Theorem,   $\mathcal{M}_f(\Lambda)\subseteq \B(K,  \zeta)\subseteq \B(K,  \zeta_0)$.   As a result,   $\Lambda\subsetneq X$.   For otherwise,   $d_H(K,  \mathcal{M}_f(\Lambda))=d_H(K,  \mathcal{M}_f(X))>\zeta$,   a contradiction.

(4) Note thet  for any $\theta$ in $\Sigma_{r_i^{l-2}}$ or $\Sigma_{r^{l-2}},$ one has $\omega_{tnl}=x_0$ for any integer $t$ where $\mathfrak{C}_{\theta}=\dots\omega_{-1}\omega_{0}\omega_{1}\dots.$ Then for any $z$ in $\Delta_i$ or $\Delta,$ $$d(f^{tnl}(z),x)\leq d(f^{tnl}(z),x_0)+d(x_0,x)<\eps+\frac{\delta}{2}<2\eps<\eps_0.$$
So for any $z$ in $\Lambda_i$ or $\Lambda$ one has  $f^{j+tnl}(z) \in B(x,\eps)$ for some $0\leq j\leq nl-1$ and any $t\in\Z$.\qed

\section{Asymptotically additive sequences: an abstract version of Theorem \ref{thm-continuous}(I)}\label{Almost Additive}
In this section we give  abstract conditions on which  Theorem \ref{thm-continuous}(I) holds in the more general context of asymptotically additive sequences of continuous functions. We first recall some definitions about asymptotically additive sequences.
Consider a dynamical system $(X,  f).$ A sequence of functions $\Phi=\left(\varphi_{n}\right)_{n\in\mathbb{N}}$ is said to be {\it additive} (with respect to  $(X,f)$ ) if for any $m,n\in \mathbb{N}$ and any $x\in X$, we have $\varphi_{m+n}(x)=\varphi_m(x)+\varphi_n(f^{m}(x)).$ Then for any continuous function $\varphi,$ $\left(\varphi_{n}=\sum_{i=1}^{n-1}\varphi\circ f^i\right)_{n\in\mathbb{N}}$ is additive.
A sequence of functions $\Phi=\left(\varphi_{n}\right)_{n\in\mathbb{N}}$ is said to be {\it asymptotically additive} (with respect to  $(X,f)$ ) if for each $\varepsilon>0$, there exists $\varphi\in C(X)$ such that
$$
\limsup\limits _{n \rightarrow \infty} \frac{1}{n} \sup _{x \in X}\left|\varphi_n(x)-S_n \varphi(x)\right| \leqslant \varepsilon,
$$
where $S_n \varphi=\sum_{k=0}^{n-1} \varphi \circ f^k$. Obviously,  every additive sequence is asymptotically additive.
We denote by $AA(f,X)$ the family of asymptotically additive sequences of continuous functions. It is not difficult to see that for any $\Phi=\left(\varphi_{n}\right)_{n\in\mathbb{N}}\in AA(f,X)$ and any $\mu\in \mathcal{M}_f(X),$ the limit $\lim\limits_{n \rightarrow \infty} \frac{1}{n} \int \varphi_{n} d \mu$ exists  and 
the function,
\begin{equation}\label{equation-P}
\mathcal{M}_f(X) \ni \mu \mapsto \lim _{n \rightarrow \infty} \frac{1}{n} \int\varphi_{n} d \mu,
\end{equation}
is continuous with the weak* topology in $\mathcal{M}_f(X),$ for example, see \cite{FH2010}.
Let $d \in \mathbb{N}$ and take $(A, B) \in AA(f,X)^{d} \times AA(f,X)^{d}$. We write
$
A=\left(\Phi^{1}, \ldots, \Phi^{d}\right) \text { and } B=\left(\Psi^{1}, \ldots, \Psi^{d}\right)
$
and also $\Phi^{i}=\left(\varphi_{n}^{i}\right)_{n\in\N}$ and $\Psi^{i}=\left(\psi_{n}^{i}\right)_{n\in\N}$. We consider $(A, B)$ satisfying the following condition:
\begin{equation}\label{equation-O}
	\lim _{n \rightarrow \infty} \frac{1}{n} \int\psi_{n}^i d \mu\geq 0 \text{ for any } \mu\in \mathcal{M}_f(X), \text{ with equality only permitted if } \lim _{n \rightarrow \infty} \frac{1}{n} \int\varphi_{n}^i d \mu\neq 0
\end{equation}
for every $i=1, \ldots, d$. 
We  consider the function $\mathcal{P}_{A,B}: \mathcal{M}_f(X) \rightarrow \mathbb{R}$ defined by:
\begin{equation}\label{equation-A}
		\mathcal{P}_{A,B}(\mu) 
		=\lim _{n \rightarrow \infty}\left(\frac{\int \varphi_{n}^{1} d \mu}{\int \psi_{n}^{1} d \mu}, \ldots, \frac{\int \varphi_{n}^{d} d \mu}{\int \psi_{n}^{d} d \mu}\right).
\end{equation}
\eqref{equation-P} ensures that the function $\mathcal{P}_{A,B}$ is continuous.

Let $\alpha:\mathcal{M}_f(X)\to \mathbb{R}$ be a continuous function. 
We define the pressure of $\alpha$ with respect to $\mu$ by $P(f,\alpha,\mu)=h_\mu(f)+\alpha(\mu).$ For any convex subset $C$ of $\mathcal{M}_f(X)$ and $a\in  \mathcal{P}_{A,B}(C),$ denote $$H_{A,B}(f,\alpha,a,C)=\sup\{P(f,\alpha,\mu):\mu\in \mathcal{P}_{A,B}^{-1}(a)\cap C\}.$$
In particular, when $\alpha\equiv0,$ we write 
$H_{A,B}(f,a,C)=H_{A,B}(f,0,a,C)=\sup\{h_\mu(f):\mu\in \mathcal{P}_{A,B}^{-1}(a)\cap C\}.$
For convenience, we write $H_{A,B}(f,\alpha,a)=H_{A,B}(f,\alpha,a,\mathcal{M}_f(X)),$ $H_{A,B}(f,a)=H_{A,B}(f,a,\mathcal{M}_f(X))$.

\begin{Def}
	Given a dynamical system $(X,  f).$ Let $C_1, C_2$ be two subset of $\mathcal{M}_f(X)$ with $C_2\subset C_1.$ We say $(X,f)$ satisfies the {\it locally conditional variational principle} with respect to $(C_1,C_2)$, if for any positive integer $m,$ any $f$-invariant measures
	$\{\mu_i\}_{i=1}^m\subseteq C_1,$ and any $\eta,  \zeta>0$,   there exist compact invariant subsets $\Lambda_i\subseteq\Lambda\subset X$ such that for each $1\leq i\leq m$
	\begin{description}
		\item[(1)] $\htop(f,  \Lambda_i)>h_{\mu_i}(f)-\eta.$
		\item[(2)] $d_H(K,  \mathcal{M}_f(\Lambda))<\zeta$,   $d_H(\mu_i,  \mathcal{M}_f(\Lambda_i))<\zeta,$ where $K=\cov\{\mu_i\}_{i=1}^m.$
		\item [(3)] $\mathcal{M}_f(\Lambda)\subset C_1.$
		\item[(4)]  for any $d \in \mathbb{N}$, any $(A, B) \in AA(f,\Lambda)^{d} \times AA(f,\Lambda)^{d}$ satisfying \eqref{equation-O}, and any $a\in \mathrm{Int}(\mathcal{P}_{A,B}(\mathcal{M}_f(\Lambda)))$ we have $H_{A,B}(f,a,\mathcal{M}_f(\Lambda))=H_{A,B}(f,a,\mathcal{M}_f(\Lambda)\cap C_2).$
	\end{description}
\end{Def}

Now we give  an abstract version of  Theorem \ref{thm-continuous}(I)  in the more general context of asymptotically additive sequences of continuous functions.

\begin{Thm}\label{thm-Almost-Additive}
	Suppose $(X,  f)$ is a dynamical system. Let $C_1, C_2$ be two subset of $\mathcal{M}_f(X)$ with $C_2\subset C_1$ such that $C_1$ is convex, and $(X,f)$ satisfies the locally conditional variational principle with respect to $(C_1,C_2)$.   Let $d \in \mathbb{N}$ and $(A, B) \in AA(f,X)^{d} \times AA(f,X)^{d}$ satisfying \eqref{equation-O} and $\mathrm{Int}(\mathcal{P}_{A,B}(C_1))\neq\emptyset$. Let $\alpha$ be a continuous function on  $C_1$. If $\mu\to h_\mu(f)$ is upper semi-continuous on $C_1$, 
	then for any $a\in \mathrm{Int}(\mathcal{P}_{A,B}(C_1)),$ any $\mu\in \mathcal{P}_{A,B}^{-1}(a)\cap C_1$ and any $\eta,  \zeta>0$, there is $\nu\in \mathcal{P}_{A,B}^{-1}(a)\cap C_2$ such that $\rho(\nu,\mu)<\zeta$ and $|P(f,\alpha,\nu)-P(f,\alpha,\mu)|<\eta.$
\end{Thm}
\begin{Ex}\label{example-1}
	The function $\alpha:\mathcal{M}_f(X)\to \mathbb{R}$ can be defined as  following:
	\begin{description}
		\item[(1)] $\alpha\equiv0.$ Then $P(f,\alpha,\mu)=h_\mu(f)$ is the metric entropy of $\mu.$
		\item[(2)] $\alpha(\mu)=\int\varphi d \mu$ with a continuous function $\varphi.$ Then from the $weak^{*}$-topology on $\mathcal{M}(X),$ $\alpha:\mathcal{M}_f(X)\to \mathbb{R}$ is a continuous function. $P(f,\varphi,\mu)=h_\mu(f)+\alpha(\mu)$ is the pressure of $\varphi$ with respect to $\mu.$
		\item[(3)] $\alpha(\mu)=\lim\limits_{n \rightarrow \infty} \frac{1}{n} \int\varphi_{n} d \mu$ with an asymptotically additive sequences of continuous functions $\Phi=\left(\varphi_{n}\right)_{n\in\N}.$ Then $\alpha:\mathcal{M}_f(X)\to \mathbb{R}$ is a continuous function from (\ref{equation-P}). 
		$P(f,\Phi,\mu)=h_\mu(f)+\lim\limits_{n \rightarrow \infty} \frac{1}{n} \int\varphi_{n} d \mu$ is the pressure of $\Phi$ with respect to $\mu.$ Readers can refer to \cite{Barreira1996,FH2010} for thermodynamic formalism of asymptotically additive sequences.
	\end{description}
\end{Ex}

\subsection{Some lemmas}
To prove Theorem \ref{thm-Almost-Additive}, we need some lemmas.
For any $r\in \mathbb{R},$ denote $r^+=\{s\in\mathbb{R}:s>r\}$ and $r^-=\{s\in\mathbb{R}:s<r\}.$  For any $d \in \mathbb{N},$ $r=\left(r_1, \ldots, r_{d}\right)\in\mathbb{R}^d$ and $\xi=\left(\xi_1, \ldots, \xi_{d}\right)\in\{+,-\}^d$, we define $$r^\xi=\{s=\left(s_1, \ldots, s_{d}\right)\in\mathbb{R}^d:s_i\in r_i^{\xi_i} \text{ for }i=1,2,\cdots,d\}.$$ We denote $F^d=\{\left(\frac{p_1}{q_1}, \ldots, \frac{p_d}{q_d}\right):p_i,q_i\in\mathbb{R}\text{ and }q_i>0 \text{ for any }1\leq i\leq d\}.$
It is easy to check that
\begin{Lem}\label{lemma-E}
	Let $b_i=\frac{p^i}{q^i}\in F^1$ for $i=1,2.$
	\begin{description}
		\item[(1)] If $b_1=b_2,$ then $\frac{\theta p^1+(1-\theta)p^2}{\theta q^1+(1-\theta)q^2}=b_1=b_2$ for any $\theta\in[0,1].$
		\item[(2)] If $b_1\neq b_2,$ then $\frac{\theta p^1+(1-\theta)p^2}{\theta q^1+(1-\theta)q^2}$ is strictly monotonic on $\theta\in[0,1].$
	\end{description}
\end{Lem}

\begin{Lem}\label{lemma-C}
	Let $d \in \mathbb{N}$ and $a=\left(\frac{p_1}{q_1}, \ldots, \frac{p_d}{q_d}\right) \in F^{d}.$  If $\{b_\xi=\left(\frac{p_1^\xi}{q_1^\xi}, \ldots, \frac{p_d^\xi}{q_d^\xi}\right)\}_{\xi\in\{+,-\}^d}\subseteq F^d$ are $2^d$ numbers satisfying $b_\xi\in a^\xi$ for any $\xi\in \{+,-\}^d,$ then there are $2^d$ numbers $\{\theta_\xi\}_{\xi\in\{+,-\}^d}\subseteq [0,1]$ such that $\sum_{\xi\in\{+,-\}^d}\theta_\xi=1$ and $ \frac{\sum_{\xi\in\{+,-\}^d}\theta_\xi p^{\xi}_i}{\sum_{\xi\in\{+,-\}^d}\theta_\xi q^{\xi}_i}=\frac{p_i}{q_i}\text{ for any }1\leq i\leq d.$
\end{Lem}
\begin{proof}
	We prove the lemma inductively. It is clearly true if $d=1$ from Lemma \ref{lemma-E}.
	Now we assume that it is true for $d=k\in\mathbb{N}.$ Let $a=\left(a_1, \ldots, a_{k+1}\right) \in \mathbb{R}^{k+1},$ and  $\{b^\xi\}_{\xi\in\{+,-\}^{k+1}}$ is $2^{k+1}$ numbers satisfies $b^\xi\in a^\xi$ for any $\xi\in \{+,-\}^{k+1}.$ Then for the $2^{k}$ numbers $\{b^\xi\}_{\xi_{k+1}=+},$ there is $2^k$ numbers $\{\tau^\xi\}_{\xi_{k+1}=+}\subseteq [0,1]$ such that  $\sum_{\xi_{k+1}=+}\tau^\xi=1$ and 
	\begin{equation}\label{equation-S}
		\frac{\sum_{\xi_{k+1}=+}\tau_\xi p^{\xi}_i}{\sum_{\xi_{k+1}=+}\tau_\xi q^{\xi}_i}=\frac{p_i}{q_i}\text{ for any }1\leq i\leq k.
	\end{equation} 
    Since $\frac{p^\xi_{k+1}}{q^\xi_{k+1}}>\frac{p_{k+1}}{q_{k+1}}$ for any $\xi\in\{+,-\}^{k+1}$ with $\xi_{k+1}=+,$ then we have 
	\begin{equation}\label{equation-Q}
		\frac{\sum_{\xi_{k+1}=+}\tau_\xi p^{\xi}_{k+1}}{\sum_{\xi_{k+1}=+}\tau_\xi q^{\xi}_{k+1}}>\frac{p_{k+1}}{q_{k+1}}.
	\end{equation}
	Similarly, the $2^{k}$ numbers $\{b^\xi\}_{\xi_{k+1}=-},$ there is $2^k$ numbers $\{\tau^\xi\}_{\xi_{k+1}=-}\subseteq [0,1]$ such that  $\sum_{\xi_{k+1}=-}\tau^\xi=1$ and 
	\begin{equation}\label{equation-T}
		\frac{\sum_{\xi_{k+1}=-}\tau_\xi p^{\xi}_i}{\sum_{\xi_{k+1}=-}\tau_\xi q^{\xi}_i}=\frac{p_i}{q_i}\text{ for any }1\leq i\leq k.
	\end{equation} 
    Since $\frac{p^\xi_{k+1}}{q^\xi_{k+1}}<\frac{p_{k+1}}{q_{k+1}}$ for any $\xi\in\{+,-\}^{k+1}$ with $\xi_{k+1}=-,$ then we have 
	\begin{equation}\label{equation-R}
		\frac{\sum_{\xi_{k+1}=-}\tau_\xi p^{\xi}_{k+1}}{\sum_{\xi_{k+1}=-}\tau_\xi q^{\xi}_{k+1}}<\frac{p_{k+1}}{q_{k+1}}.
	\end{equation}
	By \eqref{equation-Q}, \eqref{equation-R} and Lemma \ref{lemma-E}(2) there is $\tau_{k+1}\in(0,1)$ such that 
	$$\frac{\tau_{k+1}\sum_{\xi_{k+1}=+}\tau_\xi p^{\xi}_{k+1}+(1-\tau_{k+1})\sum_{\xi_{k+1}=-}\tau_\xi p^{\xi}_{k+1}}{\tau_{k+1}\sum_{\xi_{k+1}=+}\tau_\xi q^{\xi}_{k+1}+(1-\tau_{k+1})\sum_{\xi_{k+1}=-}\tau_\xi q^{\xi}_{k+1}}=\frac{p_{k+1}}{q_{k+1}}.$$
	By \eqref{equation-S}, \eqref{equation-T} and Lemma \ref{lemma-E}(1), we have $$\frac{\tau_{k+1}\sum_{\xi_{k+1}=+}\tau_\xi p^{\xi}_{i}+(1-\tau_{k+1})\sum_{\xi_{k+1}=-}\tau_\xi p^{\xi}_{i}}{\tau_{k+1}\sum_{\xi_{k+1}=+}\tau_\xi q^{\xi}_{i}+(1-\tau_{k+1})\sum_{\xi_{k+1}=-}\tau_\xi q^{\xi}_{i}}=\frac{p_{i}}{q_{i}}\text{ for any }1\leq i\leq k.$$
	Let 
	$
	\theta_\xi=\left\{\begin{array}{ll}
		\tau_{k+1}\tau_\xi, & \text { for } \xi_{k+1}=+  \\
		(1-\tau_{k+1})\tau_\xi, & \text { for } \xi_{k+1}=-.
	\end{array}\right.
	$
	Then we have $\sum_{\xi\in\{+,-\}^{k+1}}\theta_\xi=1$ and $ \frac{\sum_{\xi\in\{+,-\}^{k+1}}\theta_\xi p^{\xi}_i}{\sum_{\xi\in\{+,-\}^{k+1}}\theta_\xi q^{\xi}_i}=\frac{p_i}{q_i}\text{ for any }1\leq i\leq k+1.$ So we complete the proof of Lemma \ref{lemma-C}.
\end{proof} 
By Lemma \ref{lemma-C} we have
\begin{Cor}\label{corollary-A}
	Suppose $(X,  f)$ is a dynamical system. Let $d \in \mathbb{N}$ and $(A, B) \in AA(f,X)^{d} \times AA(f,X)^{d}$ satisfying \eqref{equation-O}. Then for any $a\in \mathcal{P}_{A,B}(\mathcal{M}_f(X))$ and $2^d$ invariant measures $\{\mu_\xi\}_{\xi\in\{+,-\}^d}$ with 
	$\mathcal{P}_{A,B}\left(\mu_\xi\right)\in a^\xi \text{ for any }\xi\in \{+,-\}^d,$ there are $2^d$ numbers $\{\theta_\xi\}_{\xi\in\{+,-\}^d}\subseteq [0,1]$ such that $\sum_{\xi\in\{+,-\}^d}\theta_\xi=1$ and $\mathcal{P}_{A,B}(\sum_{\xi\in\{+,-\}^d}\theta_\xi\mu_\xi)= a.$
\end{Cor}
\begin{Cor}\label{corollary-AA}
	Suppose $(X,  f)$ is a dynamical system, and $C$ is a convex subset of $\mathcal{M}_f(X).$ Let $d \in \mathbb{N}$ and $(A, B) \in AA(f,X)^{d} \times AA(f,X)^{d}$ satisfying \eqref{equation-O}. Given $a\in \mathcal{P}_{A,B}(C).$ If there are $2^d$ invariant measures $\{\mu_\xi\}_{\xi\in\{+,-\}^d}\subset C$ such that
	$\mathcal{P}_{A,B}\left(\mu_\xi\right)\in a^\xi$ for any $\xi\in \{+,-\}^d,$ then $a\in \mathrm{Int}(\mathcal{P}_{A,B}(C)).$
\end{Cor}
\begin{proof}
	Choose an open subset $B$ of $\mathbb{R}^d$ such  that $a\in B$ and for any $\tilde{a}\in B$ we have $\mathcal{P}_{A,B}\left(\mu_\xi\right)\in \tilde{a}^\xi$ for any $\xi\in \{+,-\}^d.$ By Corollary  \ref{corollary-A}  there are $2^d$ numbers $\{\theta_\xi\}_{\xi\in\{+,-\}^d}\subseteq [0,1]$ such that $\sum_{\xi\in\{+,-\}^d}\theta_\xi=1$ and $\mathcal{P}_{A,B}(\sum_{\xi\in\{+,-\}^d}\theta_\xi\mu_\xi)= \tilde{a}.$ $C$ is convex, then $\sum_{\xi\in\{+,-\}^d}\theta_\xi\mu_\xi\in C.$ So $B\subset \mathcal{P}_{A,B}(C)$ and thus $a\in \mathrm{Int}(\mathcal{P}_{A,B}(C)).$
\end{proof}
\begin{Lem}\label{lemma-D}
	Suppose $(X,  f)$ is a dynamical system, and  $C$ is a convex subset of $\mathcal{M}_f(X).$ Let $d \in \mathbb{N}$ and $(A, B) \in AA(f,X)^{d} \times AA(f,X)^{d}$ satisfying \eqref{equation-O} and $\mathrm{Int}(\mathcal{P}_{A,B}(C))\neq\emptyset$. Then for any $a\in \mathrm{Int}(\mathcal{P}_{A,B}(C)),$ any $\mu\in \mathcal{P}_{A,B}^{-1}(a)\cap C$ and any $\eta,\zeta>0$, there exist $2^d$ invariant measures $\{\mu_\xi\}_{\xi\in\{+,-\}^d}\subset C$ such that 
	$\mathcal{P}_{A,B}\left(\mu_\xi\right)\in a^\xi, h_{\mu_\xi}(f)>h_{\mu}(f)-\eta \text{ and } \rho(\mu_\xi,\mu)<\zeta \text{ for any }\xi\in \{+,-\}^d.$
\end{Lem}
\begin{proof}
	By $a\in \mathrm{Int}(\mathcal{P}_{A,B}(C))$ there is $\nu_\xi\in C$ such that $\mathcal{P}_{A,B}\left(\nu_\xi\right)\in a^\xi$ for any $\xi\in \{+,-\}^d.$ Then there is $\tau_\xi\in(0,1)$ close to $1$ such that $\mu_\xi=\tau_\xi\mu+(1-\tau_\xi)\nu_\xi\in C$ satisfies 
	$h_{\mu_\xi}(f)>h_{\mu}(f)-\eta \text{ and } \rho(\mu_\xi,\mu)<\zeta \text{ for any }\xi\in \{+,-\}^d.$ And we have $\mathcal{P}_{A,B}\left(\mu_\xi\right)\in a^\xi$ by $\tau^\xi>0$ and Lemma \ref{lemma-E}(2).
\end{proof}

\begin{Lem}\label{Lem-interior}
	Suppose $(X,  f)$ is a dynamical system, and  $C$ is a convex subset of $\mathcal{M}_f(X).$ Assume that $C$ satisfies the following properties: for any positive integer $m,$ any $f$-invariant measures
	$\{\mu_i\}_{i=1}^m\subseteq C,$ and any $\eta,  \zeta>0$,   there exist compact invariant subsets $\Lambda_i\subseteq\Lambda\subset X$ such that for each $1\leq i\leq m$
	\begin{description}
		\item[(1)] $\htop(f,  \Lambda_i)>h_{\mu_i}(f)-\eta.$
		\item[(2)] $d_H(K,  \mathcal{M}_f(\Lambda))<\zeta$,   $d_H(\mu_i,  \mathcal{M}_f(\Lambda_i))<\zeta,$ where $K=\cov\{\mu_i\}_{i=1}^m.$
		\item [(3)] $\mathcal{M}_f(\Lambda)\subset C.$
		\item[(4)] $\Lambda$ has property $\mathbf{P}.$
	\end{description}
    Let $d \in \mathbb{N}$ and $(A, B) \in AA(f,X)^{d} \times AA(f,X)^{d}$ satisfying \eqref{equation-O} and $\mathrm{Int}(\mathcal{P}_{A,B}(C))\neq\emptyset$.
    Then for any $a\in \mathrm{Int}(\mathcal{P}_{A,B}(C)),$ any $\mu\in \mathcal{P}_{A,B}^{-1}(a)\cap C$ and any $\eta,  \zeta>0$, there is a compact invariant subset $\Lambda\subset X$ such that
   $a\in  \mathrm{Int}(\mathcal{P}_{A,B}(\mathcal{M}_f(\Lambda))),$ $d_H(\mu,  \mathcal{M}_f(\Lambda))<\zeta,$ $\mathcal{M}_f(\Lambda)\subset C,$ $\Lambda$ has property $\mathbf{P},$
    	and  $H_{A,B}(f,a,\mathcal{M}_f(\Lambda))
    	> h_{\mu}(f)-\eta.$  If further $\mu\to h_\mu(f)$ is upper semi-continuous on $C$, then $\Lambda$ also satisfies $H_{A,B}(f,a,\mathcal{M}_f(\Lambda))<h_{\mu}(f)+\eta.$ 
\end{Lem}
\begin{proof}
	Fix $a_0\in \mathrm{Int}(\mathcal{P}_{A,B}(C)),$ $\mu_0\in \mathcal{P}_{A,B}^{-1}(a_0)\cap C$ and $\eta,  \zeta>0.$ 
	If $\mu\to h_\mu(f)$ is upper semi-continuous on $C$, there is $0<\zeta'<\zeta$ such that for any $\omega\in C$ with $\rho(\mu_0,\omega)<\zeta'$ we have
	\begin{equation}\label{equation-C}
		h_{\omega}(f)<h_{\mu_0}(f)+\eta.
	\end{equation} 
	By Lemma \ref{lemma-D} there exist $2^d$ invariant measures $\{\mu_\xi\}_{\xi\in\{+,-\}^d}\subset C$ such that 
	\begin{equation}\label{equation-U}
		\mathcal{P}_{A,B}\left(\mu_\xi\right)\in a_0^\xi, h_{\mu_\xi}(f)>h_{\mu_0}(f)-\frac{\eta}{3} \text{ and } \rho(\mu_\xi,\mu_0)<\frac{\zeta'}{2} \text{ for any }\xi\in \{+,-\}^d.
	\end{equation}
	Since the function $\mathcal{P}_{A,B}$ is continuous,  there is $0<\zeta''<\zeta'$ such that such that  for any $\omega_\xi \in C$ with $\rho(\omega_\xi,\mu_\xi)<\zeta''$ we have
	\begin{equation}\label{equation-B}
		\mathcal{P}_{A,B}\left(\omega_\xi\right)\in a_0^\xi
	\end{equation} 
	For the $2^d$ invariant measures $\{\mu_\xi\}_{\xi\in\{+,-\}^d},$  there exist compact invariant subsets $\Lambda_\xi\subseteq\Lambda\subset X$ such that for each $\xi\in \{+,-\}^d$
	\begin{description}
		\item[(1)] $\htop(f,  \Lambda_\xi)>h_{\mu_\xi}(f)-\frac{\eta}{3}.$
		\item[(2)] $d_H(\cov\{\mu_\xi\}_{\xi\in \{+,-\}^d},  \mathcal{M}_f(\Lambda))<\frac{\zeta''}{2}$,   $d_H(\mu_\xi,  \mathcal{M}_f( \Lambda_\xi))<\frac{\zeta''}{2}.$
		\item[(3)] $\mathcal{M}_f(\Lambda)\subset C.$
		\item[(4)] $\Lambda$ has property $\mathbf{P}.$
	\end{description}
	By item(1) and the variational principle of the topological entropy, there are $\nu_\xi\in \mathcal{M}_f(\Lambda_\xi)$ such that $$h_{\nu_\xi}(f)>\htop(f,  \Lambda_\xi)-\frac{\eta}{3}>h_{\mu_\xi}(f)-\frac{2\eta}{3}>h_{\mu_0}(f)-\eta.$$
	Then by item(2) and \eqref{equation-B}, we have $\mathcal{P}_{A,B}\left(\nu_\xi\right)\in a_0^\xi.$ 
	By Corollary \ref{corollary-A} there are $2^d$ numbers $\{\theta_\xi\}_{\xi\in\{+,-\}^d}\subseteq [0,1]$ such that $\sum_{\xi\in\{+,-\}^d}\theta_\xi=1$ and $\mathcal{P}_{A,B}\left(\nu'\right)= a_0$ where $\nu'=\sum_{\xi\in\{+,-\}^d}\theta_\xi\nu_\xi\in \mathcal{M}_f(\Lambda).$ 
	Then by Corollary \ref{corollary-AA} we have $a_0\in \mathrm{Int}(\mathcal{P}_{A,B}(\mathcal{M}_f(\Lambda)))$ and 
	$H_{A,B}(f,a,\mathcal{M}_f(\Lambda))
		\geq h_{\nu'}(f)\geq \min\{h_{\nu_\xi}(f):\xi\in \{+,-\}^d\}
		> h_{\mu_0}(f)-\eta.$
    By item(2) and   \eqref{equation-U}, we have $d_H(\mu,  \mathcal{M}_f(\Lambda))<\zeta'.$ 
	So by \eqref{equation-C} we have 
	$H_{A,B}(f,a,\mathcal{M}_f(\Lambda))<h_{\mu_0}(f)+\eta.$
\end{proof}

\subsection{Proof of Theorem \ref{thm-Almost-Additive}}
	Fix $a_0\in \mathrm{Int}(\mathcal{P}_{A,B}(C_1)),$ $\mu_0\in \mathcal{P}_{A,B}^{-1}(a_0)\cap C_1$ and $\eta,  \zeta>0.$ 
	Since $\alpha$ is continuous on  $C_1$, there is $0<\zeta'<\zeta$ such that for any $\omega\in C_1$ with $\rho(\mu_0,\omega)<\zeta'$ we have
	\begin{equation}\label{equation-C'}
		|\alpha(\omega)-\alpha(\mu_0)|<\frac{\eta}{2}.
	\end{equation} 
    By Lemma \ref{Lem-interior} there is a compact invariant set $\Lambda$ such that
    $a_0\in  \mathrm{Int}(\mathcal{P}_{A,B}(\mathcal{M}_f(\Lambda))),$ $d_H(\mu_0,  \mathcal{M}_f(\Lambda))<\zeta',$ $\mathcal{M}_f(\Lambda)\subset C,$ 
    $H_{A,B}(f,a_0,\mathcal{M}_f(\Lambda))=H_{A,B}(f,a_0,\mathcal{M}_f(\Lambda)\cap C_2)$
    and  $|H_{A,B}(f,a_0,\mathcal{M}_f(\Lambda))-h_{\mu_0}(f)|<\frac{\eta}{4}.$ 
    Then there is $\nu\in \mathcal{M}_f(\Lambda)\cap C_2$ such that $\mathcal{P}_{A,B}\left(\nu\right)=a_0$ and $|h_{\nu}(f)-H_{A,B}(f,a_0,\mathcal{M}_f(\Lambda))|<\frac{\eta}{4}.$ It follows  $|h_{\nu}(f)-h_{\mu_0}(f)|<\frac{\eta}{2}.$
    By \eqref{equation-C'} we have $|\alpha(\nu)-\alpha(\mu_0)|<\frac{\eta}{2},$ and thus $|P(f,\alpha,\nu)-P(f,\alpha,\mu_0)|<\eta.$\qed

\section{Asymptotically additive sequences: an abstract version of Theorem \ref{thm-continuous}(II)-(IV)}\label{section-almost2}

In this section we give  abstract conditions on which  Theorem \ref{thm-continuous}(II)-(IV) holds in the more general context of asymptotically additive sequences of continuous functions. 
Consider a dynamical system $(X,  f).$ Let $d \in \mathbb{N}$ and $(A, B) \in AA(f,X)^{d} \times AA(f,X)^{d}$ such that
$B$ satisfies \eqref{equation-O}. Let $\alpha:\mathcal{M}_f(X)\to \mathbb{R}$ be a continuous function.    
Recall that the pressure of $\alpha$ with respect to $\mu$ is $P(f,\alpha,\mu)=h_\mu(f)+\alpha(\mu).$ For any convex subset $C$ of $\mathcal{M}_f(X)$ and $a\in  \mathcal{P}_{A,B}(C),$ denote
$$\alpha_{A,B}(a,C)=\sup\{\alpha(\mu):\mu\in \mathcal{P}_{A,B}^{-1}(a)\cap C\} .$$ 
For convenience, we write $\alpha_{A,B}(a)=\alpha_{A,B}(a,\mathcal{M}_f(X)).$
We list two conditions for $\alpha:$
\begin{description}
	\item[(A.1)] For any $\mu_1, \mu_2 \in C$ with $P(f,\alpha,\mu_1) \neq P(f,\alpha,\mu_2)$
	\begin{equation}\label{equation-W}
		\beta(\theta):=P(f,\alpha,\theta \mu_1+(1-\theta) \mu_2)\text{ is strictly monotonic on }[0,1].
	\end{equation} 
	\item[(A.2)] For any $\mu_1, \mu_2 \in C$ with $P(f,\alpha,\mu_1) = P(f,\alpha,\mu_2)$
	\begin{equation}\label{equation-AF}
		\beta(\theta):=P(f,\alpha,\theta \mu_1+(1-\theta) \mu_2)\text{ is constant on }[0,1].
	\end{equation} 
\end{description}
Now we give an abstract version of  Theorem \ref{thm-continuous}(II)-(IV) in the more general context of asymptotically additive sequences of continuous functions. 
\begin{Thm}\label{thm-Almost-Additive2}
	Suppose $(X,  f)$ is a dynamical system. Assume that $C_1$ is a convex subset of $\mathcal{M}_f(X)$, $C_2$ is a dense $G_\delta$ subset of $C_1$ and $(X,f)$ satisfies the locally conditional variational principle with respect to $(C_1,C_2)$.   Let $d \in \mathbb{N}$ and $(A, B) \in AA(f,X)^{d} \times AA(f,X)^{d}$ satisfying \eqref{equation-O} and $\mathrm{Int}(\mathcal{P}_{A,B}(C_1))\neq\emptyset$. Let $\alpha$ be a continuous function on $C_1$ satisfying \eqref{equation-W} and \eqref{equation-AF}. If $\mu\to h_\mu(f)$ is upper semi-continuous on $C_1$ and $\{\mu\in C_1:h_{\mu}(f)=0\}$ is dense in $C_1,$ 
	then 
	\begin{description}
		\item[(II)] For any $a\in \mathrm{Int}(\mathcal{P}_{A,B}(C_1)),$ any $\mu\in \mathcal{P}_{A,B}^{-1}(a)\cap C_1$ with $P(f,\alpha,\mu)\geq \alpha_{A,B}(a,C_1),$ any $\alpha_{A,B}(a,C_1)\leq P\leq P(f,\alpha,\mu)$ and any $\eta,  \zeta>0$, there is $\nu\in \mathcal{P}_{A,B}^{-1}(a)\cap C_2$ such that $\rho(\nu,\mu)<\zeta$ and $|P(f,\alpha,\nu)-P|<\eta.$ 
		\item[(III)] For any $a\in \mathrm{Int}(\mathcal{P}_{A,B}(C_1))$ and $\alpha_{A,B}(a,C_1)\leq P< H_{A,B}(f,\alpha,a,C_1),$ 
		$$
			\{\mu\in \mathcal{P}_{A,B}^{-1}(a)\cap C_2:P(f,\alpha,\mu)\geq P\} \text{ is  a dense $G_\delta$ subset of }\{\mu\in \mathcal{P}_{A,B}^{-1}(a)\cap C_1:P(f,\alpha,\mu)\geq P\},
		$$
	    $$
			\{\mu\in \mathcal{P}_{A,B}^{-1}(a)\cap C_1:P(f,\alpha,\mu)= P\} \text{ is  a dense $G_\delta$ subset of }\{\mu\in \mathcal{P}_{A,B}^{-1}(a)\cap C_1:P(f,\alpha,\mu)\geq P\}.
		$$
		If there is $\mu_{\mathrm{full}}\in C_1$ such that $S_{\mu_{\mathrm{full}}}=X$, then for any $a\in \mathrm{Int}(\mathcal{P}_{A,B}(C_1))$ and $\alpha_{A,B}(a,C_1)\leq P< H_{A,B}(f,\alpha,a,C_1),$ the set 
		$\{\mu\in \mathcal{P}_{A,B}^{-1}(a)\cap C_1:P(f,\alpha,\mu)\geq P,\ S_\mu=X\}$ is  a dense $G_\delta$ subset of $\{\mu\in \mathcal{P}_{A,B}^{-1}(a)\cap C_1:P(f,\alpha,\mu)\geq P\}.
		$
		\item[(IV)] If further $
		\{\mu\in \mathcal{P}_{A,B}^{-1}(a)\cap C_2:P(f,\alpha,\mu)= P\}$  is  dense in $\{\mu\in \mathcal{P}_{A,B}^{-1}(a)\cap C_1:P(f,\alpha,\mu)\geq P\}$ for any $a\in \mathrm{Int}(\mathcal{P}_{A,B}(C_1))$ and $\alpha_{A,B}(a,C_1)\leq P< H_{A,B}(f,\alpha,a,C_1),$  then the set   $\{(\mathcal{P}_{A,B}(\mu), P(f,\alpha,\mu)):\mu\in C_1,\ a=\mathcal{P}_{A,B}(\mu)\in\mathrm{Int}(\mathcal{P}_{A,B}(C_1)),\ \alpha_{A,B}(a,C_1)\leq P(f,\alpha,\mu)< H_{A,B}(f,\alpha,a,C_1)\}$ coincides with $\{(\mathcal{P}_{A,B}(\mu), P(f,\alpha,\mu)):\mu\in C_2,\ a=\mathcal{P}_{A,B}(\mu)\in\mathrm{Int}(\mathcal{P}_{A,B}(C_1)), \alpha_{A,B}(a,C_1)\leq P(f,\alpha,\mu)< H_{A,B}(f,\alpha,a,C_1)\}.$
	\end{description}
\end{Thm}
\begin{Rem}\label{Rem-baire}
	A set $C$ is said to be  Baire  if every countable intersection of dense open sets is dense.  When $C_1=\mathcal{M}_f(X)$ in Theorem \ref{thm-Almost-Additive2}, $\{\mu\in \mathcal{P}_{A,B}^{-1}(a)\cap C_1:P(f,\alpha,\mu)\geq P\}$ is a compact subset of $\mathcal{M}_f(X)$ since $\mu\to h_\mu(f)$ is upper semi-continuous. Then $\{\mu\in \mathcal{P}_{A,B}^{-1}(a)\cap C_1:P(f,\alpha,\mu)\geq P\}$ is a Baire set. So by item(III), $
	\{\mu\in \mathcal{P}_{A,B}^{-1}(a)\cap C_2:P(f,\alpha,\mu)= P\}$  is  a dense $G_\delta$ subset of $\{\mu\in \mathcal{P}_{A,B}^{-1}(a)\cap C_1:P(f,\alpha,\mu)\geq P\},$ and if there is $\mu_{\mathrm{full}}\in C_1$ such that $S_{\mu_{\mathrm{full}}}=X$, then  
	$\{\mu\in \mathcal{P}_{A,B}^{-1}(a)\cap C_2:P(f,\alpha,\mu)= P,\ S_\mu=X\}$ is  a dense $G_\delta$ subset of $\{\mu\in \mathcal{P}_{A,B}^{-1}(a)\cap C_1:P(f,\alpha,\mu)\geq P\}.
	$
\end{Rem}
\begin{Ex}
	The function $\alpha:\mathcal{M}_f(X)\to \mathbb{R}$ can be defined as  
	\begin{description}
		\item[(1)] $\alpha\equiv0.$ 
		\item[(2)] $\alpha(\mu)=\int\varphi d \mu$ with a continuous function $\varphi.$ 
		\item[(3)] $\alpha(\mu)=\lim\limits_{n \rightarrow \infty} \frac{1}{n} \int\varphi_{n} d \mu$ with an almost additive sequences of continuous functions $\Phi=\left(\varphi_{n}\right)_{n\in\N}.$ 
	\end{description}
    Then $\alpha:\mathcal{M}_f(X)\to \mathbb{R}$ is a continuous function from Example \ref{example-1}. Furthermore, $\alpha$ is affine and thus it satisfies \eqref{equation-W} and \eqref{equation-AF} if it is defined as above.  
\end{Ex}

\subsection{Proof of Theorem \ref{thm-Almost-Additive2}}
We ﬁrst establish several auxiliary results.
\begin{Lem}\label{lemma-GG}
	Suppose $(X,  f)$ is a dynamical system. Let $C$ be a convex compact subset of $\mathcal{M}_f(X).$ Let $d \in \mathbb{N}$ and $(A, B) \in AA(f,X)^{d} \times AA(f,X)^{d}$ satisfying \eqref{equation-O}  and $\mathrm{Int}(\mathcal{P}_{A,B}(C))\neq\emptyset$.   Assume $V$ is a nonempty convex subset of $C.$ Then for any $a\in \mathrm{Int}(\mathcal{P}_{A,B}(C)),$ the following properties holds:
	\begin{enumerate}
		\item[(1)] If $V$ is a dense  subset  of  $C,$ then $\mathcal{P}_{A,B}^{-1}(a)\cap V$ is a dense  subset  of  $\mathcal{P}_{A,B}^{-1}(a)\cap C.$
		 \item[(1)] If $V$ is a dense $G_\delta$ subset  of $C,$ then $ \mathcal{P}_{A,B}^{-1}(a)\cap V$ is a dense $G_\delta$ subset  of  $\mathcal{P}_{A,B}^{-1}(a)\cap C.$
	\end{enumerate}
\end{Lem}
\begin{proof}
	(1) Fix $a\in \mathrm{Int}(\mathcal{P}_{A,B}(C)),$  $\mu\in \mathcal{P}_{A,B}^{-1}(a)\cap C$ and  $\zeta>0.$  By Lemma \ref{lemma-D} there exist $2^d$ invariant measures $\{\mu_\xi\}_{\xi\in\{+,-\}^d}\subset C$ such that 
	\begin{equation}\label{equ-AA}
		\mathcal{P}_{A,B}\left(\mu_\xi\right)\in a^\xi \text{ and } \rho(\mu_\xi,\mu)<\frac{\zeta}{2} \text{ for any }\xi\in \{+,-\}^d.
	\end{equation}
	Since $V$ is dense in $C$,  then there exist $\nu_\xi\in V$ close to $\mu_\xi$ such that 
	\begin{equation}\label{equ-BB}
		\mathcal{P}_{A,B}\left(\nu_\xi\right)\in a^\xi \text{ and } \rho(\nu_\xi,\mu_\xi)<\frac{\zeta}{2} \text{ for each } \xi\in \{+,-\}^d.
	\end{equation}
	By Corollary \ref{corollary-A} there are $2^d$ numbers $\{\theta_\xi\}_{\xi\in\{+,-\}^d}\subseteq [0,1]$ such that $\sum_{\xi\in\{+,-\}^d}\theta_\xi=1$ and $\mathcal{P}_{A,B}\left(\nu'\right)= a$ where $\nu'=\sum_{\xi\in\{+,-\}^d}\theta_\xi\nu_\xi.$ 
	Since $V$ is convex, we have $\nu'\in V.$ By \eqref{equ-AA} and \eqref{equ-BB} we have $\rho(\nu',\mu)<\zeta.$ 
	So $V\cap \mathcal{P}_{A,B}^{-1}(a)$ is a dense  subset  of  $\mathcal{P}_{A,B}^{-1}(a)\cap C.$
	
	(2) By item(1) we obtain item(2).
\end{proof}
\begin{Lem}\label{lemma-B}
	Suppose $(X,  f)$ is a dynamical system. Let $V$ be a convex subset of $\mathcal{M}_f(X).$ If there is an invariant measure $\mu_V\in V$ with $S_{\mu_V}=X,$ then $\{\mu\in V:S_\mu=X\}$ is a  dense $G_\delta$ subset of $V.$ 
\end{Lem}
\begin{proof}
	$\{\mu\in \mathcal{M}_f(X):S_\mu=X\}$ is either empty or a dense $G_\delta$ subset of  $\mathcal{M}_f(X)$ \cite[Proposition 21.11]{DGS}. So if there is  $\mu_V\in V$ with $S_{\mu_V}=X,$ then $\{\mu\in \mathcal{M}_f(X):S_\mu=X\}$ is a dense $G_\delta$ subset of $\mathcal{M}_f(X).$ Thus $\{\mu\in V:S_\mu=X\}$ is a $G_\delta$ subset of  $V.$ For any $\nu\in V$ and $\theta\in(0,1),$ we have $\nu_\theta=\theta\nu+(1-\theta)\mu_V\in V$ and $S_{\nu_\theta}=X.$ So $\{\mu\in V:S_\mu=X\}$ is dense in  $V,$ and thus is a dense $G_\delta$ subset of $V.$ 
\end{proof}

\begin{Lem}\label{lemma-A}
	Suppose $(X,  f)$ is a dynamical system. Assume that $C_1$ is a  convex subset of $\mathcal{M}_f(X)$, $C_2$ is a dense $G_\delta$ subset of $C_1$.   Let $d \in \mathbb{N}$ and $(A, B) \in AA(f,X)^{d} \times AA(f,X)^{d}$ satisfying \eqref{equation-O} and $\mathrm{Int}(\mathcal{P}_{A,B}(C_1))\neq\emptyset$. Let $\alpha$ be a continuous function on $C_1$ satisfying \eqref{equation-W} and \eqref{equation-AF}. Then for any $a\in \mathrm{Int}(\mathcal{P}_{A,B}(C))$ and $\alpha_{A,B}(a,C_1)\leq P< H_{A,B}(f,\alpha,a,C_1),$ the following properties hold:
	\begin{description}
		\item[(1)] If $\{\mu\in \mathcal{P}_{A,B}^{-1}(a)\cap C_2:P(f,\alpha,\mu)\geq P\}$ is dense in $\{\mu\in \mathcal{P}_{A,B}^{-1}(a)\cap C_1:P(f,\alpha,\mu)\geq P\},$ then  $\{\mu\in \mathcal{P}_{A,B}^{-1}(a)\cap C_2:P(f,\alpha,\mu)\geq P\}$ is a dense $G_\delta$ subset of $\{\mu\in \mathcal{P}_{A,B}^{-1}(a)\cap C_1:P(f,\alpha,\mu)\geq P\}.$
		\item[(2)] If there is $\mu_{\mathrm{full}}\in C_1$ such that $S_{\mu_{\mathrm{full}}}=X$, then $\{\mu\in \mathcal{P}_{A,B}^{-1}(a)\cap C_1:P(f,\alpha,\mu)\geq P,\ S_\mu=X\}$ is a dense  $G_\delta$ subset of $\{\mu\in \mathcal{P}_{A,B}^{-1}(a)\cap C_1:P(f,\alpha,\mu)\geq P\}.$   
		\item[(3)] If $\{\mu\in C_1:h_{\mu}(f)=0\}$ is dense in $C_1,$ then $\{\mu\in \mathcal{P}_{A,B}^{-1}(a)\cap C_1:P(f,\alpha,\mu)= P\}$ is dense in $\{\mu\in \mathcal{P}_{A,B}^{-1}(a)\cap C_1:P(f,\alpha,\mu)\geq P\}.$  If further $\mu\to h_\mu(f)$ is upper semi-continuous on $C_1$, then $\{\mu\in \mathcal{P}_{A,B}^{-1}(a)\cap C_1:P(f,\alpha,\mu)= P\}$ is a dense  $G_\delta$ subset of $\{\mu\in \mathcal{P}_{A,B}^{-1}(a)\cap C_1:P(f,\alpha,\mu)\geq P\}.$  
	\end{description}
\end{Lem}
\begin{proof}
	(1) Since $C_2$ is a $G_\delta$ subset of $C_1,$ then  $\{\mu\in \mathcal{P}_{A,B}^{-1}(a)\cap C_2:P(f,\alpha,\mu)\geq P\}$ is a $G_\delta$ subset of $\{\mu\in \mathcal{P}_{A,B}^{-1}(a)\cap C_1:P(f,\alpha,\mu)\geq P\}.$ So we have item(1).

	(2) By Lemma \ref{lemma-B}, $\{\mu\in C_1:S_\mu=X\}$ is a dense $G_\delta$ subset of $C_1.$  Note that $\{\mu\in C_1:S_\mu=X\}$ is convex.
	So by Lemma \ref{lemma-GG} there is $\omega\in \mathcal{P}_{A,B}^{-1}(a)\cap C_1$ such that $S_\omega=X.$
	Since $P< H_{A,B}(f,\alpha,a,C_1),$ there is $\nu\in \mathcal{P}_{A,B}^{-1}(a)\cap C_1$ such that $P(f,\alpha,\nu)>P.$
	By \eqref{equation-W} we can choose $\theta\in(0,1)$ close to $1$ such that $\tilde{\mu}=\theta\nu+(1-\theta)\omega$ satisfies $P(f,\alpha,\tilde{\mu})>P.$ 
	Then $\tilde{\mu}\in \{\mu\in \mathcal{P}_{A,B}^{-1}(a)\cap C_1:P(f,\alpha,\mu)\geq P,\ S_\mu=X\}.$ Note that $\{\mu\in \mathcal{P}_{A,B}^{-1}(a)\cap C_1:P(f,\alpha,\mu)\geq P\}$ is a convex set by \eqref{equation-W}, \eqref{equation-AF} and Lemma  \ref{lemma-E}(1). So by Lemma \ref{lemma-B} we complete the proof of item(2).
	
	(3) 
	Fix $\mu_0\in \mathcal{P}_{A,B}^{-1}(a)\cap C_1$ with $P(f,\alpha,\mu_0)\geq P$  and $\zeta>0.$ 
	By Lemma \ref{lemma-GG}, there is $\nu'\in  \mathcal{P}_{A,B}^{-1}(a)\cap C_1$ such that $h_{\nu'}(f)=0$ and $\rho(\nu',\mu_0)<\zeta.$  
	Now by \eqref{equation-W} we choose $\theta\in[0,1]$ such that  $\nu=\theta\mu_0+(1-\theta)\nu'\in C_1$ satisfies $P(f,\alpha,\nu)=P.$ Then by Lemma \ref{lemma-E}(1)
	$
		\mathcal{P}_{A,B}(\nu)=a \text{ and } \rho(\nu,\mu_0)<\zeta.
    $
	So $\{\mu\in \mathcal{P}_{A,B}^{-1}(a)\cap C_1:P(f,\alpha,\mu)= P\}$ is dense in $\{\mu\in \mathcal{P}_{A,B}^{-1}(a)\cap C_1:P(f,\alpha,\mu)\geq P\}.$  If further $\mu\to h_\mu(f)$ is upper semi-continuous on $C_1$, $\{\mu\in C_1:P(f,\alpha,\mu)\in[P,P+\frac{1}{n})\}$ is open in $\{\mu\in C_1:P(f,\alpha,\mu)\geq P\}$ for any $n\in\mathbb{N^{+}}.$  Then 
	$\{\mu\in \mathcal{P}_{A,B}^{-1}(a)\cap C_1:P(f,\alpha,\mu)= P\}=\cap _{n\geq 1}\{\mu\in C_1:P(f,\alpha,\mu)\in[P,P+\frac{1}{n})\}$ is a  $G_\delta$ subset of $\{\mu\in \mathcal{P}_{A,B}^{-1}(a)\cap C_1:P(f,\alpha,\mu)\geq P\},$ and thus we complete the proof of item(3).  
\end{proof}

Now we show that the result of Theorem \ref{thm-Almost-Additive}  is the keystone of the proof of Theorem \ref{thm-Almost-Additive2}.
\begin{Lem}\label{lemma-F}
	Suppose $(X,  f)$ is a dynamical system, $C_1$ is a convex subset of $\mathcal{M}_f(X)$, $C_2$ is a dense $G_\delta$ subset of $C_1$ and $(X,f)$ satisfies the locally conditional variational principle with respect to $(C_1,C_2)$.   Let $d \in \mathbb{N}$ and $(A, B) \in AA(f,X)^{d} \times AA(f,X)^{d}$ satisfying \eqref{equation-O} and $\mathrm{Int}(\mathcal{P}_{A,B}(C_1))\neq\emptyset$. Let $\alpha$ be a continuous function on $C_1$ satisfying \eqref{equation-W} and \eqref{equation-AF}. 
	Assume that for any $a\in \mathrm{Int}(\mathcal{P}_{A,B}(C_1)),$ any $\mu\in \mathcal{P}_{A,B}^{-1}(a)\cap C_1$ and any $\eta,  \zeta>0$, there is $\nu\in \mathcal{P}_{A,B}^{-1}(a)\cap C_2$ such that $\rho(\nu,\mu)<\zeta$ and $|P(f,\alpha,\nu)-P(f,\alpha,\mu)|<\eta.$
	If $\mu\to h_\mu(f)$ is upper semi-continuous on $C_1$ and $\{\mu\in C_1:h_{\mu}(f)=0\}$ is dense in $C_1,$   then items (II)(III)(IV) of Theorem  \ref{thm-Almost-Additive2} hold.
\end{Lem}
\begin{proof}
	(1) Fix $a_0\in \mathrm{Int}(\mathcal{P}_{A,B}(C_1)),$  $\mu_0\in \mathcal{P}_{A,B}^{-1}(a_0)$ with $P(f,\alpha,\mu_0)\geq \alpha_{A,B}(a_0,C_1),$   $\alpha_{A,B}(a_0,C_1)\leq P\leq P(f,\alpha,\mu_0)$
	and $\eta,  \zeta>0.$ By Lemma \ref{lemma-A}(3), there exists $\nu'\in \mathcal{P}_{A,B}^{-1}(a_0)\cap C_1$ such that $P(f,\alpha,\nu')=P$ and $\rho(\nu',\mu_0)<\frac{\zeta}{2}.$
	For the $a_0\in \mathrm{Int}(\mathcal{P}_{A,B}(C_1)),$ $\nu'\in \mathcal{P}_{A,B}^{-1}(a_0)\cap C_1$ and $\eta,  \frac{\zeta}{2}>0,$ there is $\nu\in \mathcal{P}_{A,B}^{-1}(a_0)\cap C_2$ such that $\rho(\nu,\nu')<\frac{\zeta}{2}$ and $|P(f,\alpha,\nu)-P(f,\alpha,\nu')|<\eta.$ Then we complete the proof of item(II).
	
	(2) Fix $a_0\in \mathrm{Int}(\mathcal{P}_{A,B}(C_1))$ and  $\alpha_{A,B}(a_0,C_1)\leq P< H_{A,B}(f,\alpha,a_0,C_1),$ First we show that
	\begin{equation}\label{equation-K}
		\{\mu\in \mathcal{P}_{A,B}^{-1}(a_0)\cap C_2:P(f,\alpha,\mu)\geq P\} \text{ is dense in }\{\mu\in \mathcal{P}_{A,B}^{-1}(a_0)\cap C_1:P(f,\alpha,\mu)\geq P\}.
	\end{equation} 
	Let $\mu_0\in \mathcal{P}_{A,B}^{-1}(a_0)\cap C_1$ be an invariant measure with $P(f,\alpha,\mu_0) \geq P$ and $\zeta>0$. If $P(f,\alpha,\mu_0)>P$, then there is  $\eta>0$ such that $P<P+\eta<P(f,\alpha,\mu_0).$ For the $a_0\in \mathrm{Int}(\mathcal{P}_{A,B}(C_1)),$ $\mu_0\in \mathcal{P}_{A,B}^{-1}(a_0)\cap C_1$ and $\eta,  \zeta>0,$ there exists  $\nu\in \mathcal{P}_{A,B}^{-1}(a_0)\cap C_2$ such that $\rho(\nu,\mu_0)<\zeta$ and $|P(f,\alpha,\nu)-P(f,\alpha,\mu_0)|<\eta.$ 
	If $P(f,\alpha,\mu_0)=P$, then we can pick a $\mu'\in \mathcal{P}_{A,B}^{-1}(a_0)\cap C_1$ such that $P<P(f,\alpha,\mu') \leq  H_{A,B}(f,\alpha,a_0,C_1)$, and next pick a sufficiently small number $\theta \in(0,1)$ such that $\rho\left(\mu_0, \mu''\right)<\zeta / 2$, where $\mu''=(1-\theta) \mu_0+\theta\mu'.$ By \eqref{equation-W} we have $P(f,\alpha,\mu'')>P.$ By the same argument, there exists $\nu\in \mathcal{P}_{A,B}^{-1}(a_0)\cap C_2$ such that $\rho(\nu,\mu'')<\zeta/2$ and $P(f,\alpha,\nu)> P.$  So $\rho(\nu,\mu_0)<\zeta.$
	
	By \eqref{equation-K} and Lemma \ref{lemma-A}(1),
	\begin{equation*}
		\{\mu\in \mathcal{P}_{A,B}^{-1}(a_0)\cap C_2:P(f,\alpha,\mu)\geq P\} \text{ is  a dense $G_\delta$ subset of }\{\mu\in \mathcal{P}_{A,B}^{-1}(a_0)\cap C_1:P(f,\alpha,\mu)\geq P\}.
	\end{equation*} 
    So by Lemma \ref{lemma-A}(3) and (2), we complete the proof of item(III). 
	
	(3) Fix $a_0\in \mathrm{Int}(\mathcal{P}_{A,B}(C_1))$ and $\mu_0\in \mathcal{P}_{A,B}^{-1}(a_0)\cap C_1$ with $\alpha_{A,B}(a_0,C_1)\leq P(f,\alpha,\mu_0)< H_{A,B}(f,\alpha,a_0,C_1).$ 
	Then by item(2) $\{\mu\in \mathcal{P}_{A,B}^{-1}(a_0)\cap C_2:P(f,\alpha,\mu)=P(f,\alpha,\mu_0)\}$ is a dense $G_\delta$ subset of $\{\mu\in \mathcal{P}_{A,B}^{-1}(a_0)\cap C_1:P(f,\alpha,\mu)\geq P(f,\alpha,\mu_0)\}.$ In particular, there is $\mu_a\in  \mathcal{P}_{A,B}^{-1}(a)\cap C_2$ such that $P(f,\alpha,\mu_a)=P(f,\alpha,\mu_0).$  So we complete the proof of item(IV).
\end{proof}

\noindent{\bf Proof of Theorem \ref{thm-Almost-Additive2}:} Note that the conditions of Theorem \ref{thm-Almost-Additive} is contained in Theorem \ref{thm-Almost-Additive2}. So we obtain Theorem \ref{thm-Almost-Additive2} by Lemma \ref{lemma-F}.\qed

\section{Proofs of Theorem \ref{thm-continuous}-\ref{maintheorem-robust}}\label{section-thm}
In this section, we use ’multi-horseshoe’ dense property and  results of asymptotically additive sequences   obtained in Section  \ref{section-entropy-dense}-\ref{section-almost2} to give a more general result than Theorem \ref{thm-continuous} and \ref{thm-continuous-2}. Before that, we need some results on uniqueness of equilibrium measures and conditional variational principles.

\subsection{Uniqueness of equilibrium measures}\label{subsection-equilbrium}
We first recall from \cite{Barreira1996,Barreira2006,Mummert2006} the notion of nonadditive topological pressure. Consider a dynamical sytem $(X,f)$. Let $\mathcal{U}$ be a finite open cover of $X$. Given $n \in \mathbb{N}$,  denote by $\mathcal{W}_{n}(\mathcal{U})$ the collection of $n$-tuples $U=\left(U_{1} \cdots U_{n}\right)$ with $U_{1}, \ldots, U_{n} \in \mathcal{U}$.  For each $U \in \mathcal{W}_{n}(\mathcal{U})$ we write $m(U)=n$, and define the open set
$$
X(U)=\left\{x \in X: f^{k-1} x \in U_{k} \text { for } k=1, \ldots, m(U)\right\} .
$$
We say that a collection $\Gamma \subset \bigcup_{n \in \mathbb{N}} \mathcal{W}_{n}(\mathcal{U})$ covers the set $X$ if $\bigcup_{U \in \Gamma} X(U) \supset X$. Now let $\Phi=\left(\varphi_{n}\right)_{n}$ be a sequence of continuous functions $\varphi_{n}: X \rightarrow \mathbb{R}$. We define the number
$$
\gamma_{n}(\Phi, \mathcal{U})=\sup \left\{\left|\varphi_{n}(x)-\varphi_{n}(y)\right|: x, y \in X(U) \text { with } U \in \mathcal{W}_{n}(\mathcal{U})\right\}
$$
We assume that
$
\lim _{\operatorname{diam} \mathcal{U} \rightarrow 0} \limsup _{n \rightarrow \infty} \frac{1}{n} \gamma_{n}(\Phi, \mathcal{U})=0.
$
For each $n$-tuple $U \in \mathcal{W}_{n}(\mathcal{U})$ we write $\varphi(U)=\sup _{X(U)} \varphi_{n}$ when $X(U) \neq \varnothing$, and $\varphi(U)=-\infty$ otherwise. We also define
\begin{equation}\label{equation-AD}
	M(\alpha, \Phi, \mathcal{U})=\lim _{n \rightarrow \infty} \inf _{\Gamma} \sum_{U \in \Gamma} \exp (-\alpha m(U)+\varphi(U))
\end{equation}
where the infimum is taken over all collections $\Gamma \subset \bigcup_{k \geq n} \mathcal{W}_{k}(\mathcal{U})$ covering $X$. One can show that the quantity in (\ref{equation-AD}) jumps from $+\infty$ to 0 at a unique value of $\alpha$, and thus we can define
$
P(\Phi, \mathcal{U})=\inf \{\alpha: M(\alpha, \Phi, \mathcal{U})=0\}.
$
Moreover, the limit
$$
P(\Phi)=\lim _{\operatorname{diam} U \rightarrow 0} P(\Phi, \mathcal{U})
$$
exists (see \cite{Barreira1996} for details). The number $P(\Phi)$ is called {\it the nonadditive topological pressure of the sequence of functions} $\Phi$ (with respect to $f$ on $X$).  One can easily verify that if $\Phi$ is the (additive) sequence of functions $\varphi_n=\sum_{k=0}^{n-1} \varphi \circ f^k$, for a given continuous function $\varphi: X \rightarrow \mathbb{R}$, then $P(\Phi)$ coincides with the classical topological pressure of the function $\varphi.$

A sequence of functions $\Phi=\left(\varphi_{n}\right)_{n\in\mathbb{N}}$ is said to be {\it almost additive} (with respect to  $(X,f)$ ) if there is a constant $C>0$ such that for every $n, m \in \mathbb{N}$ we have
$$
-C+\varphi_{n}+\varphi_{m} \circ f^{n} \leqslant \varphi_{n+m} \leqslant C+\varphi_{n}+\varphi_{m} \circ f^{n} .
$$
We denote by $A(f,X)$ the family of almost additive sequences of continuous functions.  It was showed in \cite{FH2010} that almost additive sequences are in fact asymptotically additive.
A measure $\mu\in \mathcal{M}_f(X)$ is said to be an {\it equilibrium measure} associated with $\Phi$ if $P(\Phi)=h_\mu(f)+\lim\limits_{n \rightarrow \infty} \frac{1}{n} \int \varphi_{n} d \mu.$
The uniqueness of equilibrium measures was established  if $\Phi$ has bounded variation, where $\Phi$ is said to have {\it bounded variation} if there exists $\varepsilon>0$ for which
$
\sup _{n \in \mathbb{N}} \gamma_{n}(\Phi, \varepsilon)<\infty
$
with $$\gamma_{n}(\Phi, \varepsilon)=\sup \left\{\left|\varphi_{n}(x)-\varphi_{n}(y)\right|: d\left(f^{k} (x), f^{k} (y)\right)<\varepsilon \text { for } k=0, \ldots, n\right\}.$$ 

\begin{Lem}\cite[Page 294]{Barreira2006} 
	Suppose that $(X,f)$ is an expansive dynamical system satisfying specification property. If $\Phi\in A(f,X)$ has bounded variation, then there is a unique equilibrium measure $\mu_{\Phi}$ for $\Phi$.
\end{Lem}
From \cite[Proposition 23.20]{DGS} it's known that if a homeomorphism of a compact metric space is expansive, mixing and has the shadowing property, then it satisfies specification property. 
\begin{Cor}\label{corollary-B}
	Suppose that $(X,  f)$ is topologically Anosov and mixing. If $\Phi\in A(f,X)$ has bounded variation, then there is a unique equilibrium measure $\mu_{\Phi}$ for $\Phi$.
\end{Cor}
Next, using a spectral decomposition theorem due to Bowen, we will show that Corollary \ref{corollary-B} is still true if $(X,  f)$ is just transitive.
\begin{Thm}\label{theorem-A}
	Suppose that $(X,  f)$ is  topologically Anosov and transitive. If $\Phi\in A(f,X)$ has bounded variation, then there is a unique equilibrium measure $\mu_{\Phi}$ for $\Phi$.
\end{Thm}
\begin{proof}
	Since $(X,  f)$ is expansive, transitive and has the shadowing property,  by \cite[Theorem 3.1.11]{AH} $X$ admits a decomposition
	$
	X=\bigsqcup_{i=0}^{m-1} f^{i}(D)
	$
	where $m>0$ is a positive integer, such that for every $0 \leq i \leq m-1,$ $f^{i}(D)$ is closed $f^{m}$-invariant subsets of $X$, $f^{i}(D)\cap f^{j}(D)=\emptyset$ for any $0\leq i<j\leq m-1,$ and
	$
	f^{m}: f^{i}(D) \rightarrow f^{i}(D)
	$
	is mixing for every $0 \leq i \leq m-1$.   By Lemma \ref{lem-AQ}, $(D,f^m)$ has the shadowing property. Since $(X,  f)$ is expansive, from the uniform continuity of $f, \cdots, f^{m-1},$ $(D,f^m)$ is also expansive. So $(D,f^m)$  is topologically Anosov and mixing. 
	For any $\mu \in \mathcal{M}(X)$, define $\sigma(\mu) \in \mathcal{M}(D)$ by: $$\sigma(\mu)(A)=\mu (A \cup f(A) \cup \dots \cup f^{m-1}(A)),$$ where $A$ ia a Borel set of $D$. By \cite[Proposition 23.17]{DGS}, $\sigma$ is a homeomorphism from $\mathcal{M}_f(X)$ onto $\mathcal{M}_{f^m}(D)$ and $$\sigma^{-1}(\nu)=\frac{1}{m}(\nu + f_{*}\nu +\dots + f^{m-1}_{*}\nu) \in \mathcal{M}_f(X)$$ for any $\nu \in \mathcal{M}_{f^m}(D)$ where $f_*\nu(B)=\nu(f^{-1}(B))$ for any Borel set $B$.
	Note that $$h_{\sigma(\mu)}(f^m)=mh_{\mu}(f),\lim\limits_{n \rightarrow \infty} \frac{1}{n} \int \sum_{i=0}^{m-1}\varphi_{n}\circ f^i d \sigma(\mu)=m\lim\limits_{n \rightarrow \infty} \frac{1}{n} \int \varphi_{n} d \mu.$$ So maximizing $$h_{\sigma(\mu)}(f^m)+\lim\limits_{n \rightarrow \infty} \frac{1}{n} \int \sum_{i=0}^{m-1}\varphi_{n}\circ f^i d \sigma(\mu)$$ is equivalent to maximizing $$h_\mu(f)+\lim\limits_{n \rightarrow \infty} \frac{1}{n} \int \varphi_{n} d \mu.$$
	For $\Phi=\left(\varphi_{n}\right)_{n\in\mathbb{N}}$, define $\Psi=\left(\sum_{i=0}^{m-1}\varphi_{n}\circ f^i\right)_{n\in\mathbb{N}}.$ From the uniform continuity of $f, \cdots, f^{m-1},$ if $\Phi$ is an almost additive sequence of continuous functions with bounded variation, then so does $\Psi.$ Since $(D,f^m)$  is topologically Anosov and mixing, there is a unique equilibrium measure $\nu_{\Psi}$ for $\Psi$ by  Corollary \ref{corollary-B}. So $\sigma^{-1}(\nu_\Psi)$ is the unique equilibrium measure for $\Phi.$
\end{proof}

\begin{Lem}
	Given a two-sided full shift $(\Sigma,\sigma)$. Let $x,y\in \Sigma,$ $n\in\Z^+,$ and $\varepsilon>0.$ If $d(\sigma^i(x),\sigma^i(y))<\varepsilon$ for any $0\leq i\leq n,$ then we have $d(\sigma^i(x),\sigma^i(y))<\varepsilon \cdot2^{-\min\{i,n-i\}}$ for any $0\leq i\leq n.$
\end{Lem}
\begin{proof}
	For the $\varepsilon>0,$ there is an integer $i$ such that $\frac{1}{2^{j+1}}< \varepsilon\leq\frac{1}{2^{j}}.$ Then for any $0\leq i\leq n,$ by $d(\sigma^i(x),\sigma^i(y))<\varepsilon,$ we have $(\sigma^i(x))_l=(\sigma^i(y))_l$ for any $-j\leq l\leq j.$ Thus we have $x_l=y_l$ for any $-j\leq l\leq n+j.$ This implies   $(\sigma^i(x))_l=(\sigma^i(y))_l$ for any $-j-i\leq l\leq n+j-i$ and any $0\leq i\leq n.$ So for any $0\leq i\leq n,$ we have
	$
		d(\sigma^i(x),\sigma^i(y))\leq 2^{-\min\{i+j,n+j-i\}-1}= 2^{-j-1} \cdot 2^{-\min\{i,n-i\}}<\varepsilon\cdot 2^{-\min\{i,n-i\}}.
	$
\end{proof}

Let $\varphi$ a Hölder continuous function  $\Sigma$ with constant $K$ and exponent $\alpha.$ If $x,y\in \Sigma,$ $n\in\Z^+,$ and $\varepsilon>0$ satisfy $d(\sigma^i(x),\sigma^i(y))<\varepsilon$ for any $0\leq i\leq n,$ then we have 
\begin{equation*}
	\begin{split}
		\left|\sum_{k=0}^{n}\varphi(\sigma^k(x))-\sum_{k=0}^{n}\varphi(\sigma^k(y))\right|\leq &\sum_{k=0}^{n}\left|\varphi(\sigma^k(x))-\varphi(\sigma^k(y))\right|\\
		\leq &\sum_{k=0}^{n}K(d(\sigma^k(x),\sigma^k(y)))^\alpha\\
		\leq &\sum_{k=0}^{n}K(\varepsilon \cdot 2^{-\min\{i,n-i\}})^\alpha
		\leq K (\varepsilon)^\alpha\cdot \frac{1}{1-2^{-\alpha}}.
	\end{split}
\end{equation*} 
This implies that every Hölder continuous function on $\Sigma$ has bounded variation. From Lemma \ref{SFT}, every  two-sided subshift of finite type is topologically Anosov. So we have the following corollary.
\begin{Cor}\label{cor-sft-equil}
	Suppose that $(X,  \sigma)$ is a transitive two-sided subshift of finite type. Let $\varphi$ be a Hölder continuous function. Then there is a unique equilibrium measure $\mu_{\varphi}$ for $\varphi$.
\end{Cor}

\subsection{Conditional variational principles}
Let $d \in \mathbb{N}$ and take $(A, B) \in A(f,X)^{d} \times A(f,X)^{d}$.  We consider $(A, B)$ satisfying the following condition:
\begin{equation}\label{equation-BA}
	\liminf _{m \rightarrow \infty} \frac{\psi_{m}^{i}(x)}{m}>0 \quad \text { and } \quad \psi_{n}^{i}(x)>0
\end{equation}
for every $i=1, \ldots, d, x \in X$, and $n \in \mathbb{N}$. Given $a=\left(a_1, \ldots, a_{d}\right) \in \mathbb{R}^{d}$, we define:
$
R_{A,B}(a)=\bigcap_{i=1}^{d}\left\{x \in X: \lim _{n \rightarrow \infty} \frac{\varphi_{n}^{i}(x)}{\psi_{n}^{i}(x)}=a_{i}\right\}.
$
Let $E(f,X) \subseteq A(f,X)$ be the family of sequences with a unique equilibrium measure. 
Denote the sequence of constant functions by $U=(u_n)_{n\in\mathbb{N}}$ with $u_n\equiv n$ for any $n\in\mathbb{N}.$
In \cite{BarreiraDoutor2009} L. Barreira and P. Doutor give the conditional variational principle as following.
\begin{Thm}\cite[Theorem 3]{BarreiraDoutor2009}\label{BarreiraDoutor2009-theorem3}
	Suppose $(X,  f)$ is a dynamical system whose entropy function is upper semi-continuous. Let $d \in \mathbb{N}$ and $(A, B) \in A(f,X)^{d} \times A(f,X)^{d}$ such that
	$
	\text{span}\left\{\Phi^1, \Psi^1,\cdots,\Phi^d,\Psi^d,U\right\} \subseteq E(f,X),
	$ 
	$B$ satisfies \eqref{equation-BA}. 
	If $a \not\in \mathcal{P}_{A,B}(\mathcal{M}_f(X))$, then $R_{A,B}(a)=\emptyset.$ If $a \in \mathrm{Int}(\mathcal{P}_{A,B}(\mathcal{M}_f(X)))$, then $R_{A,B}(a)\neq \emptyset$, and the following properties hold:
	\begin{description}
		\item[(1)] 
		$
		\htop(f,R_{A,B}(a))=\max \left\{h_{\mu}(f): \mu \in \mathcal{M}_f(X) \text { and } \mathcal{P}_{A,B}(\mu)=a\right\}.
		$
		\item[(2)] There is $\mu_{a} \in \mathcal{M}_f^e(X)$ such that $\mathcal{P}_{A,B}\left(\mu_{a}\right)=a$ and
		$
		\htop(f,R_{A,B}(a))=h_{\mu_{a}}(f).
		$
	\end{description}
\end{Thm}

Without using the uniqueness of equilibrium measures, C. Holanda  obtain a conditional variational principle for asymptotically additive families of continuous functions.
\begin{Thm}\label{condition-variation-principle}\cite[Corollary 13]{Holanda}
	Suppose $(X,  f)$ is a dynamical system such that entropy function is upper semi-continuous and $\htop(f,X)<\infty$. 
	Assume that there is a dense subspace $D$ of the space of continuous functions such that every $\varphi \in D$ has a unique equilibrium measure. 
	Let $d \in \mathbb{N}$ and $(A, B) \in AA(f,X)^{d} \times AA(f,X)^{d}$ satisfying \eqref{equation-O}.
	If $a \not\in \mathcal{P}_{A,B}(\mathcal{M}_f(X))$, then $R_{A,B}(a)=\emptyset.$ If $a \in \mathrm{Int}(\mathcal{P}_{A,B}(\mathcal{M}_f(X)))$, then $R_{A,B}(a)\neq \emptyset$, and the following properties hold:
	\begin{description}
		\item[(1)] 
		$
		\htop(f,R_{A,B}(a))=\sup \left\{h_{\mu}(f): \mu \in \mathcal{M}_f(X) \text { and } \mathcal{P}_{A,B}(\mu)=a\right\}.
		$
		\item[(2)]  For any $\varepsilon>0,$ there is $\mu_{a} \in \mathcal{M}_f^e(X)$ such that $\mathcal{P}_{A,B}\left(\mu_{a}\right)=a$ and
		$
		|\htop(f,R_{A,B}(a))-h_{\mu_{a}}(f)|<\varepsilon.
		$
	\end{description}
\end{Thm}
\begin{Rem}
	In fact,  using the work of Climenhaga  \cite[Theorem 3.3]{Climen2013} and Cuneo \cite[Theorem 1.2]{Cuneo2020}, C. Holanda in \cite[Corollary 13]{Holanda} give the proof of Theorem \ref{condition-variation-principle} under the assumption $d=1$ and $(A, B) \in A(f,X)^{d} \times A(f,X)^{d}.$ However, \cite[Theorem 3.3]{Climen2013} is stated for any $d\geq 1$ and \cite[Theorem 1.2]{Cuneo2020} is stated for asymptotically additive  sequence of continuous functions. So one can obtain Theorem \ref{condition-variation-principle} for any $d\geq 1$ and $(A, B) \in AA(f,X)^{d} \times AA(f,X)^{d}.$
\end{Rem}

Consider a  transitive two-sided subshift of finite type $(X,  \sigma)$. Let $D(X)$ be the space of Hölder continuous functions on $X$, then $D(X)$ is a dense subspace of the space of continuous functions on $X$, and every $\varphi \in D(X)$ has a unique equilibrium measure by Corollary \ref{cor-sft-equil}. So the results of Theorem \ref{condition-variation-principle} hold for every transitive two-sided subshift of finite type. And It is easy to see that 
the results of Theorem \ref{condition-variation-principle} are conjugacy invariant. Then we obtain the following results.
\begin{Thm}\label{condition-variation-principle2}
	Suppose $(X,  f)$ is a dynamical system which conjugates to a transitive subshift of ﬁnite type. 
	Let $d \in \mathbb{N}$ and $(A, B) \in AA(f,X)^{d} \times AA(f,X)^{d}$ satisfying \eqref{equation-O}. 
	If $a \not\in \mathcal{P}_{A,B}(\mathcal{M}_f(X))$, then $R_{A,B}(a)=\emptyset.$ If $a \in \mathrm{Int}(\mathcal{P}_{A,B}(\mathcal{M}_f(X)))$, then $R_{A,B}(a)\neq \emptyset$, and the following properties hold:
	\begin{description}
		\item[(1)] $
		\htop(f,R_{A,B}(a))=\sup \left\{h_{\mu}(f): \mu \in \mathcal{M}_f(X) \text { and } \mathcal{P}_{A,B}(\mu)=a\right\}.
		$
		\item[(2)]  For any $\varepsilon>0,$ there is $\mu_{a} \in \mathcal{M}_f^e(X)$ such that $\mathcal{P}_{A,B}\left(\mu_{a}\right)=a$ and
		$
		|\htop(f,R_{A,B}(a))-h_{\mu_{a}}(f)|<\varepsilon.
		$
	\end{description}
\end{Thm}
For a transitive locally maximal hyperbolic set, every Hölder continuous function has a unique equilibrium measure \cite[Example 2]{Bowen}. So we have the following.
\begin{Cor}\label{condition-variation-principle3}
	Suppose that  $(X,f)$ is a system restricted on a transitive locally maximal hyperbolic set. Then the result of Theorem \ref{condition-variation-principle2} holds.
\end{Cor}

\subsection{Proof of Theorem \ref{thm-continuous}}
Now we show that the results of Theorem \ref{thm-Almost-Additive} and \ref{thm-Almost-Additive2} hold for transitive topologically Anosov systems and asymptotically additive sequences. Let $d=1$ in Theorem \ref{thm-almost}, we obain Theorem \ref{thm-continuous}.
\begin{Thm}\label{thm-almost}
	Suppose that $(X,  f)$ is topologically Anosov and transitive. Let $d \in \mathbb{N}$ and $(A, B) \in AA(f,X)^{d} \times AA(f,X)^{d}$ such that
	$B$ satisfies \eqref{equation-O} and $\mathrm{Int}(\mathcal{P}_{A,B}(\mathcal{M}_f(X)))\neq\emptyset$.  Let $\alpha:\mathcal{M}_f(X)\to \mathbb{R}$ be a continuous function satisfying \eqref{equation-W} and \eqref{equation-AF}. 
	Then:
	\begin{description}
		\item[(I)] For any $a\in \mathrm{Int}(\mathcal{P}_{A,B}(\mathcal{M}_f(X))),$ any $\mu\in \mathcal{P}_{A,B}^{-1}(a)$ and any $\eta,  \zeta>0$, there is $\nu\in \mathcal{P}_{A,B}^{-1}(a)\cap\mathcal{M}_f^{e}(X)$ such that $\rho(\nu,\mu)<\zeta$ and $|P(f,\alpha,\nu)-P(f,\alpha,\mu)|<\eta.$
		\item[(II)] For any $a\in \mathrm{Int}(\mathcal{P}_{A,B}(\mathcal{M}_f(X))),$ any $\mu\in \mathcal{P}_{A,B}^{-1}(a)$ with $P(f,\alpha,\mu)\geq \alpha_{A,B}(a),$  any $\alpha_{A,B}(a)\leq P\leq P(f,\alpha,\mu)$ and any $\eta,  \zeta>0$, there is $\nu\in \mathcal{P}_{A,B}^{-1}(a)\cap \mathcal{M}_f^e(X)$ such that $\rho(\nu,\mu)<\zeta$ and $|P(f,\alpha,\nu)-P|<\eta.$ 
		\item[(III)] For any $a\in \mathrm{Int}(\mathcal{P}_{A,B}(\mathcal{M}_f(X)))$ and $\alpha_{A,B}(a)\leq P< H_{A,B}(f,\alpha,a),$
		the set $\{\mu\in \mathcal{P}_{A,B}^{-1}(a)\cap \mathcal{M}_f^e(X):P(f,\alpha,\mu)=P,S_\mu=X\}$ is a dense $G_\delta$ subset of $\{\mu\in \mathcal{P}_{A,B}^{-1}(a):P(f,\alpha,\mu)\geq P\}.$
		\item[(IV)] $\{(\mathcal{P}_{A,B}(\mu), P(f,\alpha,\mu)):\mu\in \mathcal{M}_f(X),\ a=\mathcal{P}_{A,B}(\mu)\in\mathrm{Int}(\mathcal{P}_{A,B}(\mathcal{M}_f(X))),\ \alpha_{A,B}(a)\leq P(f,\alpha,\mu)< H_{A,B}(f,\alpha,a)
		\}$ = $\{(\mathcal{P}_{A,B}(\mu), P(f,\alpha,\mu)):\mu\in \mathcal{M}_f^e(X),\ a=\mathcal{P}_{A,B}(\mu)\in\mathrm{Int}(\mathcal{P}_{A,B}(\mathcal{M}_f(X))), \alpha_{A,B}(a)\leq P(f,\alpha,\mu)< H_{A,B}(f,\alpha,a)\}.$   
	\end{description}
    If further $\alpha=0,$ $(A, B) \in A(f,X)^{d} \times A(f,X)^{d}$ satisfying \eqref{equation-BA} and $\Phi^i, \Psi^i$  have bounded variation for any $1\leq i\leq d,$ then
	\begin{description}
	\item[(V)]   $\{(\mathcal{P}_{A,B}(\mu), h_{\mu}(f)):\mu\in \mathcal{M}_f(X),\ \mathcal{P}_{A,B}(\mu)\in\mathrm{Int}(\mathcal{P}_{A,B}(\mathcal{M}_f(X)))\}=\{(\mathcal{P}_{A,B}(\mu), h_\mu(f)):\mu\in \mathcal{M}_f^e(X),\ \mathcal{P}_{A,B}(\mu)\in\mathrm{Int}(\mathcal{P}_{A,B}(\mathcal{M}_f(X)))\}.$
\end{description}
\end{Thm}
\begin{proof}
	By Theorem \ref{Mainlemma-convex-by-horseshoe} and Theorem \ref{condition-variation-principle2}, $(X,  f)$ satisfies the  locally conditional variational principle with respect to $(\mathcal{M}_f(X),\mathcal{M}_f^e(X))$. By Corollary \ref{Corollary-zero-metric-entropy}(4),
	$\mathcal{M}_f^e(X)$ is a dense $G_\delta$ subset of $\mathcal{M}_f(X)$.
	Since $(X,f)$ is expansive, then the entropy function is upper semi-continuous from \cite[Theorem 8.2]{Walters}.   
	Then we obtain item(I) by Theorem \ref{thm-Almost-Additive}, obtain item(II)-(IV) by Theorem \ref{thm-Almost-Additive2}, Remark \ref{Rem-baire} and Corollary \ref{Corollary-zero-metric-entropy}.
	Finally, if further $\alpha=0,$ then by item(IV) we have  $\{(\mathcal{P}_{A,B}(\mu), h_{\mu}(f)):\mu\in \mathcal{M}_f(X),\ a=\mathcal{P}_{A,B}(\mu)\in\mathrm{Int}(\mathcal{P}_{A,B}(\mathcal{M}_f(X))),\ 0\leq h_{\mu}(f)< H_{A,B}(f,a)\}$ coincides with $\{(\mathcal{P}_{A,B}(\mu), h_{\mu}(f)):\mu\in \mathcal{M}_f^e(X),\ a=\mathcal{P}_{A,B}(\mu)\in\mathrm{Int}(\mathcal{P}_{A,B}(\mathcal{M}_f(X))),\ 0\leq h_{\mu}(f)< H_{A,B}(f,a)\}.$  Combining with Theorem \ref{BarreiraDoutor2009-theorem3} and Theorem \ref{theorem-A}, we obtain item (V).
\end{proof}
\begin{Rem}
	Denote $\mathcal{T}_{A,B,f}(\mu)=(\mathcal{P}_{A,B}(\mu),h_\mu(f))$ for any $\mu\in\mathcal{M}_f(X).$  When $\alpha=0,$ by (\ref{equ-product-inter}),  $$\mathrm{Int}(\mathcal{T}_{A,B,f}(\mathcal{M}_f(X)))=\{(a,h): a\in\mathrm{Int}(\mathcal{P}_{A,B}(\mathcal{M}_f(X))),\ 0< h< H_{A,B}(f,a)\}.$$     By Theorem \ref{thm-almost}(III), we have $\mathrm{Int}(\mathcal{T}_{A,B,f}(\mathcal{M}_f(X)))\subset \mathcal{T}_{A,B,f}(\mathcal{M}_f^e(X)).$ So $\mathrm{Int}(\mathcal{T}_{A,B,f}(\mathcal{M}_f(X)))\subset \mathrm{Int}(\mathcal{T}_{A,B,f}(\mathcal{M}_f^e(X))),$ and thus  $\mathrm{Int}(\mathcal{T}_{A,B,f}(\mathcal{M}_f(X)))= \mathrm{Int}(\mathcal{T}_{A,B,f}(\mathcal{M}_f^e(X))).$
\end{Rem}

From Lemma \ref{LH} and \ref{SFT}, every system restricted on a locally maximal hyperbolic set or a  two-sided subshift of finite type is topologically Anosov. So we have the following corollary.
\begin{Cor}\label{corollary-hyper-aa}
	Suppose that  $(X,f)$ is a system restricted on a transitive locally maximal hyperbolic set or a transitive two-sided subshift of finite type.  Then the results of Theorem \ref{thm-almost} hold.
\end{Cor}

\subsection{Proof of Theorem \ref{thm-continuous-2}} \label{section-applications}
Now we show that  Theorem \ref{thm-continuous-2} holds in the more general context of asymptotically additive sequences. 

Let $f$ be a $C^{1}$ diffeomorphism on a compact Riemannian manifold  $M$ and $p$ be a hyperbolic periodic point.  
We say  $\mu\in \mathcal{M}_f(M)$ is supported on a $p$-horseshoe if there is  a transitive locally maximal hyperbolic set which contains a hyperbolic periodic point $q$ homoclinically related to $p$ such that $\mu\in \mathcal{M}_f(\Lambda).$
We denote $\mathcal{M}(p)$ the set of invariant measures  supported on $p$-horseshoes.

\begin{Thm}\label{thm-almost-2}
	Let $f$ be a $C^{1}$ diffeomorphism on a compact Riemannian manifold  $M$, $p$ be a hyperbolic periodic point, and $C$ be a convex set satisfying $\mathcal{M}(p)\subset C\subset \mathcal{M}_{horse}(p).$  Assume that the entropy function $\mu\to h_\mu(f)$ is upper semi-continuous on $C$. Let $d \in \mathbb{N}$ and $(A, B) \in AA(f,M)^{d} \times AA(f,M)^{d}$ satisfying \eqref{equation-O} and $\mathrm{Int}(\mathcal{P}_{A,B}(C))\neq\emptyset$. Let $\alpha:\mathcal{M}_f(M)\to \mathbb{R}$ be a continuous function satisfying \eqref{equation-W} and \eqref{equation-AF}. 
	Then:
	\begin{description}
		\item[(I)] For any $a\in \mathrm{Int}(\mathcal{P}_{A,B}(C)),$ any $\mu\in \mathcal{P}_{A,B}^{-1}(a)\cap C$ and any $\eta,  \zeta>0$, there is $\nu\in \mathcal{P}_{A,B}^{-1}(a)\cap C\cap  \mathcal{M}_f^e(M)$ such that $\rho(\nu,\mu)<\zeta$ and $|P(f,\alpha,\nu)-P(f,\alpha,\mu)|<\eta.$
		\item[(II)] For any $a\in \mathrm{Int}(\mathcal{P}_{A,B}(C)),$ any $\mu\in \mathcal{P}_{A,B}^{-1}(a)\cap C$ with $P(f,\alpha,\mu)\geq \alpha_{A,B}(a,C),$ any $\alpha_{A,B}(a,C)  \leq P\leq P(f,\alpha,\mu)$ and any $\eta,  \zeta>0$, there is $\nu\in \mathcal{P}_{A,B}^{-1}(a)\cap C\cap  \mathcal{M}_f^e(M)$ such that $\rho(\nu,\mu)<\zeta$ and $|P(f,\alpha,\nu)-P|<\eta.$ 
		\item[(III)] For  any $a\in \mathrm{Int}(\mathcal{P}_{A,B}(C))$ and $\alpha_{A,B}(a,C)\leq P< H_{A,B}(f,\alpha,a,C),$ $\{\mu\in \mathcal{P}_{A,B}^{-1}(a)\cap C\cap  \mathcal{M}_f^e(M):P(f,\alpha,\mu)=P\}$ is a dense $G_\delta$ subset of $\{\mu\in \mathcal{P}_{A,B}^{-1}(a)\cap C:P(f,\alpha,\mu)\geq P\}.$
		\item[(IV)] The set $\{(\mathcal{P}_{A,B}(\mu), P(f,\alpha,\mu)):\mu\in C,\ a=\mathcal{P}_{A,B}(\mu)\in\mathrm{Int}(\mathcal{P}_{A,B}(C)), \alpha_{A,B}(a,C)\leq P(f,\alpha,\mu)< H_{A,B}(f,\alpha,a,C)\}$ coincides with the set  $\{(\mathcal{P}_{A,B}(\mu), P(f,\alpha,\mu)):\mu\in C\cap  \mathcal{M}_f^e(M),\ a=\mathcal{P}_{A,B}(\mu)\in\mathrm{Int}(\mathcal{P}_{A,B}(C)), \alpha_{A,B}(a,C)\leq P(f,\alpha,\mu)< H_{A,B}(f,\alpha,a,C)\}.$ 
	\end{description}
\end{Thm}
\begin{Rem}
	By Lemma \ref{lemma-H(p)}, $\mathcal{M}_{horse}(p)$ is convex. So we obtain Theorem \ref{thm-continuous-2} by applying Theorem \ref{thm-almost-2} to $C=\mathcal{M}_{horse}(p).$
\end{Rem}
There are many transitive systems for which the whole space is a homoclinic class and the entropy function is upper semi-continuous. For these systems we dan define $\mathcal{M}_{horse}(p)$. \\
(i) the nonuniformly hyperbolic diffeomorphisms constructed by Katok \cite{Katok-ex}. For arbitrary compact connected two-dimensional manifold $M$, A. Katok proved that there exists a $C^\infty$ diffeomorphism $f$ such that the Riemannian volume $m$ is an $f$-invariant ergodic hyperbolic measure. From \cite[Theorem S.5.3]{KatHas}) we know that the support of any ergodic and non-atomic hyperbolic measure of a $C^{1+\alpha}$ diffeomorphism is contained in a non-trivial homoclinic class, then there is a hyperbolic periodic point $p$ such that $M=S_m=H(p).$ Moreover, J. Buzzi \cite{Buzzi1997-2} showed that every $C^\infty$ diffeomorphism is asymptotically entropy expansive which implies that the entropy function is upper semi-continuous by \cite[Theorem 20.9]{DGS}.
\\(ii) generic systems in the space of robustly transitive diffeomorphisms $\operatorname{Diff}^{1}_{RT}(M).$ By the robustly transitive partially hyperbolic diffeomorphisms constructed by Ma\~{n}\'{e} \cite{Mane-ex} and the robustly transitive nonpartially hyperbolic diffeomorphisms constructed by Bonatti and Viana \cite{BV-ex}, we know that $\operatorname{Diff}^{1}_{RT}(M)$ is a non-empty open set in $\operatorname{Diff}^{1}(M).$ Since any non-trivial isolated transitive set of $C^{1}$ generic diffeomorphism is a non-trivial homoclinic class \cite{BD1999},  we have that $$\mathcal{R}_1=\{f\in \operatorname{Diff}^{1}_{RT}(M): \text{ there is a hyperbolic periodic point }  p \text{ such that } M=H(p) \}$$ is generic in  $\operatorname{Diff}^{1}_{RT}(M).$ Moreover,  $C^1$ generically in any dimension, isolated homoclinic classes are entropy expansive \cite{PV2008}.  Since  entropy expansive implies upper semi-continuous of the entropy function, then $\mathcal{R}_2=\{f\in \mathcal{R}_1: \text{ the entropy function is upper semi-continuous} \}$ is generic in  $\operatorname{Diff}^{1}_{RT}(M).$\\
(iii) generic systems in the space of volume-preserving diffeomorphisms $\operatorname{Diff}^{1}_{vol}(M).$ Let $M$ be a compact connected Riemannian manifold. Bonatti and Crovisier proved in \cite[Theorem 1.3]{BC2004} that there exists a residual $C^1$-subset $\mathcal{R}_1$ of $\operatorname{Diff}^{1}_{vol}(M)$ such that every $f\in\mathcal{R}_1$ is  transitive. Moreover, by its proof on page 79 and page 87 of \cite{BC2004}, if $f\in\mathcal{R}_1$ then there is a hyperbolic periodic point $p$ such that $M=H(p).$ Since the space of diffeomorphisms away from homoclinic tangencies  $\operatorname{Diff}^{1}(M)\setminus\overline{HT}$ is open in $\operatorname{Diff}^{1}(M),$ then
$\mathcal{R}_2=\mathcal{R}_1\cap \operatorname{Diff}^{1}(M)\setminus\overline{HT}$ is generic in $\operatorname{Diff}^{1}_{vol}(M)\setminus\overline{HT}.$ Moreover, every $C^1$ diffeomorphism away from homoclinic tangencies is entropy expansive \cite{LVY2013}.  Note that  entropy expansive implies upper semi-continuous of the entropy function, if $f\in\mathcal{R}_2$ then  there is a hyperbolic periodic point $p$ such that $M=H(p)$ and the entropy function is upper semi-continuous.

\subsubsection{Some lemmas}
\begin{Lem}\label{lemma-H(p)}
	Let $f$ be a $C^{1}$ diffeomorphism on a compact Riemannian manifold  $M$ and $p$ be a hyperbolic periodic point. Then the set $\mathcal{M}_{horse}(p)$ is convex.
\end{Lem}
\begin{proof}
	Fix $\mu_1,\mu_2\in \mathcal{M}_{horse}(p)$ and $\theta\in [0,1].$ Then for any $\varepsilon>0$ and $i\in\{1,2\},$ there is an  $f$-invariant compact subset $\Lambda_{\varepsilon}^i \subseteq H(p)$ and a  $\mu_\varepsilon^i\in \mathcal{M}_f(\Lambda^i_\varepsilon)$ satisfying the following three properties 
	\begin{description}
		\item[(1)] $\Lambda_\varepsilon^i$ is a transitive locally maximal hyperbolic set which contains a hyperbolic periodic point  $q_i$ homoclinically related to $p.$
		\item[(2)] $\rho(\mu_i,\mu_\varepsilon^i)<\varepsilon.$
		\item[(3)] $h_{\mu_\varepsilon^i}(f)>h_{\mu_i}(f)-\varepsilon.$
	\end{description}
    Then $q_1$ is homoclinically related to $q_2,$ since homoclinically related is an equivalence relation  by \cite[Proposition 2.1]{Newhouse1972}. This implies that there is  a transitive locally maximal hyperbolic set $\Lambda_\varepsilon$ which  contains $\Lambda_\varepsilon^1$ and $\Lambda_\varepsilon^2$  (for example, see \cite[Lemma 8]{Newho1979}). Let $\mu_\varepsilon=\theta\mu_\varepsilon^1+(1-\theta)\mu_\varepsilon^2.$ Then we have 
    \begin{equation*}
    	\rho(\theta\mu_1+(1-\theta)\mu_2,\mu_\varepsilon)\leq \theta\rho(\mu_1,\mu_\varepsilon^1)+(1-\theta)\rho(\mu_2,\mu_\varepsilon^2)<\varepsilon.
    \end{equation*}
    \begin{equation*}
    	h_{\mu_\varepsilon}(f)=\theta h_{\mu_\varepsilon^1}(f)+(1-\theta)h_{\mu_\varepsilon^2}(f)>\theta h_{\mu_1}(f)+(1-\theta)h_{\mu_2}(f)-\varepsilon=h_{\theta\mu_1+(1-\theta)\mu_2}(f)-\varepsilon.
    \end{equation*}
    Note that $\mu_\varepsilon$ is supported on $\Lambda_\varepsilon.$ So $\theta\mu_1+(1-\theta)\mu_2\in \mathcal{M}_{horse}(p)$ and thus $\mathcal{M}_{horse}(p)$ is convex.
\end{proof}

\begin{Thm}\label{def-strong-basic-2}
	Let $f$ be a $C^{1}$ diffeomorphism on a compact Riemannian manifold  $M$, $p$ be a hyperbolic periodic point, and $C$ be a convex set satisfying $\mathcal{M}(p)\subset C\subset \mathcal{M}_{horse}(p).$  Then $(X,  f)$ satisfies the 'multi-horseshoe' entropy-dense property on $C$, that is,   for any positive integer $m,$ any $f$-invariant measures $\{\mu_i\}_{i=1}^m\subseteq C,$ and any $\eta,  \zeta>0$,   there exist compact invariant subsets $\Lambda_i\subseteq\Lambda\subsetneq M$ such that for each $1\leq i\leq m$
	\begin{enumerate}
		\item $\Lambda_i$ and $\Lambda$ are transitive locally maximal hyperbolic sets, and $\Lambda$ contains a hyperbolic periodic point $q$ homoclinically related to $p.$ In particular, one has $\mathcal{M}_f(\Lambda)\subset \mathcal{M}(p)\subset C.$
		\item $\htop(f,  \Lambda_i)>h_{\mu_i}(f)-\eta.$
		\item $d_H(K,  \mathcal{M}_f(\Lambda))<\zeta$,   $d_H(\mu_i,  \mathcal{M}_f(\Lambda_i))<\zeta$, where $K=\cov\{\mu_i\}_{i=1}^m.$
	\end{enumerate}
\end{Thm}
\begin{proof}
	Fix $m>0,$  $K=\cov\{\mu_i\}_{i=1}^m\subseteq C,$ and any $\eta,  \zeta>0.$ Denote $\tau=\frac{1}{2}\min\{\eta,\zeta\}.$ Then for any $1\leq i\leq m$ there is an  $f$-invariant compact subset $\Lambda_{\tau}^i $ and an invariant measure  $\mu_\tau^i\in \mathcal{M}_f(\Lambda_{\tau}^i)$ satisfies the following three properties 
	\begin{description}
		\item[(1)] $\Lambda_\tau^i$ is a transitive locally maximal hyperbolic set which contains a hyperbolic periodic point $q_i$ homoclinically related to $p.$
		\item[(2)] $\rho(\mu_i,\mu_\tau^i)<\tau.$
		\item[(3)] $h_{\mu_\tau^i}(f)>h_{\mu_i}(f)-\tau.$
	\end{description}
     Then $q_i$ is homoclinically related to $q_j,$ since homoclinically related is an equivalence relation  by \cite[Proposition 2.1]{Newhouse1972}. This implies that there is  a transitive locally maximal hyperbolic set $\Lambda_\tau$ which  contains $\bigcup_{i=1}^{m}\Lambda_\tau^i$. By \cite[Proposition 6.4.21]{KatHas}, a  locally maximal hyperbolic set has local product structure. Thus there is $0<\tilde{\tau}<\tau$ such that two periodic points $p_1$ and $p_2$  are homoclinically related if $p_1,p_2\in \Lambda_{\tau}$ satisfy $d(p_1,p_2)<\tilde{\tau}.$ Note that a transitive locally maximal hyperbolic set is expansive by \cite[Corollary 6.4.10]{KatHas} and has shadowing property  by \cite[Theorem 18.1.2]{KatHas}. Then $(\Lambda_\tau,f|_{\Lambda_\tau})$ has the 'multi-horseshoe' entropy-dense property by Theorem \ref{Mainlemma-convex-by-horseshoe}. Thus for the $K_\tau=\cov\{\mu_\tau^i\}_{i=1}^m\subseteq \mathcal{M}_f(\Lambda_\tau),$ and  $\tilde{\tau}>0$,   there exist compact invariant subsets $\Lambda_i\subseteq\Lambda\subsetneq M$ such that for each $1\leq i\leq m$
     \begin{enumerate}
     	\item $(\Lambda_i,f|_{\Lambda_i})$ and $(\Lambda,f|_{\Lambda})$ conjugate to transitive subshifts of ﬁnite type.
     	\item $\htop(f,  \Lambda_i)>h_{\mu_\tau^i}(f)-\tilde{\tau}>h_{\mu_i}(f)-2\tau>h_{\mu_i}(f)-\eta.$
     	\item $d_H(K_\tau,  \mathcal{M}_f(\Lambda))<\tilde{\tau}$,   $d_H(\mu_\tau^i,  \mathcal{M}_f(\Lambda_i))<\tilde{\tau}$.
     	\item There is a positive integer $L$ such that for any $z$ in   $\Lambda$ one has  $f^{j+tL}(z) \in B(q_1,\tilde{\tau})$ for some $0\leq j\leq L-1$ and any $t\in\Z$.
     \end{enumerate}
     From \cite{Anosov2010} any hyperbolic set conjugate to a subshift of finite type is locally maximal. So $\Lambda_i$ and $\Lambda$ are transitive locally maximal hyperbolic sets. By item 3 we have $$d_H(K,  \mathcal{M}_f(\Lambda))<d_H(K,  K_\tau)+d_H(K_\tau,  \mathcal{M}_f(\Lambda))<2\tau<\zeta,$$ $$d_H(\mu_i,  \mathcal{M}_f(\Lambda_i))<d_H(\mu_i,\mu_\tau^i)+d_H(\mu_\tau^i,  \mathcal{M}_f( \Lambda_i))<2\tau<\zeta.$$
     Finally, by item 4 every periodic point in $\Lambda$ is homoclinically related to $q_1,$ and thus is homoclinically related to $p.$
\end{proof}

By Theorem \ref{def-strong-basic-2} and Corollary \ref{Corollary-zero-metric-entropy}(1)(4), we have the following result.
\begin{Lem}\label{lemma-G}
	Let $f$ be a $C^{1}$ diffeomorphism on a compact Riemannian manifold  $M$, $p$ be a hyperbolic periodic point, and $C$ be a convex set satisfying $\mathcal{M}(p)\subset C\subset \mathcal{M}_{horse}(p).$ Then $\{\mu\in C:h_{\mu}(f)=0\}$ is dense in $C,$ $C\cap \mathcal{M}_f^e(M)$ is dense in $C.$
\end{Lem}

\begin{Lem}\label{Lem-baire}
	Let $f$ be a $C^{1}$ diffeomorphism on a compact Riemannian manifold  $M$, $p$ be a hyperbolic periodic point, and $C$ be a convex set satisfying $\mathcal{M}(p)\subset C\subset \mathcal{M}_{horse}(p).$ Assume that $\mu\to h_\mu(f)$ is upper semi-continuous on $C$. Let $d \in \mathbb{N}$ and $(A, B) \in AA(f,M)^{d} \times AA(f,M)^{d}$ satisfying \eqref{equation-O} and $\mathrm{Int}(\mathcal{P}_{A,B}(C))\neq\emptyset$. Let $\alpha:\mathcal{M}_f(M)\to \mathbb{R}$ be a continuous function satisfying \eqref{equation-W} and \eqref{equation-AF}. 
	Then for any $a\in \mathrm{Int}(\mathcal{P}_{A,B}(C))$ and $\alpha_{A,B}(a,C)\leq P< H_{A,B}(f,\alpha,a,C),$ $\{\mu\in \mathcal{P}_{A,B}^{-1}(a)\cap C\cap \mathcal{M}_f^e(M):P(f,\alpha,\mu)=P\}$ is a  dense $G_\delta$  subset of $\{\mu\in \mathcal{P}_{A,B}^{-1}(a)\cap C:P(f,\alpha,\mu)\geq P\}.$
\end{Lem}
\begin{proof}
	Fix $a_0\in \mathrm{Int}(\mathcal{P}_{A,B}(C)),$ $\alpha_{A,B}(a_0,C)\leq P_0< H_{A,B}(f,\alpha,a_0,C),$ $\mu_0\in \mathcal{P}_{A,B}^{-1}(a)\cap C$ with $P(f,\alpha,\mu_0)\geq P_0$ and $ \zeta>0.$ Since $P_0< H_{A,B}(f,\alpha,a_0,C)$, we can pick a $\mu'\in \mathcal{P}_{A,B}^{-1}(a_0)\cap C$ such that $P_0<P(f,\alpha,\mu') \leq  H_{A,B}(f,\alpha,a_0,C)$, and next pick a sufficiently small number $\theta \in(0,1)$ such that $\rho\left(\mu_0, \mu''\right)<\zeta / 2$, where $\mu''=(1-\theta) \mu_0+\theta\mu'.$ Then $\mu''\in \mathcal{P}_{A,B}^{-1}(a_0)\cap C$ and by \eqref{equation-W} we have $P(f,\alpha,\mu'')>P_0.$  Denote $\eta=\frac{P(f,\alpha,\mu'')-P_0}{2}.$
	Since $\alpha$ is continuous, there is $0<\zeta'<\frac{\zeta}{2}$ such that for any $\omega\in \mathcal{M}_f(M)$ with $\rho(\mu'',\omega)<\zeta'$ we have
	\begin{equation}\label{equation-C''}
		|\alpha(\omega)-\alpha(\mu'')|<\frac{\eta}{2}.
	\end{equation} 
	Then by Lemma \ref{Lem-interior} and Lemma \ref{def-strong-basic-2},  there is a compact invariant subset $\Lambda\subset M$ such that
	$a_0\in  \mathrm{Int}(\mathcal{P}_{A,B}(\mathcal{M}_f(\Lambda))),$ $d_H(\mu'',  \mathcal{M}_f(\Lambda))<\zeta',$ $\mathcal{M}_f(\Lambda)\subset C,$ $\Lambda$  is a transitive locally maximal hyperbolic set,
	and  
	\begin{equation}\label{equation-C'''}
		|H_{A,B}(f,a_0,\mathcal{M}_f(\Lambda))- h_{\mu''}(f)|<\frac{\eta}{2}.
	\end{equation} 
    By (\ref{equation-C''}) and (\ref{equation-C'''}), we have $H_{A,B}(f,\alpha,a_0,\mathcal{M}_f(\Lambda))> P(f,\alpha,\mu'')-\eta>P_0.$ By Corollary \ref{corollary-hyper-aa} and Theorem \ref{thm-almost}(III),
   there exists $\nu\in \mathcal{M}_f^e(\Lambda)$ such that $\mathcal{P}_{A,B}(\nu)=a_0$ and $P(f,\alpha,\nu)=P_0.$ Since  $\rho(\mu_0,\nu)<\rho(\mu_0,\mu'')+\rho(\mu'',\nu)<\zeta$, then $\{\mu\in \mathcal{P}_{A,B}^{-1}(a_0)\cap C\cap \mathcal{M}_f^e(M):P(f,\alpha,\mu)=P_0\}$ is dense in  $\{\mu\in \mathcal{P}_{A,B}^{-1}(a_0)\cap C:P(f,\alpha,\mu)\geq P_0\}.$ Finally, by Theorem \ref{def-strong-basic-2} and Corollary \ref{condition-variation-principle3}, $(X,f)$ satisfies the  locally conditional variational principle with respect to $(C,C\cap \mathcal{M}_f^e(M))$, so by Theorem  \ref{thm-Almost-Additive2} and Lemma \ref{lemma-G}, $\{\mu\in \mathcal{P}_{A,B}^{-1}(a_0)\cap C\cap \mathcal{M}_f^e(M):P(f,\alpha,\mu)=P_0\}$ is a   $G_\delta$  subset of $\{\mu\in \mathcal{P}_{A,B}^{-1}(a_0)\cap C:P(f,\alpha,\mu)\geq P_0\}.$
\end{proof}

\subsubsection{Proof of Theorem \ref{thm-almost-2}}
By Theorem \ref{def-strong-basic-2}, Corollary \ref{condition-variation-principle3} and Theorem \ref{condition-variation-principle2}, $(X,f)$ satisfies the  locally conditional variational principle with respect to $(C,C\cap \mathcal{M}_f^e(M))$. By Lemma  \ref{lemma-G},   $C\cap \mathcal{M}_f^e(M)$ is dense in $C.$ By  \cite[Proposition 5.7]{DGM2019}, $\mathcal{M}_f^e(X)$ is a  $G_\delta$ subset of $\mathcal{M}_f(X)$, so $C\cap \mathcal{M}_f^e(M)$ is a dense $G_\delta$ subset of  $C.$ Then we obtain item(I) by Theorem \ref{thm-Almost-Additive}, obtain item(II)-(IV) by Theorem \ref{thm-Almost-Additive2}, Lemma \ref{lemma-G} and Lemma \ref{Lem-baire}.\qed

\subsection{Proof of Theorem \ref{maintheorem-skew}}
Given $N\geq 2$ and  $F\in \operatorname{SP}_{\text {shyp }}^1\left(\Sigma_N \times \mathbb{S}^1\right)$. Now we recall some properties of $F$ from \cite{DGM2017,DGM2019,DGM2022}, and give the proof of Theorem \ref{maintheorem-skew}.
Given a compact $F$-invariant set $\Gamma \subset \Sigma_N \times \mathbb{S}^1$, we say that $\Gamma$ has uniform fiber expansion (contraction) if every ergodic measure $\mu \in \mathcal{M}_{F}^e(\Gamma)$ has a positive (a negative) Lyapunov exponent. It is hyperbolic if it either has uniform fiber expansion or uniform fiber contraction. We say that a set is basic (with respect to $F$ ) if it is compact, $F$-invariant, locally maximal, topologically transitive, and hyperbolic. Basic sets have same properties as transitive locally maximal hyperbolic sets in a differentiable setting.
As used in \cite{DGM2019}, basic sets have same properties as transitive locally maximal hyperbolic sets in a differentiable setting (for example, see every basic set has the speciﬁcation property \cite[Page 369]{DGM2019}, every Hölder continuous potential has a unique equilibrium state for a basic set \cite[Page 373]{DGM2019}.)

\begin{Lem}(\cite[Theorem 1]{DGM2017} or \cite[Proposition 4.4]{DGM2019})\label{Lem-skew-horseshoe}
	Given $\mu\in \mathcal{M}_F^e(\Sigma_N \times \mathbb{S}^1)$ with $\chi(\mu)\geq 0$. Then for every $\eta,\lambda,\zeta>0,$  there is a basic set $\Gamma \subset \Sigma_N \times \mathbb{S}^1$ with uniform fiber expansion such that
	\begin{enumerate}
		\item $h_{top}(\Gamma) >h_\mu(F)-\eta.$
		\item $d_H(\mu,  \mathcal{M}_f(\Gamma))<\zeta$
	\end{enumerate}
   The analogous result holds for any $\chi(\mu)\leq 0.$
\end{Lem}

A periodic point of F is said to be hyperbolic or a saddle if its (ﬁber) Lyapunov exponent is nonzero. We say that two saddles are of the same type if either both have negative exponents or both have positive exponents. Given a saddle $P$ we define the stable and unstable sets of its orbit $\mathcal{O}(P)$ by
$$
W^{\mathrm{s}}(\mathcal{O}(P)) \stackrel{\text { def }}{=}\left\{X: \lim _{n \rightarrow \infty} d\left(F^n(X), \mathcal{O}(P)\right)=0\right\}
$$
and
$$
W^{\mathrm{u}}(\mathcal{O}(P)) \stackrel{\text { def }}{=}\left\{X: \lim _{n \rightarrow \infty} d\left(F^{-n}(X), \mathcal{O}(P)\right)=0\right\}
$$
respectively.
Two saddles $P$ and $Q$ of the same index are homoclinically related if the stable and unstable sets of their orbits intersect cyclically, that is, if
$$
W^{\mathrm{s}}(\mathcal{O}(P)) \cap W^{\mathrm{u}}(\mathcal{O}(Q)) \neq \emptyset \neq W^{\mathrm{s}}(\mathcal{O}(Q)) \cap W^{\mathrm{u}}(\mathcal{O}(P)) .
$$
\begin{Lem}\cite[Lemma 6.4 and 6.5]{DGM2019}\label{Lem-skew-bridge}
	(1) Any pair of saddles $P, Q \in \Sigma_N \times \mathbb{S}^1$ of the same type are homoclinically related.
	
	(2) Consider two basic sets $\Gamma_1, \Gamma_2 \subset \Sigma_N \times \mathbb{S}^1$ of $F$ which are homoclinically related. Then there is a basic set $\Gamma$ of $F$ containing $\Gamma_1 \cup \Gamma_2$.
\end{Lem}

A set  $\mathcal{M} \subset \mathcal{M}_{\text {inv }}(M, f)$ is path connected, if for any $\mu, v \in \mathcal{M}$, there exists a continuous path $\left\{v_t\right\}_{t \in[0,1]}$ in $\mathcal{M}$ such that $v_0=\mu$ and $v_1=v$.
Following the argument of \cite[Corollary 1.4]{YZ2020}, one can obtain the path connectedness of the space of ergodic measures.
\begin{Cor}\label{Cor-skew}
	$\mathcal{M}_F^e(\Sigma_N \times \mathbb{S}^1)$ is path connected.
\end{Cor}
\begin{proof}
	One only needs to show that given a  ergodic measure $\mu$ with positive Lyapunov exponent and a  ergodic measure $v$ with zero Lyapunov exponent, there exists a path connecting them.
	By Lemma \ref{Lem-skew-horseshoe}, there exist sequences of periodic orbits $\left\{p_n\right\}_{n\geq 1}$ and $\left\{q_n\right\}_{n\geq 1}$ of same type such that $\delta_{\mathcal{O}_{p_n}}$ converges to $\mu$ and $\delta_{\mathcal{O}_{q_n}}$ converges to $v$. By Lemma \ref{Lem-skew-bridge}(1), the periodic points $p_n$ and $q_n$ are pairwise homoclinically related. Since $q_{n+1}$ and $q_n$ are homoclinically related, by Lemma \ref{Lem-skew-bridge}(2), there exists a hyperbolic horseshoe $\Lambda_n$ containing $q_{n+1}$ and $q_n$, then by \cite[Theorem B]{Sigmund1977} (see also the comments after it), there exists a path $\left\{v_t\right\}_{t \in\left[1-3^{-n}, 1-3^{-n-1}\right]}$ in the set of ergodic measures supported on $\Lambda_n$ which connects $\delta_{\mathcal{O}_{q_n}}$ to $\delta_{\mathcal{O}_{q_{n+1}}}$. Analogously, in the set of ergodic measures, one has paths $\left\{v_t\right\}_{t \in\left[3^{-n-1}, 3^{-n}\right]}$ which connects $\delta_{\mathcal{O}_{p_{n+1}}}$ to $\delta_{\mathcal{O}_{p_n}}$ and $v_{t \in\left[\frac{1}{3}, \frac{2}{3}\right]}$ which connects $\delta_{\mathcal{O}_{p_1}}$ to $\delta_{\mathcal{O}_{q_1}}$. Let $v_0=\mu$ and $v_1=v$, then $\left\{v_t\right\}_{t \in[0,1]}$ gives a path connecting $\mu$ to $v$.
\end{proof}
Same as \cite[(2.3)]{DGM2019}, we define a  continuous function $\varphi: \Sigma_N \times \mathbb{S}^1 \rightarrow \mathbb{R}$  by
$$
\varphi(X) \stackrel{\text { def }}{=} \log \left|\left(f_{\xi_0}\right)^{\prime}(x)\right| .
$$
where $X=(\xi, x)$.
Then for any $F$-invariant measure $\mu$, 
$
\chi(\mu) = \int \varphi d \mu,
$
and for any $X\in \Sigma_N \times \mathbb{S}^1,$ $\chi(X)=\lim _{n \rightarrow \pm \infty} \frac{1}{n} \sum_{i=0}^{n-1}\varphi(F^i(X)).$
\begin{Lem}\label{Lem-skew-min-max}
	(1) $\max\{\chi(\mu):\mu \in  \mathcal{M}_{F}(\Sigma_N \times \mathbb{S}^1)\}=\max\{\chi(\mu):\mu \in  \mathcal{M}_{F}^e(\Sigma_N \times \mathbb{S}^1)\}=\sup\{a:\mathcal{L}(a) \neq \emptyset\}.$
	(2) $\min\{\chi(\mu):\mu \in  \mathcal{M}_{F}(\Sigma_N \times \mathbb{S}^1)\}=\min\{\chi(\mu):\mu \in  \mathcal{M}_{F}^e(\Sigma_N \times \mathbb{S}^1)\}=\inf\{a:\mathcal{L}(a) \neq \emptyset\}.$
\end{Lem}
\begin{proof}
	One only needs to prove item(1).
	
	Take $\mu_{\max}\in \mathcal{M}_{F}(\Sigma_N \times \mathbb{S}^1)$ such that $\chi(\mu_{\max})=\max\{\chi(\mu):\mu \in  \mathcal{M}_{F}(\Sigma_N \times \mathbb{S}^1)\}.$ Then by ergodic decomposition theorem, there is $\nu_{\max}\in \mathcal{M}_{F}^e(\Sigma_N \times \mathbb{S}^1)$ such that $\chi(\nu_{\max})=\chi(\mu_{\max}).$  It follows $\max\{\chi(\mu):\mu \in  \mathcal{M}_{F}(\Sigma_N \times \mathbb{S}^1)\}= \max\{\chi(\mu):\mu \in  \mathcal{M}_{F}^e(\Sigma_N \times \mathbb{S}^1)\}.$
	
	By Birkhoff's ergodic theorem, there  exists $G_{\nu_{\max}}\subset \Sigma_N \times \mathbb{S}^1$ such that $\nu_{\max}(G_{\nu_{\max}})=1$ and $\chi(X)=\chi(\nu_{\max})$ for any $X\in G_{\nu_{\max}}.$ Then $G_{\nu_{\max}}\subset \mathcal{L}(\chi(\nu_{\max})),$ and thus $\mathcal{L}(\chi(\nu_{\max}))\neq\emptyset.$ It follows
	$\max\{\chi(\mu):\mu \in  \mathcal{M}_{F}^e(\Sigma_N \times \mathbb{S}^1)\}=\chi(\nu_{\max})\leq \sup\{a:\mathcal{L}(a) \neq \emptyset\}.$
	
	For any $a_0$ with $\mathcal{L}(a_0)\neq \emptyset,$ take $X_0\in \mathcal{L}(a_0).$ Then $\chi(X_0)=\lim _{n \rightarrow \pm \infty} \frac{1}{n} \sum_{i=0}^{n-1}\varphi(F^i(X_0))=a_0.$ It implies $\chi(\mu_0)=a_0$ for each limit point $\mu_0$ of $\{\frac{1}{n}\sum_{i=0}^{n-1}\delta_{F^i(X_0)}\}_{n\geq 1}.$ Thus $\max\{\chi(\mu):\mu \in  \mathcal{M}_{F}(\Sigma_N \times \mathbb{S}^1)\}\geq a_0.$ By the arbitrariness of $a_0,$  we obtain $\max\{\chi(\mu):\mu \in  \mathcal{M}_{F}(\Sigma_N \times \mathbb{S}^1)\}\geq\sup\{a:\mathcal{L}(a) \neq \emptyset\}.$
\end{proof}

By Lemma \ref{Lem-skew-horseshoe}, there exists a basic set $\Gamma^+$ which has uniform fiber expansion. Take a periodic point $P^+\in \Gamma^+.$ 
Then $P^+$ has positive Lyapunov exponent. 
Recall the definition of  $\mathcal{M}_{horse}(P^+)$ from section \ref{section-nonhyperblic}, $\mathcal{M}_{horse}(P^+)$ is the set of invariant measures  which can be approximated by $P^+$-horseshoes,
By Lemma \ref{Lem-skew-bridge}(1), every basic set with uniform fiber expansion contains periodic points  homoclinically related to $P^+.$
Then by  Lemma \ref{Lem-skew-horseshoe}, we have 
\begin{equation}\label{equation-skew-proof}
	\{\mu\in \mathcal{M}_F^e(\Sigma_N \times \mathbb{S}^1):\chi(\mu)\geq 0\}\subset \mathcal{M}_{horse}(P^+).
\end{equation}
By \cite[Page 77]{DGM2022}, the entropy function $\mu\to h_\mu(f)$ is upper semi-continuous on $\mathcal{M}_F(\Sigma_N \times \mathbb{S}^1)$.  So the results of Theorem \ref{thm-continuous-2} holds for $\mathcal{M}_{horse}(P^+).$

By Corollary \ref{Cor-skew}, $\mathcal{P}_\varphi(\mathcal{M}_F^e(\Sigma_N \times \mathbb{S}^1))$ is an interval.  Combining with Lemma   \ref{Lem-skew-min-max}, $\mathcal{P}_\varphi(\mathcal{M}_F^e(\Sigma_N \times \mathbb{S}^1))=[a_{\min},a_{\max}].$
Then by (\ref{equation-skew-proof}), we have $(0,a_{\max })\subset \mathrm{Int}(\mathcal{P}_\varphi(\mathcal{M}_{horse}(P^+))).$  Apply Theorem \ref{thm-continuous-2}(III) to $\varphi$, for  every $a\in \left(0, a_{\max }\right)$ we have
\begin{equation}\label{equation-skew-proof2}
    \begin{split}
    		[0,\sup\mathcal{E}_F(\mathcal{P}_\varphi^{-1}(a)\cap \mathcal{M}_{horse}(P^+)))&\subset  \left\{h_\mu(F): \mu \in \mathcal{M}_{horse}^e(P^+), \chi(\mu)=a\right\} \\
    	&=  \left\{h_\mu(F): \mu \in \mathcal{M}_{horse}(P^+)\cap \mathcal{M}_F^e(\Sigma_N \times \mathbb{S}^1), \chi(\mu)=a\right\} \\
    	&\subset   \left\{h_\mu(F): \mu \in  \mathcal{M}_F^e(\Sigma_N \times \mathbb{S}^1), \chi(\mu)=a\right\}.
    \end{split}
\end{equation}
By (\ref{equation-skew-proof})
\begin{equation*}
	\begin{split}
		\sup \{h_\mu(F): \mu \in \mathcal{M}_{F}^e(\Sigma_N \times \mathbb{S}^1),\chi(\mu)=a\}
		&\leq  \sup \{h_\mu(F): \mu\in \mathcal{M}_{horse}(P^+),\chi(\mu)=a\}\\
		&=\sup\mathcal{E}_F(\mathcal{P}_\varphi^{-1}(a)\cap \mathcal{M}_{horse}(P^+)).
	\end{split}
\end{equation*}
Combining with (\ref{equation-skew-proof2}), we have
$$
[0,H(f,\chi,a))\subset  \left\{h_\mu(F): \mu \in \mathcal{M}_{F}^e(\Sigma_N \times \mathbb{S}^1), \chi(\mu)=a\right\} \subset [0,H(f,\chi,a)],
$$
where $H(f,\chi,a)=\sup \left\{h_\mu(F): \mu \in \mathcal{M}_{F}^e(\Sigma_N \times \mathbb{S}^1),\chi(\mu)=a\right\}.$

Finally, by Theorem \ref{Thm-skew}$$
h_{top}(\mathcal{L}(a))=\sup \{h_\mu(F): \mu \in \mathcal{M}_{F}^e(\Sigma_N \times \mathbb{S}^1),\chi(\mu)=a\}.
$$
So we have
$$
[0,h_{top}(\mathcal{L}(a)))\subset  \left\{h_\mu(F): \mu \in \mathcal{M}_{F}^e(\Sigma_N \times \mathbb{S}^1), \chi(\mu)=a\right\} \subset [0,h_{top}(\mathcal{L}(a)].
$$
The analogous result holds for any $a\in (a_{\min },0).$\qed 

\subsection{Proof of Theorem \ref{maintheorem-cocycle}}
Note that the projective line $\mathbb{P}^1$ is topologically the circle $\mathbb{S}^1$ and the action of any $\mathrm{\operatorname{SL}}(2, \mathbb{R})$ matrix on $\mathbb{P}^1$ is a diffeomorphism. Given a matrix $A \in \operatorname{SL}(2, \mathbb{R})$, define $f_A: \mathbb{P}^1 \rightarrow \mathbb{P}^1$ by
$
f_A(v) {=} \frac{A v}{\|A v\|} .
$
Given a one-step $2 \times 2$ matrix cocycle $\mathbf{A}$, we denote by $F_{\mathbf{A}}$ the associated step skewproduct generated by the family of maps $f_{A_0}, \ldots, f_{A_{N-1}}$ as in (\ref{equation-step-skew}). 

Now we fix $N\geq 2$ and  $\mathbf{A}\in \mathfrak{E}_{N \text {,shyp }}$. By \cite[Proposition 11.23]{DGM2019}, we have $F_{\mathbf{A}}\in \operatorname{SP}_{\text {shyp }}^1\left(\Sigma_N \times \mathbb{S}^1\right).$ Denote $a_{\max }=\frac{1}{2}\sup\{a:\mathcal{L}(a) \neq \emptyset\}.$ 
\begin{Lem}\cite[Theorem 5]{DGM2019}
	For every $a \in\left[0, a_{\max }\right],$ one has $\mathcal{L}_{\mathbf{A}}^{+}(a) \neq \emptyset$ and
	$$
	h_{top}(\mathcal{L}_{\mathbf{A}}^{+}(a))=h_{top}(\mathcal{L}(2a)).
	$$
\end{Lem}
Then by Theorem \ref{Thm-skew} we have $h_{top}(\mathcal{L}_{\mathbf{A}}^{+}(a))=h_{top}(\mathcal{L}(2a))=\sup \{h_\mu(F_{\mathbf{A}}): \mu \in \mathcal{M}_{F_{\mathbf{A}}}^e(\Sigma_N \times \mathbb{S}^1),\chi(\mu)=2a\}.$ Fix  $a\in (0,\max)$ and $0\leq h<h_{top}(\mathcal{L}_{\mathbf{A}}^{+}(a)).$ Then by Theorem \ref{maintheorem-skew} there exists $\mu_{a,h}\in  \mathcal{M}_{F_{\mathbf{A}}}^e(\Sigma_N \times \mathbb{S}^1)$ such that $\chi(\mu_{a,h})=2a$ and $h_{\mu_{a,h}}(F_{\mathbf{A}})=h,$ 

Given $\xi^{+} \in \Sigma_N^{+}$ and $\ell \in \mathbb{N}$, denote by $v_{+}(\xi^{+}, \ell) \in \mathbb{P}^1$ a vector at which $|(f_{\xi^{+}}^{\ell})^{\prime}|$ attains its maximum; note that this vector is unique unless $f_{\xi^{+}}^{\ell}$ is an isometry.
\begin{Lem}\cite[Proposition 11.5]{DGM2019}\label{Lem-cocycle}
	Assume $\xi^{+} \in \Sigma_N^{+}$ satisfies $\lambda_1\left(\mathbf{A}, \xi^{+}\right)=a$.
	\begin{enumerate}
		\item If $a=0$, then $\chi^{+}\left(\xi^{+}, v\right)=0$ for all $v \in \mathbb{P}^1$, where $
		\chi^{+}\left(\xi^{+}, v\right) {=} \lim _{n \rightarrow \infty} \frac{1}{n} \log |(f_{\xi^{+}}^n)^{\prime}(v)|.
		$
		\item If $a>0$, then the limit $v_0\left(\xi^{+}\right)=\lim _{\ell \rightarrow \infty} v_{+}\left(\xi^{+}, \ell\right)$ exists and it holds
		$$
		\chi^{+}\left(\xi^{+}, v\right)= \begin{cases}2 a & \text { for } v=v_0\left(\xi^{+}\right) \\ -2 a & \text { otherwise }\end{cases}
		$$
	\end{enumerate}
\end{Lem}
A sequence $\xi=\left(\ldots \xi_{-1} \cdot \xi_0 \xi_1 \ldots\right) \in \Sigma_N$ can be written as $\xi=\xi^{-} . \xi^{+}$, where $\xi^{+} \in \Sigma_N^{+} {=}\{0, \ldots, N-1\}^{\mathbb{N}_0}$ and $\xi^{-} \in \Sigma_N^{-} {=}\{0, \ldots, N-1\}^{-\mathbb{N}}$.
Consider the projections $\pi^{+}: \Sigma_N \rightarrow \Sigma_N^{+}$, $\pi^{+}\left(\xi^{-} . \xi^{+}\right)=\xi^{+}$, and $\pi_1: \Sigma_N \times \mathbb{P}^1 \rightarrow \Sigma_N, \pi_1(\xi, x)=\xi$.
\begin{Lem}
	Let $\nu_{a,h}{=}\left(\pi^{+} \circ \pi_1\right)_* \mu_{a,h}\in \mathcal{M}_{\sigma^+}^e(\Sigma_N^{+})$. Then $\lambda_1(\mathbf{A}, \nu_{a,h})=a$. 
\end{Lem}
\begin{proof}
	By ergodicity, $\chi(\xi, v)=2a$ for $\mu_{a,h}$-almost every $(\xi, v)$. Denote by $\mu^{+}_{a,h}$ the ergodic measure obtained as the push-forward of $\mu_{a,h}$ by the map $(\xi, v) \mapsto\left(\xi^{+}, v\right)$. Hence for $\mu^{+}_{a,h}$-almost every $\left(\xi^{+}, v\right)$ it holds $\chi^{+}\left(\xi^{+}, v\right)=2a$. It follows from Lemma \ref{Lem-cocycle} that $\lambda_1\left(\mathbf{A}, \xi^{+}\right)=a$ for $\nu_{a,h}$-almost every $\xi^{+}$. Note that $v^{+}$ is ergodic. Hence, by the subadditive ergodic theorem, the claim follows.
\end{proof}
Finally, by \cite[Page 82]{DGM2022} we have $h_{\nu_{a,h}}(\mathbf{A})=h_{\mu_{a,h}}(F_\mathbf{A})=h.$
\qed

\subsection{Proof of Theorem \ref{maintheorem-robust}}
In the proof of Theorem \ref{maintheorem-skew}, we mainly use the result of horseshoe's approximation given in \cite{DGM2017}. By  \cite[Theorem 1.1]{YZ2020}, every $f\in \mathcal{V}(M)$ also has the result of horseshoe's approximation. So we can  prove Theorem \ref{maintheorem-robust} by same method of \ref{maintheorem-skew}. We omit the proof.

\section{Proofs of Theorem \ref{thm-inter-huasdorff}, \ref{thm-Lyapunov} and \ref{thm-first-return}}\label{section-inter-huasdorff}
\subsection{Two lemmas}
Given  $(X,f),$ $d \in \mathbb{N}$ and $\Phi_d=\{\varphi_i\}_{i=1}^{d}\subset C(X)$.  Denote $\mathcal{P}_{\Phi_d}(\mu)=(\int \varphi d\mu,\dots,\int \varphi_d d\mu)$ for any $\mu\in\mathcal{M}_f(X).$  Then  $\mathcal{P}_{\Phi_d}(\mathcal{M}_f(X))$ is a convex compact subset of $\mathbb{R}^d.$  Denote $\mathcal{T}_{\Phi_d,f}(\mu)=(\mathcal{P}_{\Phi_d}(\mu),h_\mu(f))$ for any $\mu\in\mathcal{M}_f(X).$  Then $\mathcal{T}_{\Phi_d,f}$ is affine, and thus $\mathcal{T}_{\Phi_d,f}(\mathcal{M}_f(X))$ is a convex subset of $\mathbb{R}^{d+1}.$  
\begin{Lem}\label{lemma-relative}
	Suppose that $(X,  f)$ is topologically Anosov and transitive. Let $d \in \mathbb{N}$ and $\Phi_d=\{\varphi_i\}_{i=1}^{d}\subset C(X)$.  Assume that $\mathcal{Q}:\mathbb{R}^{d+1}\to \mathbb{R}$ is a continuous function.
	Then we have 
	\begin{enumerate}
		\item[(a)]  $\mathrm { relint } (\mathcal{P}_{\Phi_d}(\mathcal{M}_f(X)))=\mathrm { relint } (\mathcal{P}_{\Phi_d}(\mathcal{M}_f^e(X)))$ and for any $a\in \mathrm { relint } (\mathcal{P}_{\Phi_d}(\mathcal{M}_f(X))),$ there are $\mu_0,\mu_+\in \mathcal{M}_f^e(X)\cap\mathcal{P}_{\Phi_d}^{-1}(a)$ such that $h_{\mu_0}(f)=0,h_{\mu_+}(f)>0.$
		\item[(b)] $\mathrm { relint } (\mathcal{T}_{\Phi_d,f}(\mathcal{M}_f(X)))=\mathrm { relint } (\mathcal{T}_{\Phi_d,f}(\mathcal{M}_f^e(X))).$
		\item[(c)] $\mathcal{Q}(\mathcal{T}_{\Phi_d,f}(\mathcal{M}_f^e(X)))$ and $\mathcal{T}_{\Phi_d,f}(\{\mathcal{M}_f^e(X):h_\mu(f)>0\})$ are two intervals, and $\overline{\mathcal{Q}(\mathcal{T}_{\Phi_d,f}(\mathcal{M}_f^e(X)))}=\overline{\mathcal{T}_{\Phi_d,f}(\{\mathcal{M}_f^e(X):h_\mu(f)>0\}}=\overline{\mathcal{Q}(\mathcal{T}_{\Phi_d,f}(\mathcal{M}_f(X))))}$.
	\end{enumerate}
\end{Lem}
\begin{proof}
	First, we give the proof of item(a) and (b). 
	
	Since $\mathcal{P}_{\Phi_d}(\mathcal{M}_f(X))$ is a subset of $\mathbb{R}^d,$ we have $\dim_{aff}(\mathcal{P}_{\Phi_d}(\mathcal{M}_f(X)))\in \{0,1,\dots,d\}$.
	
	(1) When $\dim_{aff}(\mathcal{P}_{\Phi_d}(\mathcal{M}_f(X)))=0$, there exists $c\in\mathbb{R}^d$ such that $\mathcal{P}_{\Phi_d}(\mu)=c$ for any $\mu\in\mathcal{M}_f(M).$ Then $\mathrm { relint } (\mathcal{P}_{\Phi_d}(\mathcal{M}_f(X)))=\mathrm { relint } (\mathcal{P}_{\Phi_d}(\mathcal{M}_f^e(X)))=\{c\}$ and $\mathcal{T}_{\Phi_d,f}(\mathcal{M}_f(X))=\{(c,h_\mu(f)):\mu\in\mathcal{M}_f(M)\}$. By Lemma \ref{Lemma-inter-entropy}, we have $[0,h_{top}(f)]=\{h_{\mu}(f):\mu\in\mathcal{M}_f(M)\}=\{h_{\mu}(f):\mu\in\mathcal{M}_f^e(M)\}.$ Hence, $\mathcal{T}_{\Phi_d,f}(\mathcal{M}_f(X))=\mathcal{T}_{\Phi_d,f}(\mathcal{M}_f^e(X))=\{(c,h):0\leq h \leq h_{top}(f)\}.$ So there are $\mu_0,\mu_+\in \mathcal{M}_f^e(X)\cap\mathcal{P}_{\Phi_d}^{-1}(c)$ such that $h_{\mu_0}(f)=0,h_{\mu_+}(f)>0$ and $\mathrm { relint } (\mathcal{T}_{\Phi_d,f}(\mathcal{M}_f(X)))=\mathrm { relint } (\mathcal{T}_{\Phi_d,f}(\mathcal{M}_f^e(X))).$

	(2) When $\dim_{aff}(\mathcal{P}_{\Phi_d}(\mathcal{M}_f(X)))=d$, we have $\mathrm{relint}(\mathcal{P}_{\Phi_d}(\mathcal{M}_f(X)))=\mathrm{Int}(\mathcal{P}_{\Phi_d}(\mathcal{M}_f(X))).$
	From item(III) of Theorem \ref{thm-almost},  $\mathrm { Int } (\mathcal{P}_{\Phi_d}(\mathcal{M}_f(X)))=\mathrm { Int } (\mathcal{P}_{\Phi_d}(\mathcal{M}_f^e(X)))$ and for any $a\in \mathrm { Int } (\mathcal{P}_{\Phi_d}(\mathcal{M}_f(X)))$ we have $\{(a,h):0\leq h< H_{\Phi_d}(f,a)\}\subset \mathcal{T}_{\Phi_d,f}(\mathcal{M}_f^e(X)),$ where $H_{\Phi_d}(f,a)=\sup\{h_\mu(f):\mu\in \mathcal{P}_{\Phi_d}^{-1}(a)\}.$  Since $\{\mu\in\mathcal{M}_f(X):h_\mu(f)>0\}$ is convex, then by Corollary \ref{Corollary-zero-metric-entropy}(2) and Lemma \ref{lemma-GG} we have $H_{\Phi_d}(f,a)>0.$ So there are $\mu_0,\mu_+\in \mathcal{M}_f^e(X)\cap\mathcal{P}_{\Phi_d}^{-1}(a)$ such that $h_{\mu_0}(f)=0,h_{\mu_+}(f)>0.$

	By (\ref{equ-product-inter}), we have $\mathrm{Int}(\mathcal{T}_{\Phi_d,f}(\mathcal{M}_f(X)))=\{(a,h): a\in\mathrm{Int}(\mathcal{P}_{\Phi_d}(\mathcal{M}_f(X))),\ 0< h< H_{\Phi_d}(f,a)\}.$     From Theorem \ref{thm-almost}(IV), we have $\mathrm{Int}(\mathcal{T}_{\Phi_d,f}(\mathcal{M}_f(X)))\subset \mathcal{T}_{\Phi_d,f}(\mathcal{M}_f^e(X)).$ This implies $\mathrm{Int}(\mathcal{T}_{\Phi_d,f}(\mathcal{M}_f(X)))=\mathrm{Int}(\mathcal{T}_{\Phi_d,f}(\mathcal{M}_f^e(X))).$ So $\mathrm { relint } (\mathcal{T}_{\Phi_d,f}(\mathcal{M}_f(X)))=\mathrm { relint } (\mathcal{T}_{\Phi_d,f}(\mathcal{M}_f^e(X))).$

	(3) When $0<\dim_{aff}(\mathcal{P}_{\Phi_d}(\mathcal{M}_f(X)))<d$, there is $I_{aff}\subset\{1,\dots,d\}$ such that $\pi_{aff}(\mathcal{P}_{\Phi_d}(\mathcal{M}_f(X)))=\mathcal{P}_{\{\varphi\}_{i\in I_{aff}}}(\mathcal{M}_f(X)).$ Denote $\Phi_{aff}=\{\varphi_i\}_{i\in I_{aff}}.$ Then  $\pi_{aff}(\mathrm { relint } (\mathcal{P}_{\Phi_d}(\mathcal{M}_f(X)))=\mathrm{Int}(\mathcal{P}_{\Phi_{aff}}(\mathcal{M}_f(X))).$ 
	From  Theorem \ref{thm-almost}(III), $ \mathrm{Int}(\mathcal{P}_{\Phi_{aff}}(\mathcal{M}_f(X)))= \mathrm{Int}(\mathcal{P}_{\Phi_{aff}}(\mathcal{M}_f^e(X)))$ and for any $a_{aff}\in \mathrm{Int}(\mathcal{P}_{\Phi_{aff}}(\mathcal{M}_f(X)))$ we have
	$\{(a_{aff},h):0\leq h< H_{\Phi_{aff}}(f,a_{aff})\}\subset \mathcal{T}_{\Phi_{aff},f}(\mathcal{M}_f^e(X)),$
	By Corollary \ref{Corollary-zero-metric-entropy}(2) and Lemma \ref{lemma-GG} we have $H_{\Phi_{aff}}(f,a_{aff})>0.$ 
	Then we have  $\mathrm { relint } (\mathcal{P}_{\Phi_d}(\mathcal{M}_f(X)))=\mathrm { relint } (\mathcal{P}_{\Phi_d}(\mathcal{M}_f^e(X)))$, $H_{\Phi_d}(f,a)>0$ and
	\begin{equation}\label{equ-relative}
	\{(a,h):0\leq h< H_{\Phi_{d}}(f,a)\}\subset\mathcal{T}_{\Phi_{d},f}(\mathcal{M}_f^e(X))
	\end{equation} 
	for any $a\in \mathrm { relint } (\mathcal{P}_{\Phi_d}(\mathcal{M}_f(X))).$ Hence, there are $\mu_0,\mu_+\in \mathcal{M}_f^e(X)\cap\mathcal{P}_{\Phi_d}^{-1}(a)$ such that $h_{\mu_0}(f)=0,h_{\mu_+}(f)>0.$ By (\ref{equ-product-inter}), 
	$\mathrm{relint}(\mathcal{T}_{\Phi_d,f}(\mathcal{M}_f(X)))=\{(a,h):a\in\mathrm{relint}(\mathcal{P}_{\Phi_d}(\mathcal{M}_f(X))),\ 0< h< H_{\Phi_d}(f,a)\}.$ So from (\ref{equ-relative}) we have $\mathrm{relint}(\mathcal{T}_{\Phi_d,f}(\mathcal{M}_f(X)))\subset \mathcal{T}_{\Phi_d,f}(\mathcal{M}_f^e(X)),$ and thus $\mathrm { relint } (\mathcal{T}_{\Phi_d,f}(\mathcal{M}_f(X)))=\mathrm { relint } (\mathcal{T}_{\Phi_d,f}(\mathcal{M}_f^e(X))).$
	
	 Finally, we prove item(c).
	Since $\mathcal{T}_{\Phi_d,f}(\mathcal{M}_f(X))$ is convex, by (\ref{equ-close-relint}) we have $\mathcal{T}_{\Phi_d,f}(\mathcal{M}_f(X))\subset \overline{\mathrm { relint } (\mathcal{T}_{\Phi_d,f}(\mathcal{M}_f(X)))}.$ Then by item(b), $\mathrm { relint } (\mathcal{T}_{\Phi_d,f}(\mathcal{M}_f^e(X)))=\mathrm { relint } (\mathcal{T}_{\Phi_d,f}(\mathcal{M}_f(X)))$ is convex, and
	$$\mathcal{T}_{\Phi_d,f}(\mathcal{M}_f^e(X))\subset \mathcal{T}_{\Phi_d,f}(\mathcal{M}_f(X))\subset \overline{\mathrm { relint } (\mathcal{T}_{\Phi_d,f}(\mathcal{M}_f(X)))}= \overline{\mathrm { relint } (\mathcal{T}_{\Phi_d,f}(\mathcal{M}_f^e(X)))}.$$
	Since $\mathcal{Q}$ is continuous, then $\mathcal{Q}(\mathrm { relint } (\mathcal{T}_{\Phi_d,f}(\mathcal{M}_f^e(X))))$ is an  interval and $$\mathcal{Q}(\mathrm{relint}(\mathcal{T}_{\Phi_d,f}(\mathcal{M}_f^e(X))))\subseteq\mathcal{Q}(\mathcal{T}_{\Phi_d,f}(\mathcal{M}_f^e(X)))\subseteq \overline{\mathcal{Q}(\mathrm{relint}(\mathcal{T}_{\Phi_d,f}(\mathcal{M}_f^e(X))))}.$$
	This implies $\mathcal{Q}(\mathcal{T}_{\Phi_d,f}(\mathcal{M}_f^e(X)))$ is an interval, and $\overline{\mathcal{Q}(\mathcal{T}_{\Phi_d,f}(\mathcal{M}_f^e(X)))}=\overline{\mathcal{Q}(\mathrm{relint}(\mathcal{T}_{\Phi_d,f}(\mathcal{M}_f^e(X))))}=\overline{\mathcal{Q}(\mathcal{T}_{\Phi_d,f}(\mathcal{M}_f(X))))}$.
	
	Since $\mathrm{relint}(\mathcal{T}_{\Phi_d,f}(\mathcal{M}_f^e(X)))=\mathrm{relint}(\mathcal{T}_{\Phi_d,f}(\mathcal{M}_f(X)))=\{(a,h):a\in\mathrm{relint}(\mathcal{P}_{\Phi_d}(\mathcal{M}_f(X))),\ 0< h< H_{\Phi_d}(f,a)\},$ then we have $\mathrm{relint}(\mathcal{T}_{\Phi_d,f}(\mathcal{M}_f^e(X)))\subset \mathcal{T}_{\Phi_d,f}(\{\mathcal{M}_f^e(X):h_\mu(f)>0\})\subset \mathcal{Q}(\mathcal{T}_{\Phi_d,f}(\mathcal{M}_f^e(\Lambda)))\subseteq \overline{\mathcal{Q}(\mathrm{relint}(\mathcal{T}_{\Phi_d,f}(\mathcal{M}_f^e(\Lambda))))}.$ This implies  $\mathcal{T}_{\Phi_d,f}(\{\mathcal{M}_f^e(X):h_\mu(f)>0\})$ is also an interval, and we have $\overline{\mathcal{T}_{\Phi_d,f}(\{\mathcal{M}_f^e(X):h_\mu(f)>0\}}=\overline{\mathcal{Q}(\mathrm{relint}(\mathcal{T}_{\Phi_d,f}(\mathcal{M}_f^e(\Lambda))))}=\overline{\mathcal{Q}(\mathcal{T}_{\Phi_d,f}(\mathcal{M}_f(X))))}$.
\end{proof}

\begin{Lem}\label{lemma-relative-2}
	Suppose that $(X,  f)$ is topologically Anosov and transitive. Let $d_1,d_2\in \mathbb{N}$ and $\Phi_{d_1}=\{\varphi_i\}_{i=1}^{d_1},\Phi_{d_2}=\{\varphi_i\}_{i=d_1}^{d_2}\subset C(X)$.  Assume that $\mathcal{Q}:\mathbb{R}^{d_2+1}\to \mathbb{R}$ is a continuous function.
	Then for any $a_1\in \mathrm { relint } (\mathcal{P}_{\Phi_{d_1}}(\mathcal{M}_f(X))),$ we have
	\begin{enumerate}
		\item[(a)]  there are $\mu_0,\mu_+\in \mathcal{M}_f^e(X)\cap\mathcal{P}_{\Phi_{d_1}}^{-1}(a_1)$ such that $h_{\mu_0}(f)=0,h_{\mu_+}(f)>0.$
		\item[(b)] $\mathrm { relint } (\{\mathcal{T}_{\Phi_{d_2},f}(\mu):\mu\in \mathcal{P}_{\Phi_{d_1}}^{-1}(a_1)\})=\mathrm { relint } (\{\mathcal{T}_{\Phi_{d_2},f}(\mu):\mu\in \mathcal{M}_f^e(X)\cap  \mathcal{P}_{\Phi_{d_1}}^{-1}(a_1)\}).$
		\item[(c)] $\mathcal{Q}(\{\mathcal{T}_{\Phi_{d_2},f}(\mu):\mu\in \mathcal{M}_f^e(X)\cap   \mathcal{P}_{\Phi_{d_1}}^{-1}(a_1)\})$ is an interval,  $$\overline{\mathcal{Q}(\{\mathcal{T}_{\Phi_{d_2},f}(\mu):\mu\in \mathcal{M}_f^e(X)\cap   \mathcal{P}_{\Phi_{d_1}}^{-1}(a_1)\})}=\mathcal{Q}(\{\mathcal{T}_{\Phi_{d_2},f}(\mu):\mu\in \mathcal{M}_f(X)\cap   \mathcal{P}_{\Phi_{d_1}}^{-1}(a_1)\}).$$
	\end{enumerate}
\end{Lem}
\begin{proof}
	Take $a_2\in \mathrm { relint } (\{\mathcal{P}_{\Phi_{d_2}}(\mu):\mu\in \mathcal{P}_{\Phi_{d_1}}^{-1}(a_1)\})$, then we have $(a_1,a_2)\in \mathrm { relint } (\mathcal{P}_{\Phi_{d_1+d_2}}(\mathcal{M}_f(X)))$ by (\ref{equ-product-inter}).
	By Lemma \ref{lemma-relative}(a), there are $\mu_0,\mu_+\in \mathcal{M}_f^e(X)\cap\mathcal{P}_{\Phi_{d_1+d_2}}^{-1}(a_1,a_2)\subset \mathcal{M}_f^e(X)\cap\mathcal{P}_{\Phi_{d_1}}^{-1}(a_1)$ such that $h_{\mu_0}(f)=0,h_{\mu_+}(f)>0.$
	
	By Lemma \ref{lemma-relative}(b), we have $\mathrm { relint } (\mathcal{T}_{\Phi_{d_1+d_2},f}(\mathcal{M}_f(X)))=\mathrm { relint } (\mathcal{T}_{\Phi_{d_1+d_2},f}(\mathcal{M}_f^e(X))).$ For any  $(a_2,h)\in \mathrm { relint } (\{\mathcal{T}_{\Phi_{d_2},f}(\mu):\mu\in \mathcal{P}_{\Phi_{d_1}}^{-1}(a_1)\}),$ by (\ref{equ-product-inter}) we have $$(a_1,a_2,h)\in \mathrm { relint } (\mathcal{T}_{\Phi_{d_1+d_2},f}(\mathcal{M}_f(X)))=\mathrm { relint } (\mathcal{T}_{\Phi_{d_1+d_2},f}(\mathcal{M}_f^e(X))).$$ Hence, $(a_2,h)\in \{\mathcal{T}_{\Phi_{d_2},f}(\mu):\mu\in \mathcal{M}_f^e(X)\cap  \mathcal{P}_{\Phi_{d_1}}^{-1}(a_1)\}.$ 
	So we have $\mathrm { relint } (\{\mathcal{T}_{\Phi_{d_2},f}(\mu):\mu\in \mathcal{P}_{\Phi_{d_1}}^{-1}(a_1)\})=\mathrm { relint } (\{\mathcal{T}_{\Phi_{d_2},f}(\mu):\mu\in \mathcal{M}_f^e(X)\cap  \mathcal{P}_{\Phi_{d_1}}^{-1}(a_1)\}).$
	
	Similar as Lemma \ref{lemma-relative}, we obtain item(3).
\end{proof}

\subsection{Proof of Theorem \ref{thm-inter-huasdorff}}
\begin{Lem}\label{lemma-locally-haus}
	Let $f: M \mapsto M$ be a $C^1$  diffeomorphism on a compact Riemannian manifold $M,$ and let $\Lambda\subset M$ be a transitive locally maximal hyperbolic set such that $\Lambda$ is average conformal.  Then $$[0,\sup\limits_{\mu\in \mathcal{M}_f^e(\Lambda)}\dim_H\mu)\subset \{\dim_H\mu:\mu\in \mathcal{M}_f^e(\Lambda)\}\subset  [0,\sup\limits_{\mu\in \mathcal{M}_f^e(\Lambda)}\dim_H\mu].$$ 
	If further $f$ is $C^{1+\alpha}$ and $f:\Lambda\to \Lambda$ is mixing, then $\{\operatorname{dim}_H \mu: \mu\in\mathcal{M}_f^e(\Lambda)\}= \left\{\operatorname{dim}_H \mu: \mu\in\mathcal{M}_f(\Lambda)\right\}.$
\end{Lem}
\begin{proof}
	By (\ref{equation-WC}) we have $\dim_H \mu=\frac{h_\mu(f)}{\chi_u(\mu)}-\frac{h_\mu(f)}{\chi_s(\mu)}$ for any $\mu\in\mathcal{M}_f^e(\Lambda).$ 
	Combining with (\ref{equation-Birk}), we have 
	$
	\dim_H \mu=\frac{d_uh_\mu(f)}{\int \psi^u d\mu}-\frac{d_sh_\mu(f)}{\int \psi^s d\mu}.
	$
	Since $\Lambda$ is hyperbolic, there is $\delta_0>0$ such that $\psi^u(x)>\delta_0,-\psi^s(x)>\delta_0$ for any $x\in \Lambda$.
	Define a map from $\mathbb{R}^3$ to $\mathbb{R}$ as following: $$\mathcal{Q}:(x,y,z)\to\frac{d_uz}{x}-\frac{d_sz}{y}.$$ 
	Then $\dim_H\mu=\mathcal{Q}(\mathcal{T}_{\psi^u,\psi^s,f}(\mu))$ for any $\mu\in \mathcal{M}_f^e(\Lambda),$ $\mathcal{Q}$ is continuous on $\mathcal{T}_{\psi^u,\psi^s,f}(\mathcal{M}_f(\Lambda)).$ 
	By Lemma \ref{lemma-relative}(c), $\mathcal{Q}(\mathcal{T}_{\psi^u,\psi^s,f}(\mathcal{M}_f^e(\Lambda)))$ is an interval.
	So $\{\dim_H\mu:\mu\in \mathcal{M}_f^e(\Lambda)\}$ is an interval. 
	By Lemma \ref{lemma-relative}(a) there are ergodic measures supported on $\Lambda$ with zero entropy. Then $0\in \{\dim_H\mu:\mu\in \mathcal{M}_f^e(\Lambda)\}.$
	Since $\{\dim_H\mu:\mu\in \mathcal{M}_f^e(\Lambda)\}$ is an interval and $$\sup\limits_{\mu\in \mathcal{M}_f^e(\Lambda)}\dim_H\mu\in \overline{\{\dim_H\mu:\mu\in \mathcal{M}_f^e(\Lambda)\}},$$ then we have $[0,\sup\limits_{\mu\in \mathcal{M}_f^e(\Lambda)}\dim_H\mu)\subset \{\dim_H\mu:\mu\in \mathcal{M}_f^e(\Lambda)\}\subset  [0,\sup\limits_{\mu\in \mathcal{M}_f^e(\Lambda)}\dim_H\mu].$
	
	If further $f$ is $C^{1+\alpha}$ and $f:\Lambda\to \Lambda$ is mixing, by (\ref{max-dim}), we have $\max\limits_{\mu\in \mathcal{M}_f^e(\Lambda)}\dim_H\mu=\max\limits_{\mu\in \mathcal{M}_f(\Lambda)}\dim_H\mu.$ This implies $\max\limits_{\mu\in \mathcal{M}_f(\Lambda)}\dim_H\mu\in \{\dim_H\mu:\mu\in \mathcal{M}_f^e(\Lambda)\}.$ So  we have $\{\dim_H\mu:\mu\in \mathcal{M}_f^e(\Lambda)\}=[0,\max\limits_{\mu\in \mathcal{M}_f(\Lambda)}\dim_H\mu]=\{\dim_H\mu:\mu\in \mathcal{M}_f(\Lambda)\}.$
\end{proof}

\begin{Lem}\label{lemma-locally-haus-2}
	Let $f: M \mapsto M$ be a $C^1$  diffeomorphism on a compact Riemannian manifold $M,$ and let $\Lambda\subset M$ be a transitive locally maximal hyperbolic set such that $\Lambda$ is average conformal. Let $d\in \mathbb{N}$ and $\Phi_{d}=\{\varphi_i\}_{i=1}^{d}\subset C(\Lambda)$.  Then for any $a\in \mathrm { relint } (\mathcal{P}_{\Phi_{d}}(\mathcal{M}_f(\Lambda))),$ we have $\sup\limits_{\mu\in \mathcal{M}_f^e(\Lambda)\cap   \mathcal{P}_{\Phi_{d}}^{-1}(a)}\dim_H\mu>0$ and $$[0,\sup\limits_{\mu\in \mathcal{M}_f^e(\Lambda)\cap   \mathcal{P}_{\Phi_{d}}^{-1}(a)}\dim_H\mu)\subset \{\dim_H\mu:\mu\in \mathcal{M}_f^e(\Lambda)\cap   \mathcal{P}_{\Phi_{d}}^{-1}(a)\}\subset  [0,\sup\limits_{\mu\in \mathcal{M}_f^e(\Lambda)\cap   \mathcal{P}_{\Phi_{d}}^{-1}(a)}\dim_H\mu].$$ 
\end{Lem}
\begin{proof}
	Following the argument of Lemma \ref{lemma-locally-haus},  
	$
	\dim_H \mu=\frac{d_uh_\mu(f)}{\int \psi^u d\mu}-\frac{d_sh_\mu(f)}{\int \psi^s d\mu},
	$
	$\mathcal{Q}:(x,y,z)\to\frac{d_uz}{x}-\frac{d_sz}{y}$ is continuous on $\mathcal{T}_{\psi^u,\psi^s,f}(\mathcal{M}_f(\Lambda)),$ and 
	$\dim_H\mu=\mathcal{Q}(\mathcal{T}_{\psi^u,\psi^s,f}(\mu))$ for any $\mu\in \mathcal{M}_f^e(\Lambda).$ 
	By Lemma \ref{lemma-relative-2}(c), $\mathcal{Q}(\{\mathcal{T}_{\psi^u,\psi^s,f}(\mu):\mu\in \mathcal{M}_f^e(\Lambda)\cap   \mathcal{P}_{\Phi_{d}}^{-1}(a)\})$  is an interval.
	So $\{\dim_H\mu:\mu\in \mathcal{M}_f^e(\Lambda)\cap   \mathcal{P}_{\Phi_{d}}^{-1}(a)\}$ is an interval.

	By Lemma \ref{lemma-relative-2}(a) there are $\mu_0,\mu_+\in \mathcal{M}_f^e(\Lambda)\cap\mathcal{P}_{\Phi_{d}}^{-1}(a)$ such that $h_{\mu_0}(f)=0,h_{\mu_+}(f)>0.$
	Then we have $0\in \{\dim_H\mu:\mu\in \mathcal{M}_f^e(\Lambda)\cap   \mathcal{P}_{\Phi_{d}}^{-1}(a)\}$ and $\sup\limits_{\mu\in \mathcal{M}_f^e(\Lambda)\cap   \mathcal{P}_{\Phi_{d}}^{-1}(a)}\dim_H\mu>0.$
	Since $\{\dim_H\mu:\mu\in \mathcal{M}_f^e(\Lambda)\cap   \mathcal{P}_{\Phi_{d}}^{-1}(a)\}$ is an interval and $$\sup\limits_{\mu\in \mathcal{M}_f^e(\Lambda)\cap   \mathcal{P}_{\Phi_{d}}^{-1}(a)}\dim_H\mu\in \overline{\{\dim_H\mu:\mu\in \mathcal{M}_f^e(\Lambda)\cap   \mathcal{P}_{\Phi_{d}}^{-1}(a)\}},$$ we have $[0,\sup\limits_{\mu\in \mathcal{M}_f^e(\Lambda)\cap   \mathcal{P}_{\Phi_{d}}^{-1}(a)}\dim_H\mu)\subset \{\dim_H\mu:\mu\in \mathcal{M}_f^e(\Lambda)\cap   \mathcal{P}_{\Phi_{d}}^{-1}(a)\}\subset  [0,\sup\limits_{\mu\in \mathcal{M}_f^e(\Lambda)\cap   \mathcal{P}_{\Phi_{d}}^{-1}(a)}\dim_H\mu].$
\end{proof}

\noindent {\textbf{Proof of Theorem \ref{thm-inter-huasdorff}:}} By Lemma \ref{lemma-locally-haus} and \ref{lemma-locally-haus-2}, we obtain item(1) and (2). Now we give the proof of item(3).
If $h_{\nu}(f)=0,$ then by (\ref{equation-WC}) we have $\dim_H \nu=0.$ Hence, $\dim_H \nu=0\in\{\dim_H\mu:\mu\in \mathcal{M}_f^e(M)\}$.
Now we assume that $h_{\nu}(f)>0.$
Let $\chi_s(\nu)<0<\chi_u(\nu)$ be the two Lyapunov exponents of $\nu$.  By Theorem \ref{Thm-Katok’s Horseshoe} for any $\varepsilon>0$ there exists a transitive locally maximal hyperbolic set $\Lambda_\varepsilon$ and $\mu_\varepsilon\in \mathcal{M}_f^e(\Lambda_\varepsilon)$ such that $|h_{\nu}(f)-h_{\mu_\varepsilon}(f)|<\varepsilon$, $|\chi_u(\nu)-\chi_u(\mu_\varepsilon)|<\varepsilon$ and $|\chi_s(\nu)-\chi_s(\mu_\varepsilon)|<\varepsilon$ where $
\chi_s(\mu_\varepsilon) < 0<  \chi_{u}(\mu_\varepsilon)$ are   the Laypunov exponents of $\mu_\varepsilon.$
By Lemma \ref{lemma-locally-haus},  $\{\dim_H\mu:\mu\in \mathcal{M}_f^e(M)\}\supset \{\dim_H\mu:\mu\in \mathcal{M}_f^e(\Lambda_\varepsilon)\} \supset [0,\dim_H\mu_\varepsilon].$ Since $\varepsilon$ is arbitrary and $\lim\limits_{\varepsilon\to 0}\dim_H\mu_\varepsilon=\lim\limits_{\varepsilon\to 0}(\frac{h_{\mu_\varepsilon}(f)}{\chi_u(\mu_\varepsilon)}-\frac{h_{\mu_\varepsilon}(f)}{\chi_s(\mu_\varepsilon)})=\frac{h_\nu(f)}{\chi_u(\nu)}-\frac{h_\nu(f)}{\chi_s(\nu)}=\dim_H\nu,$ we have $\{\dim_H\mu:\mu\in \mathcal{M}_f^e(M)\}\supset [0,\dim_H\nu].$\qed

\subsection{Proof of Theorem \ref{thm-Lyapunov}}\label{section-Lyapunov}
Now we consider an transitive multi-average conformal  Anosov diffeomorphism $f:M\to M$.
Denote $\psi^j(x)=\log|\det Df|_{E_x^{j}}|$ for any $x\in M$ and any $1\leq j\leq t_u+t_s.$ Similar as (\ref{equation-Birk0}), we have $\int\psi^j d\mu=d_j\chi_{\sum_{k=1}^{j-1}d_k+1}(\mu).$ Denote $\phi^i=\frac{\psi^j}{d_j}$ for any $\sum_{k=1}^{j-1}d_k+1\leq i\leq \sum_{k=1}^{j}d_k.$  Then we have  $\mathrm{Lya}(\mu)=(\int \phi^1 d\mu,\dots,\int \phi^{\dim M} d\mu )$ for any $\mu\in \mathcal{M}_f(M).$  By Lemma \ref{lemma-relative} (a) and (b) we obtain (1) and (2).
By Lemma \ref{lemma-locally-haus-2} we obtain (3).
\qed

\subsection{Proof of Theorem \ref{thm-first-return}}\label{section-proof-F}
\begin{Lem}\cite[Theorem 1]{STV2003} and \cite[Corollary 3.1]{OliveTian 2013}\label{Lem-first}
	Let $f: M \rightarrow M$ be a $C^1$ diffeomorphism on a  compact Riemannian manifold, and $\nu$ an $f$-invariant ergodic hyperbolic measure with $h_\nu(f)>0$. Denote  the Laypunov exponents of $\nu$ by $
	\chi_1(\nu) \geq \chi_2(\nu) \geq \cdots\geq \chi_{d_u}(\mu)>0>\chi_{d_u+1}(\nu) \geq\dots\geq \chi_{\dim M}(\nu).$  If $f$ is $C^{1+\alpha}$ for some $0<\alpha<1$ or the Oseledec splitting of $\nu$ is dominated,   then for  $\mu$ a.e. $x \in M$,
	$$
	\frac{1}{\chi_{1}(\mu)}-\frac{1}{\chi_{\dim M}(\mu)}\leq \limsup _{r \rightarrow 0} \frac{\tau(B(x, r))}{-\log r} \leq\liminf _{r \rightarrow 0} \frac{\tau(B(x, r))}{-\log r} \leq \frac{1}{\chi_{d_u}(\mu)}-\frac{1}{\chi_{d_u+1}(\mu)}.
	$$
\end{Lem}
\begin{Rem}
	Lemma \ref{Lem-first} was proved for $C^{1+\alpha}$ case in \cite{STV2003,OliveTian 2013}. Then dominated case can be proved using Shadowing lemma in $C^1$ setting. See, for example, \cite{SunTian2015}.
\end{Rem}

\begin{Lem}\label{lemma-locally-first}
	Let $f: M \mapsto M$ be a $C^1$  diffeomorphism on a compact Riemannian manifold $M,$ and let $\Lambda\subset M$ be a transitive locally maximal hyperbolic set such that $\Lambda$ is average conformal.   Then $\{r_f(\mu):\mu\in \mathcal{M}_f^e(\Lambda) \text{ and }h_\mu(f)>0\}$ is an interval.
\end{Lem}
\begin{proof}
	For any $\mu\in\mathcal{M}_f^e(\Lambda),$ let $\chi_s(\mu)<0<\chi_u(\mu)$ be the two Lyapunov exponents of $\nu$. Then by Lemma \ref{Lem-first}, $r_f(\mu)=\frac{1}{\lambda_u(\mu)}-\frac{1}{\lambda_s(\mu)}$ if $h_\mu(f)>0.$
	By (\ref{equation-Birk0}), we have 
	$
	r_\mu=\frac{d_u}{\int \psi^u d\mu}-\frac{d_s}{\int \psi^s d\mu}.
	$ 
	Since $\Lambda$ is hyperbolic, there is $\delta_0>0$ such that $\psi^u(x)>\delta_0,-\psi^s(x)>\delta_0$ for any $x\in \Lambda$.
	Define a map from $\mathbb{R}^3$ to $\mathbb{R}$ as following: $$\mathcal{Q}:(x,y,z)\to\frac{d_u}{x}-\frac{d_s}{y}.$$ 
	Then $r_f(\mu)=\mathcal{Q}(\mathcal{T}_{\psi^u,\psi^s,f}(\mu))$ for any $\mu\in \mathcal{M}_f^e(\Lambda)$ with $h_\mu(f)>0,$ $\mathcal{Q}$ is continuous on $\mathcal{T}_{\psi^u,\psi^s,f}(\mathcal{M}_f(\Lambda)).$ 
	By Lemma \ref{lemma-relative}(c), $\mathcal{Q}(\mathcal{T}_{\psi^u,\psi^s,f}(\{\mathcal{M}_f^e(X):h_\mu(f)>0\})$ is an interval.
	So $\{r_f(\mu):\mu\in \mathcal{M}_f^e(\Lambda) \text{ and }h_\mu(f)>0\}$ is an interval.
\end{proof}

By Lemma \ref{lemma-locally-first} we obtain Theorem \ref{thm-first-return}.



\section{Other applications}\label{section-Applications}
\subsection{Proof of Corollary \ref{thm-pressure}}\label{section-proof-D}
First, we give a lemma.
\begin{Lem}\label{lemma-locally-press}
	Let $f: M \mapsto M$ be a $C^1$  diffeomorphism on a compact Riemannian manifold $M,$ and let $\Lambda\subset M$ be a transitive locally maximal hyperbolic set.  Then $$(-\psi^u_{\sup},P^u_{sup}]\subseteq \{P^{u}(\mu):\mu\in \mathcal{M}_f^e(\Lambda)\}\subseteq [-\psi^u_{\sup},P^u_{sup}],$$ where $\psi^u_{\sup}=\sup\limits_{\mu\in \mathcal{M}_f(\Lambda)}\int \sum_{i=1}^{\operatorname{dim} M} \chi_i^{+} d\mu,$ $P^u_{sup}=\sup\limits_{\mu\in\mathcal{M}_f(\Lambda)}P^u(\mu).$ Moreover, given $d\in \mathbb{N}$ and $\Phi_{d}=\{\varphi_i\}_{i=1}^{d}\subset C(\Lambda)$, then for any $a\in \mathrm { relint } (\mathcal{P}_{\Phi_{d}}(\mathcal{M}_f(\Lambda))),$   $\{P^u(f):\mu\in \mathcal{M}_f^e(\Lambda)\cap   \mathcal{P}_{\Phi_{d}}^{-1}(a)\}$ is an interval and $\overline{\{P^u(f):\mu\in \mathcal{M}_f^e(\Lambda)\cap   \mathcal{P}_{\Phi_{d}}^{-1}(a)\}}=\overline{\{P^u(f):\mu\in \mathcal{M}_f(\Lambda)\cap   \mathcal{P}_{\Phi_{d}}^{-1}(a)\}}.$
\end{Lem}
\begin{proof}
	By (\ref{equation-Birk0}), $\int \psi^u d\mu=\int \sum_{i=1}^{\operatorname{dim} M} \chi_i^{+}(x) d \mu.$ 
	Then $P^u(\mu)=h_{\mu}(f)-\int \psi^u d\mu$ for any $\mu\in\mathcal{M}_f(M).$ 
	Define a map from $\mathbb{R}^2$ to $\mathbb{R}$ as  $\mathcal{Q}:(x,y)\to y-x.$
	Then $P^u(\mu)=\mathcal{Q}(\mathcal{T}_{\psi^u,f}(\mu))$ for any $\mu\in \mathcal{M}_f(\Lambda),$ $\mathcal{Q}$ is continuous on $\mathcal{T}_{\psi^u,f}(\mathcal{M}_f(\Lambda)).$ 
	By Lemma \ref{lemma-relative}(c), $\mathcal{Q}(\mathcal{T}_{\psi^u,f}(\mathcal{M}_f^e(\Lambda)))$ is an interval.
	So $\{P^u(\mu):\mu\in \mathcal{M}_f^e(\Lambda)\}$ is an interval. 
	By Lemma \ref{LH}, $f:\Lambda\to \Lambda$ is expansive, and thus $\mu\to h_{\mu}(f)$ is upper semi-continuous. 
	Then there is $\nu\in\mathcal{M}_f(M)$ such that $P^u(\nu)=P^u_{sup}.$ Combinging with the ergodic decomposition theorem, there is $\nu^u\in\mathcal{M}_f^e(M)$ such that 
$
		P^u(\nu^u)=P^u_{sup}.
$
    Hence, $P^u_{\sup}\in \{P^u(\mu):\mu\in \mathcal{M}_f^e(\Lambda)\}.$ 
    By Lemma \ref{lemma-relative}(a),
    $\mathcal{T}_{\psi^u,f}(\mathcal{M}_f^e(\Lambda))\supset \mathrm{relint}(\mathcal{P}_{\psi^u}(\mathcal{M}_f(\Lambda)))\times \{0\}.$
    Then $-\psi^u_{\sup}\in \mathcal{Q}(\overline{\mathcal{T}_{\psi^u,f}(\mathcal{M}_f^e(\Lambda))})\subset \overline{\{P^u(\mu):\mu\in \mathcal{M}_f^e(\Lambda)\}}.$ So we have $(-\psi^u_{\sup},P^u_{\sup}]\subseteq \{P^{u}(\mu):\mu\in \mathcal{M}_f^e(M)\}\subseteq [-\psi^u_{\sup},P^u_{\sup}].$
    
    By Lemma \ref{lemma-relative-2}(c),  $\{P^u(f):\mu\in \mathcal{M}_f^e(\Lambda)\cap   \mathcal{P}_{\Phi_{d}}^{-1}(a)\}$ is an interval and $\overline{\{P^u(f):\mu\in \mathcal{M}_f^e(\Lambda)\cap   \mathcal{P}_{\Phi_{d}}^{-1}(a)\}}=\overline{\{P^u(f):\mu\in \mathcal{M}_f(\Lambda)\cap   \mathcal{P}_{\Phi_{d}}^{-1}(a)\}}.$
\end{proof}

\noindent {\textbf{Proof of Corollary \ref{thm-pressure}} (1) Let $f: M \mapsto M$ be a  transitive Anosov diffeomorphism. 
By \cite{CatsEnri2011}, there is $\nu\in\mathcal{M}_f(M)$ such that $\nu$ is an SRB-like measure for $f.$ Then by \cite[Corollary 2]{CatsCerEnri2011},  $h_{\nu}(f)=\int \sum_{i=1}^{\operatorname{dim} M} \chi_i^{+}(x) d \nu.$ Combinging with Ruelle’s inequality, we have $P^u_{sup}=0.$ Then by Lemma \ref{lemma-locally-press} we obtain (1).

(2) If $h_{\nu}(f)=0,$ then $P^u(\nu)=-\int \sum_{i=1}^{\operatorname{dim} M} \chi_i^{+} d\nu.$ It's obvious that  $\{P^{u}(\mu):\mu\in \mathcal{M}_f^e(M)\}\ni P^u(\nu).$ Now we assume that $h_{\nu}(f)>0.$ Denote the Laypunov exponents of $\nu$ by $
\chi_1(\nu) \geq \chi_2(\nu) \geq \cdots\geq \chi_{d_u}(\mu)>0>\chi_{d_u+1}(\nu) \geq\dots\geq \chi_{\dim M}(\nu).$ By Theorem \ref{Thm-Katok’s Horseshoe}, for any $\varepsilon>0$ there exists a transitive locally maximal hyperbolic set $\Lambda_\varepsilon$ and $\mu_\varepsilon\in \mathcal{M}_f^e(\Lambda_\varepsilon)$ such that $|h_{\nu}(f)-h_{\mu_\varepsilon}(f)|<\varepsilon$ and $|\sum_{i=1}^{d_u}\chi_i(\nu)-\sum_{i=1}^{d_u}\chi_i(\mu_\varepsilon)|<\varepsilon$ where $
\chi_1(\mu_\varepsilon) \geq \chi_2(\mu_\varepsilon) \geq \cdots\geq \chi_{d_u}(\mu_\varepsilon)>0>\chi_{d_u+1}(\mu_\varepsilon) \geq\dots\geq \chi_{\dim M}(\mu_\varepsilon)$ are   the Laypunov exponents of $\mu_\varepsilon.$
By Lemma \ref{lemma-locally-press},  we have $\{P^u(\mu):\mu\in \mathcal{M}_f^e(M)\}\supset \{P^u(\mu):\mu\in \mathcal{M}_f^e(\Lambda_{\varepsilon})\}\supset (-\sum_{i=1}^{d_u}\chi_i(\mu_\varepsilon),P^u(\mu_\varepsilon)].$ Since $\varepsilon$ is arbitrary, $\lim\limits_{\varepsilon\to 0}\sum_{i=1}^{d_u}\chi_i(\mu_\varepsilon)=\sum_{i=1}^{d_u}\chi_i(\nu),$ $\lim\limits_{\varepsilon\to 0}P^u(\mu_\varepsilon)=\lim\limits_{\varepsilon\to 0}(h_{\mu_\varepsilon}(f)-\sum_{i=1}^{d_u}\chi_i(\mu_\varepsilon))=h_{\nu}(f)-\sum_{i=1}^{d_u}\chi_i(\nu)=P^u(\nu),$ we have 
$\{P^{u}(\mu):\mu\in \mathcal{M}_f^e(M)\}\supset (-\int \sum_{i=1}^{\operatorname{dim} M} \chi_i^{+} d\nu,P^{u}(\nu)].$\qed

\subsection{Proof of Corollary  \ref{thm-inter-huasdorff-3}}\label{section-proof-E}
First, we give a lemma.
\begin{Lem}\label{lemma-locally-unstable}
		Let $f: M \mapsto M$ be a $C^1$  diffeomorphism on a compact Riemannian manifold $M,$ and let $\Lambda\subset M$ be a transitive locally maximal hyperbolic set.  Then $\{\dim^{u}_H\mu:\mu\in \mathcal{M}_f^e(\Lambda)\}\supset [0,\sup_{\mu\in\mathcal{M}_f^e(\Lambda)}\dim_H^u\mu).$ Moreover, given $d\in \mathbb{N}$ and $\Phi_{d}=\{\varphi_i\}_{i=1}^{d}\subset C(\Lambda)$, then for any $a\in \mathrm { relint } (\mathcal{P}_{\Phi_{d}}(\mathcal{M}_f(\Lambda))),$ we have $\dim_H^u(a)=\sup\limits_{\mu\in \mathcal{M}_f^e(\Lambda)\cap   \mathcal{P}_{\Phi_{d}}^{-1}(a)}\dim_H^u\mu>0$  and $$[0,\dim_H^u(a))\subset \{\dim_H^u\mu:\mu\in \mathcal{M}_f^e(\Lambda)\cap   \mathcal{P}_{\Phi_{d}}^{-1}(a)\}\subset  [0,\dim_H^u(a)],$$ $$\{\dim_H^u\mu:\mu\in \mathcal{M}_f(\Lambda)\cap   \mathcal{P}_{\Phi_{d}}^{-1}(a)\}= [0,\dim_H^u(a)].$$
\end{Lem}
\begin{proof}
	By (\ref{equation-Birk0}), $\int \psi^u d\mu=\int \sum_{i=1}^{\operatorname{dim} M} \chi_i^{+}(x) d \mu.$ 
	Then $\dim^u_H\mu=\frac{h_\mu(f)}{\int \psi^u d\mu}$ for any $\mu\in\mathcal{M}_f(\Lambda).$ 
	Define a map from $\mathbb{R}^2$ to $\mathbb{R}$ as  $\mathcal{Q}:(x,y)\to \frac{y}{x}.$
	Then $\dim^u_H\mu=\mathcal{Q}(\mathcal{T}_{\psi^u,f}(\mu))$ for any $\mu\in \mathcal{M}_f(\Lambda),$ $\mathcal{Q}$ is continuous on $\mathcal{T}_{\psi^u,f}(\mathcal{M}_f(\Lambda)).$ 
	By Lemma \ref{lemma-relative}(c), $\mathcal{Q}(\mathcal{T}_{\psi^u,f}(\mathcal{M}_f^e(\Lambda)))$ is an interval.
	So $\{\dim^u_H\mu:\mu\in \mathcal{M}_f^e(\Lambda)\}$ is an interval. 
	By Lemma \ref{lemma-relative}(a) there are ergodic measures supported on $\Lambda$ with zero entropy. Then $0\in \{\dim_H^u\mu:\mu\in \mathcal{M}_f^e(\Lambda)\}.$
	Since $\sup_{\mu\in\mathcal{M}_f^e(\Lambda)}\dim_H^u\mu\in \overline{\{\dim_H^u\mu:\mu\in \mathcal{M}_f^e(M)\}},$ then $\{\dim^{u}_H\mu:\mu\in \mathcal{M}_f^e(\Lambda)\}\supset [0,\sup_{\mu\in\mathcal{M}_f^e(\Lambda)}\dim_H^u\mu).$
	
	By Lemma \ref{lemma-relative-2}(c), $\mathcal{Q}(\{\mathcal{T}_{\psi^uf}(\mu):\mu\in \mathcal{M}_f^e(\Lambda)\cap   \mathcal{P}_{\Phi_{d}}^{-1}(a)\})$  is an interval.
	So $\{\dim_H^u\mu:\mu\in \mathcal{M}_f^e(\Lambda)\cap   \mathcal{P}_{\Phi_{d}}^{-1}(a)\}$ is an interval. 
	By Lemma \ref{lemma-relative-2}(a) there are $\mu_0,\mu_+\in \mathcal{M}_f^e(\Lambda)\cap\mathcal{P}_{\Phi_{d}}^{-1}(a)$ such that $h_{\mu_0}(f)=0,h_{\mu_+}(f)>0.$
	Then we have $0\in \{\dim_H^u\mu:\mu\in \mathcal{M}_f^e(\Lambda)\cap   \mathcal{P}_{\Phi_{d}}^{-1}(a)\}$ and $\dim_H^u(a)=\sup\limits_{\mu\in \mathcal{M}_f^e(\Lambda)\cap   \mathcal{P}_{\Phi_{d}}^{-1}(a)}\dim_H^u\mu>0.$
	Since $\{\dim_H^u\mu:\mu\in \mathcal{M}_f^e(\Lambda)\cap   \mathcal{P}_{\Phi_{d}}^{-1}(a)\}$ is an interval and $\dim_H^u(a)\in \overline{\{\dim_H^u\mu:\mu\in \mathcal{M}_f^e(\Lambda)\cap   \mathcal{P}_{\Phi_{d}}^{-1}(a)\}},$ we have $$[0,\dim_H^u(a))\subset \{\dim_H^u\mu:\mu\in \mathcal{M}_f^e(\Lambda)\cap   \mathcal{P}_{\Phi_{d}}^{-1}(a)\}\subset  [0,\dim_H^u(a)].$$ 
	By Lemma \ref{lemma-relative-2}(c), $\overline{\{\dim_H^u\mu:\mu\in \mathcal{M}_f(\Lambda)\cap   \mathcal{P}_{\Phi_{d}}^{-1}(a)\}}=
	\overline{\{\dim_H^u\mu:\mu\in \mathcal{M}_f^e(\Lambda)\cap   \mathcal{P}_{\Phi_{d}}^{-1}(a)\}.}$ By Lemma \ref{LH}, $f:\Lambda\to \Lambda$ is expansive, and thus $\mu\to h_{\mu}(f)$ is upper semi-continuous. So $\mu\to \dim_H^u\mu$ is upper semi-continuous, and thus $\{\dim_H^u\mu:\mu\in \mathcal{M}_f(\Lambda)\cap   \mathcal{P}_{\Phi_{d}}^{-1}(a)\}= [0,\dim_H^u(a)].$
\end{proof}

\noindent {\textbf{Proof of Corollary  \ref{thm-inter-huasdorff-3}:}} (1) Let $f: M \mapsto M$ be a $C^1$ transitive Anosov diffeomorphism.
By \cite{CatsEnri2011} there is $\nu\in\mathcal{M}_f(M)$ such that $\nu$ is an SRB-like measure for $f.$ Then by \cite[Corollary 2]{CatsCerEnri2011}, one has $h_{\nu}(f)=\int \sum_{i=1}^{\operatorname{dim} M} \chi_i^{+}(x) d \nu=\int \psi^u d\nu.$ Combinging with Ruelle’s inequality and the ergodic decomposition theorem, there is $\nu^u\in\mathcal{M}_f^e(M)$ such that 
$
	h_{\nu^u}(f)=\int \psi^u d\nu^u.
$
Then $\dim_H^u\mu^u=1.$ By Lemma \ref{lemma-locally-unstable},  $\{\dim^{u}_H\mu:\mu\in \mathcal{M}_f^e(\Lambda)\}\supset [0,1].$ By Ruelle’s inequality, $\{\dim^{u}_H\mu:\mu\in \mathcal{M}_f(\Lambda)\}\subset [0,1].$  So $\{\dim^{u}_H\mu:\mu\in \mathcal{M}_f^e(\Lambda)\}=\{\dim^{u}_H\mu:\mu\in \mathcal{M}_f(\Lambda)\}= [0,1].$ By Lemma \ref{lemma-locally-unstable} we obtain item(1).

(2) If $h_{\nu}(f)=0,$ then $\dim_H^u\nu=0.$ It's obvious that $\{\dim^{u}_H\mu:\mu\in \mathcal{M}_f^e(M)\}\ni 0.$  Now we assume that $h_{\nu}(f)>0.$ Denote the Laypunov exponents of $\nu$ by $
\chi_1(\nu) \geq \chi_2(\nu) \geq \cdots\geq \chi_{d_u}(\mu)>0>\chi_{d_u+1}(\nu) \geq\dots\geq \chi_{\dim M}(\nu).$ By Theorem \ref{Thm-Katok’s Horseshoe}, for any $\varepsilon>0$ there exists a transitive locally maximal hyperbolic set $\Lambda_\varepsilon$ and $\mu_\varepsilon\in \mathcal{M}_f^e(\Lambda_\varepsilon)$ such that $|h_{\nu}(f)-h_{\mu_\varepsilon}(f)|<\varepsilon$ and $|\sum_{i=1}^{d_u}\chi_i(\nu)-\sum_{i=1}^{d_u}\chi_i(\mu_\varepsilon)|<\varepsilon$ where $
\chi_1(\mu_\varepsilon) \geq \chi_2(\mu_\varepsilon) \geq \cdots\geq \chi_{d_u}(\mu_\varepsilon)>0>\chi_{d_u+1}(\mu_\varepsilon) \geq\dots\geq \chi_{\dim M}(\mu_\varepsilon)$ are   the Laypunov exponents of $\mu_\varepsilon.$
By Lemma \ref{lemma-locally-unstable},  $\{\dim^{u}_H\mu:\mu\in \mathcal{M}_f^e(M)\}\supset \{\dim^{u}_H\mu:\mu\in \mathcal{M}_f^e(\Lambda)\}\supset [0,\dim_H^u\mu_\varepsilon].$ Since  $\lim\limits_{\varepsilon\to 0}\dim^u_H\mu_\varepsilon=\lim\limits_{\varepsilon\to 0}\frac{h_{\mu_\varepsilon}(f)}{\sum_{i=1}^{d_u}\chi_i(\mu_\varepsilon)}=\frac{h_{\nu}(f)}{\sum_{i=1}^{d_u}\chi_i(\nu)}=\dim_H^u\nu,$ we have $\{\dim^{u}_H\mu:\mu\in \mathcal{M}_f^e(M)\}\supset [0,\dim_H^u\nu].$\qed

\section{Question}
The 'multi-horseshoe' entropy-dense property plays a key role in the proof of our results. For dynamical systems with specification-like properties, we don't know how to obtain the 'multi-horseshoe' entropy-dense property. So we ask the following question:
\begin{Que}
	How can we obtain the 'multi-horseshoe' entropy-dense property  for  dynamical systems with specification-like properties, or how can we  answer Question \ref{Conjecture-2} without using the 'multi-horseshoe' entropy-dense property?
\end{Que}

\bigskip

{\bf Acknowledgements.  } We thank Prof. Jinhua Zhang for many helpful discussions, and thank Wanshan Lin for some comments. The authors are  
supported by National Natural Science Foundation of China (grant No.   12071082) and in part by Shanghai Science and Technology Research Program (grant No. 21JC1400700).

\end{document}